\documentclass[11pt,oneside]{amsart}

\usepackage{microtype}

\usepackage{amsmath,ifthen, amsfonts, amssymb,
srcltx, amsopn, color, enumerate} 
\usepackage[linktocpage=true]{hyperref}
\usepackage[cmtip,arrow]{xy}
\usepackage{pb-diagram, pb-xy}
\usepackage{overpic}

 \def\notorth{\:\ensuremath{\reflectbox{\rotatebox[origin=c]{90}{$\nvdash$}}}}

\newcommand{\showcomments}{yes}

\newsavebox{\commentbox}

\newcounter{ax}
\newtheorem{thm}{Theorem}[section]
\newtheorem{lem}[thm]{Lemma}
\newtheorem{lemma}[thm]{Lemma}

\newtheorem{cor}[thm]{Corollary}

\newtheorem{prop}[thm]{Proposition}

\newtheorem{thmi}{Theorem}
\newtheorem{cori}[thmi]{Corollary}
\newtheorem{probi}[thmi]{Problem}

\newtheorem{assumption}[ax]{Assumption}
\newtheorem{claim}[thm]{Claim}

\newtheorem{axiom}[thm]{Axiom}

\theoremstyle{definition}
\newtheorem{defn}[thm]{Definition}
\newtheorem{rem}[thm]{Remark}
\newtheorem*{remi}{Remark}
\newtheorem{exmp}[thm]{Example}

\newtheorem{notation}[thm]{Notation}

\newtheorem{claim*}{Claim}

\DeclareMathOperator{\dimension}{dim}
\DeclareMathOperator{\image}{im}

\DeclareMathOperator{\Aut}{Aut}

\DeclareMathOperator{\diam}{\textup{\textsf{diam}}}

\DeclareMathOperator{\hull}{hull}

\newcommand{\neb}{\mathcal N}

\newcommand{\field}[1]{\mathbb{#1}}
\newcommand{\integers}{\ensuremath{\field{Z}}}
\newcommand{\rationals}{\ensuremath{\field{Q}}}
\newcommand{\naturals}{\ensuremath{\field{N}}}
\newcommand{\reals}{\ensuremath{\field{R}}}

 \newcommand{\FU}[1]{F_{#1}}
 \newcommand{\EU}[1]{E_{#1}}

\newcommand{\boundary}{{\ensuremath \partial}}

\newcommand{\interior} [1] {{\ensuremath \text{\rm Int}(#1) }}

\makeatletter

\newcommand{\Rmnum}[1]{\mathbf{{\expandafter\@slowromancap\romannumeral #1@}}}

\newcommand{\contact}[1]{\ensuremath{\mathcal C#1}}

\makeatother

\newcommand{\tup}[1]{\vec{#1}}

\let\oldmarginpar\marginpar
\renewcommand\marginpar[1]{\-\oldmarginpar[\raggedleft\footnotesize #1]{\raggedright\footnotesize #1}}

\newcommand{\tsh}[1]{\left\{\kern-.7ex\left\{#1\right\}\kern-.7ex\right\}}
\newcommand{\Tsh}[2]{\tsh{#2}_{#1}}
\newcommand{\ignore}[2]{\Tsh{#2}{#1}}

\newcommand{\co}{\colon}

\newcounter{enumitemp}

\newcommand{\dist}{\textup{\textsf{d}}}

\newcommand{\OL}{\overleftarrow}
\newcommand{\OR}{\overrightarrow}

\newcommand{\cuco}[1]{{\mathcal #1}}

\newcommand{\fontact}{{\mathcal C}}

\newcommand{\calF}{\mathcal F}
\newcommand{\gate}{\mathfrak g}

\newcommand{\seq}[1]{\mbox{\boldmath$#1$}}

\newcommand{\subseq}[1]{\mbox{\boldmath$\scriptstyle#1$}}

\newcommand{\nest}{\sqsubseteq}
 \usepackage{mathabx}
\newcommand{\propnest}{\sqsubsetneq}
\newcommand{\orth}{\bot}
\newcommand{\transverse}{\pitchfork}

\newcommand{\relevant}{\mathbf{Rel}}

\newcommand{\hinges}{\mathbf{Hinge}}
\begin{document}

\title{quasiflats in hierarchically hyperbolic spaces}
\author[J. Behrstock]{Jason Behrstock}
\address{Lehman College and The Graduate Center, CUNY, New York, New York, USA}
\email{jason.behrstock@lehman.cuny.edu}
\author[M.F. Hagen]{Mark F Hagen}
\address{Dept. of Pure Maths and Math. Stat., University of Cambridge, Cambridge, UK}
\curraddr{School of Mathematics, University of Bristol, Bristol, UK}
\email{markfhagen@gmail.com}
\author[A. Sisto]{Alessandro Sisto}
	\address{Department of Mathematics, Heriot-Watt University, Edinburgh, UK}
	\email{a.sisto@hw.ac.uk}
\date{\today}


\maketitle

\begin{abstract}
The rank of a hierarchically hyperbolic space is the maximal number of
unbounded factors in a standard product region.  For hierarchically hyperbolic groups, this coincides with the
maximal dimension of a quasiflat. Several noteworthy examples for which the rank coincides with familiar
quantities include: the dimension of maximal Dehn twist flats for
mapping class groups, the maximal rank of a free abelian subgroup for
right-angled Coxeter groups and right-angled Artin groups (in the
latter this can also be observed as the clique number of the defining
graph), and, for the Weil--Petersson metric, the rank is 
the integer part of half the complex dimension of
Teichm\"{u}ller space.

 We prove that, in a hierarchically hyperbolic space, any quasiflat of
 dimension equal to the rank lies within finite distance of a union of
 standard orthants (under a very mild condition on the HHS satisfied
 by all natural examples).  This resolves outstanding conjectures when
 applied to a number of different groups and spaces.  
 
 In the case of the mapping
 class group, we verify a conjecture of Farb; for Teichm\"{u}ller
 space we answer a question of Brock; in the context of certain CAT(0)
 cubical groups, our result handles novel special cases, including right-angled Coxeter groups.  
 
 An important ingredient in
 the proof, which we expect will have other applications, is that the \emph{hull} of any finite set in an HHS is 
quasi-isometric  to a CAT(0) cube complex of dimension bounded by the rank (if the HHS is a
 CAT(0) cube complex, the rank can be lower than the dimension of the space).
 
 We deduce a number of applications of these results.  For instance, we 
 show that any quasi-isometry between HHSs induces a quasi-isometry between certain
 \emph{factored spaces}, which are simpler HHSs. This allows one, for
 example, to distinguish quasi-isometry classes of right-angled Artin/Coxeter groups.
 
 Another application of our results is to quasi-isometric rigidity.
 Our tools in many cases allow one to reduce the problem of
 quasi-isometric rigidity for a given hierarchically hyperbolic group
 to a combinatorial problem.  As a template, we give a new proof of
 quasi-isometric rigidity of mapping class groups, which, once we've 
 established our general quasiflats theorem, uses simpler combinatorial arguments than in previous proofs.
\end{abstract}

\tableofcontents

\section*{Introduction}

A classical result of Morse shows that, in a hyperbolic space,
quasigeodesics lie close to geodesics \cite{Morse:fundamentalclass}.  This raises the question of what constraints exist on 
the geometry of quasiflats in more general coarsely non-positively curved spaces.  

A key step in proving Mostow Rigidity is proving that an equivariant
quasi-isometry of a symmetric space sends each flat to within a
bounded neighborhood of a flat \cite{Mostow:rigidity}.  Unlike a quasigeodesic in a hyperbolic space, a quasiflat
need not lie close to any one flat.  However, generalizing Mostow's result, 
Eskin--Farb and Kleiner--Leeb each proved that, in
a higher-rank symmetric space, an arbitrary quasiflat must lie close to
a finite number of flats \cite{EskinFarb,KleinerLeeb:buildings}.  This
result can be used to prove quasi-isometric rigidity for uniform
lattices in higher-rank symmetric spaces \cite{KleinerLeeb:buildings};
see also \cite{EskinFarb}.

In this paper, we explain the structure of quasiflats in a broad class
of spaces and groups with a property called \emph{hierarchical
hyperbolicity}~\cite{hhs1, hhs2,hhs3}.  Hierarchical hyperbolicity
captures the coarse nonpositive curvature visible in many important
groups and spaces, including mapping class groups, right-angled Artin
groups, many CAT(0) cube complexes, most
$3$--manifold groups, Teichm\"{u}ller space (in any of the standard
metrics), etc.  

Hierarchical hyperbolicity generalizes, and was inspired by, theorems about the mapping
class group established by Masur--Minsky \cite{MasurMinsky:complex2}, Behrstock
\cite{Behrstock:asymptotic}, and others.  Motivation also comes from Kim--Koberda's work 
towards obtaining an analogue of some of those mapping class group results in the setting of right angled Artin 
groups \cite{KimKoberda:curve_graph}. To approach other problems, some  
features of the mapping class group were axiomatized by 
Bestvina--Bromberg--Fujiwara to great effect \cite{BBF:quasi_tree, BBF:actions}.

The class of hierarchically hyperbolic 
spaces is preserved by quasi-isometries, and also includes many examples not on the preceding list: one can readily produce 
new hierarchically hyperbolic spaces from old.  In particular, trees of hierarchically hyperbolic spaces 
satisfying natural constraints (and thus many graphs of hierarchically 
hyperbolic groups) are again hierarchically hyperbolic~\cite{hhs2,BerlaiRobbio}.  Groups 
that are hyperbolic relative to hierarchically hyperbolic groups are again 
hierarchically hyperbolic~\cite{hhs2}.  It is shown in~\cite{hhs3} that 
suitable small-cancellation quotients of hierarchically hyperbolic groups are 
again hierarchically hyperbolic.  

This article establishes a relationship between some of these examples: in particular, we 
show that these spaces all admit a very strong local approximation by CAT(0) cube complexes 
(Theorem~\ref{thm:cubulated_hulls}). This allows us to use 
cubical techniques in new settings.  For example, it enables application of cubical geometry to mapping class groups. 

Even for CAT(0) cube complexes our 
approximation provides new information.  The reason is that Theorem~\ref{thm:cubulated_hulls} allows one to approximate 
convex hulls of finite sets in an HHS by finite CAT(0) cube complexes, and if the ambient HHS is a CAT(0) cube complex, the dimension of 
the approximating complex --- which is bounded by the rank --- can be much lower than the dimension of the ambient complex.  
This is essential for our applications to quasiflats.

Our techniques are intrinsic to the category of hierarchically hyperbolic 
spaces, in the sense that the arguments in this paper couldn't be 
carried out strictly in the context of any of the motivating examples 
alone, for example CAT(0) cube complexes or mapping class groups. 

Formal definitions and relevant properties of hierarchically
hyperbolic spaces (HHSs) will be given below in
Section~\ref{sec:background}. For now, we recall that a \emph{hierarchically hyperbolic space} 
consists of: a quasigeodesic metric space, $\cuco X$; an \emph{index set}, $\mathfrak
S$; a hyperbolic
space~$\fontact U$ for each $U\in\mathfrak S$; some relations between elements of 
the index set and maps between the associated hyperbolic spaces.  There are also projections $\cuco X\to\fontact 
U,U\in\mathfrak S$, and various axioms governing all of this data.  

Before stating the main theorem, we informally recall a few geometric features of HHSs:
\begin{itemize}
     \item Any HHS $\cuco X$ contains certain
\emph{standard product regions}, in which each of the (boundedly many)
factors is itself an HHS. 

In mapping class groups, these are products
of mapping class groups of pairwise disjoint subsurfaces (for example 
the subgroup generated by Dehn twists along disjoint annuli).  In CAT(0) cube
complexes, these are certain convex subcomplexes that split as
products (for example in the Salvetti complex of a right-angled Artin 
group, they are subcomplexes associated to subgraphs of the defining 
graph that decompose as joins). 

\item Pairs of points in $\cuco X$ can be joined by particularly
well-behaved quasigeodesics called \emph{hierarchy paths}, and
similarly we have well-behaved quasigeodesic rays called
\emph{hierarchy rays}.  Given a standard product region $P$, and a
hierarchy ray in each of the $k$ factors of $P$, the product of the
$k$ hierarchy rays $[0,\infty)\to\cuco X$ is a quasi-isometric
embedding $[0,\infty)^k\to\cuco X$ which we call a \emph{standard
orthant}.

\item The \emph{rank} $\nu$ of an HHS is the largest possible number of factors in a standard product region, each of whose 
factors is unbounded.  (Equivalently, it is the maximal integer so that there exist pairwise \emph{orthogonal}
$U_1,\dots,U_\nu\in\mathfrak S$ for which each $\fontact U_{i}$ is unbounded.)

\end{itemize}

We will impose a mild technical assumption on our spaces, which we call being
\emph{asymphoric}; this condition is satisfied by the motivating
examples of HHSs, including all hierarchically hyperbolic groups.  Under this condition, Theorem~\ref{thm:rank_inv} implies that the rank is a quasi-isometry invariant.

\subsection*{Quasiflats}
Understanding the structure of quasiflats in a given 
metric space or group is often critical in understanding the geometry 
of that space. 

An early version of a ``quasiflats theorem'' is Mostow's 
result that in a rank-one symmetric space, any quasi-geodesic lies within a 
uniformly bounded distance of a geodesic \cite{Mostow:rigidity}. 

A well known 
generalization of this is due to 
Schwartz, who proved that the image of any quasi-isometric embedding of 
$\reals^{n},n\geq 2$ into a non-uniform lattice in a rank-one 
symmetric space lies within a uniformly bounded distance of a horosphere 
\cite{Schwartz:RankOne}.  (Actually, Schwartz proved a more general result, namely that the image of any 
quasi-isometric embedding from a space whose asymptotic cone doesn't 
have cut-points into a ``neutered space'' lies uniformly close to a 
horosphere; he credits unpublished work of Gersten for the case of Euclidean space.) 

Schwartz's result was 
generalized by Dru\c{t}u-Sapir, who replaced the target space by an arbitrary relatively hyperbolic space and showed  that 
the image of the quasi-isometric embedding lies uniformly close 
to a peripheral subspace \cite{DrutuSapir:TreeGraded}.  

This result was in turn generalized by 
Behrstock--Dru\c{t}u--Mosher, who weakened the hypothesis on the domain to allow any space 
which is itself not relatively hyperbolic \cite{BehrstockDrutuMosher:thick}. 

As noted above, 
Eskin--Farb and Kleiner--Leeb proved that, in
a higher-rank symmetric space, an arbitrary quasiflat must lie within finite distance of
a finite number of flats \cite{EskinFarb,KleinerLeeb:buildings}. 

As discussed further below, there has been much work toward
quasiflats theorems in other contexts. We now state our main 
result in this direction, explain some consequences, and describe    
interactions with related work. At the end of the introduction we sketch the proof.

\begin{thmi}[Quasiflats Theorem for HHSs]\label{thmi:main}
 Let $\cuco X$ be an HHG of rank $\nu$.  Let $f\co\reals^\nu\to\cuco
 X$ be a quasi-isometric embedding.  Then there exist standard
 orthants $Q_i\subseteq \cuco X$, $i=1,\dots,k$, so that
 $\dist_{haus}(f(\reals^\nu),\cup_{i=1}^kQ_i)<\infty$.  More
 generally, the same result holds for any space $\cuco X$ which is an
 asymphoric HHS of rank $\nu$.
\end{thmi}

We now give a few immediate applications of this theorem and discuss 
related results.

Mapping class groups have been studied for the past two decades using
tools introduced by Masur--Minsky \cite{MasurMinsky:complex2}; these
tools were then further developed in \cite{Behrstock:asymptotic, BKMM}
and elsewhere.  The results of \cite{Behrstock:asymptotic, BKMM,
MasurMinsky:complex2} together show that mapping class groups are
hierarchically hyperbolic; see \cite[Theorem~11.1]{hhs2} for details.
Theorem~\ref{thmi:main}, applied to mapping class groups, resolves a conjecture of
Farb.  Outside of the hyperbolic cases, this question was completely
open.
    
\begin{cori}[Farb Conjecture: Quasiflats theorem for mapping class 
    groups] Any
    top-dimensional quasiflat in the mapping class group lies within a uniformly bounded 
	distance from a finite union of standard flats.
\end{cori}

Although the structure of quasiflats in the mapping class group was 
unsettled, numerous prior results were obtained in pursuit 
of the resolution of this conjecture. One partial 
result in this direction was Behrstock--Minsky's theorem that $\reals^{n}$ can 
only be quasi-isometrically embedded in a given mapping class group 
if $n$ is at most the complexity of the surface 
\cite{BehrstockMinsky:rankconj}.  This established the dimension of 
the top-dimensional quasiflats in the mapping class group. 

Also 
significant are a number of results which give some local control of 
top-dimensional quasiflats in the mapping class group. In particular, see results of 
Behrstock--Kleiner--Minsky--Mosher \cite{BKMM}, Bowditch 
\cite{Bowditch:rigid}, and Eskin--Masur--Rafi \cite{EskinMasurRafi:large_scale_rank}. Although those 
prior results yield some control over quasiflats, Theorem~\ref{thmi:main} is 
the first to completely describe the structure of quasiflats in 
the mapping class group. As we will describe in more detail later, we 
use some of the tools developed by Bowditch in \cite{Bowditch:rigid} in our 
proof of Theorem~\ref{thmi:main}.

Outside of the setting of groups, we apply Theorem~\ref{thmi:main} to the 
Weil-Petersson metric on Teichm\"{u}ller space, which is an asymphoric HHS by virtue of Brock's theorem that the pants 
graph is quasi-isometric to the Weil-Petersson metric \cite{Brock:wp} 
and results of \cite{Behrstock:asymptotic, BKMM,
MasurMinsky:complex2}; for details see \cite[Theorem~G]{hhs1}. 

Brock asked whether every top-dimensional quasiflat in the
Weil-Petersson metric is a bounded distance
from a finite number of top-dimensional flats
\cite[Question~5.3]{Brock:pantsWP}. From Theorem~\ref{thmi:main} we obtain the
following, answering his question in the affirmative.
    
\begin{cori}[Affirmation of Brock's Question: Quasiflats theorem for 
    Weil-Petersson metric] Any top-dimensional quasiflat in the 
    Weil-Petersson metric on Teichm\"{u}ller space lies within a uniformly 
	bounded distance from a finite union of standard flats.
\end{cori}

The previously answered cases of Brock's question were: in the rank
one cases, a positive answer comes from Brock--Farb's result that 
the space is hyperbolic \cite{BrockFarb:curvature}; in the three rank-two cases, Brock--Masur proved that 
the space is relatively hyperbolic and thus that each quasiflat is contained in a single 
peripheral subset \cite[Theorem~3]{BrockMasur:WPrelhyp}.  In the 
general case, there were partial results providing coarse local control; in particular, there are theorems about flats being 
locally contained in linear size neighborhoods of standard flats, e.g., \cite[Theorem~8.5]{BKMM} and 
\cite[Theorem~A]{EskinMasurRafi:large_scale_rank}.

Fundamental groups of non-geometric $3$--manifolds are HHSs 
of rank 2 \cite[Theorem 10.1]{hhs2}.  For these groups, Theorem~\ref{thmi:main}
allows us to recover the following quasiflats theorem, which was first 
established by  Kapovich--Leeb:

\begin{cori}[Quasiflats theorem for non-geometric $3$--manifolds; 
    \cite{KapovichLeeb:haken}] Any top-dimensional quasiflat in a 
    non-geometric $3$--manifold is a uniformly bounded distance from a 
    finite union of standard flats.
\end{cori}

Some quasiflat theorems have previously been obtained 
for CAT(0) spaces satisfying particular conditions. 

One such result is due to Bestvina--Kleiner--Sageev, who proved that, for two-dimensional, proper, 
piecewise Euclidean CAT(0) complexes admitting cocompact group actions, every two-dimensional quasiflat lies within  
finite distance from a subset which is locally flat outside a compact 
set \cite{BKS:quasiflatsCAT0}. 

Generalizing that result, Huang proved 
that in an $N$--dimensional CAT(0) cube complex, every 
$N$--dimensional quasiflat lies within finite distance from a 
finite union of standard orthants \cite{Huang:quasiflats}. 
The following corollary of Theorem~\ref{thmi:main}  
generalizes the above noted theorems of \cite{BKS:quasiflatsCAT0} and \cite{Huang:quasiflats} in certain cocompact
cases:

\begin{cori}[Quasiflats theorem for cubulated groups with factor 
	systems] 
    \label{cori:cubulated}
 Let $\cuco X$ be a CAT(0) cube complex admitting a \emph{factor system} in the sense of~\cite{hhs1}.  Let $\nu$
 be the maximum
 dimension of an $\ell_1$--isometrically embedded cubical orthant in
 $\cuco X$.  Let $f\co \reals^\nu\to\cuco X$ be a quasi-isometric
 embedding.  Then
 $\dist_{haus}(f(\reals^\nu),\cup_{i=1}^kQ_i)<\infty$, where each
 $Q_i$ can be chosen to be: 
 \begin{itemize}
  \item an $\ell^1$--isometrically embedded copy of the standard cubical tiling of $[0,\infty)^\nu$, or
  \item if $\cuco X$ admits a cocompact group action,  
a $CAT(0)$--isometrically embedded copy of $[0,\infty)^\nu$ with the Euclidean metric.
 \end{itemize}
\end{cori}

It was established in \cite{hhs1} that all CAT(0) cube complexes with
proper, cocompact, cospecial (in the sense of Haglund-Wise~\cite{HaglundWise}) group actions admit factor systems.  More
generally, it is shown in~\cite{HagenSusse} that a CAT(0) cube complex
$\cuco X$ has a factor system whenever it admits a proper cocompact
action by a group $G$ satisfying any one of a number of natural
algebraic conditions, e.g., finite height for hyperplane stabilizers
or other weak versions of virtual cospecialness of the $G$--action.
In fact, that paper contains a characterization of actions that give
rise to a factor system.  We are not aware of any proper CAT(0) cube
complex that admits a proper cocompact group action but does not
contain a factor system (indeed, we have conjectured 
that all cubical groups admit factor systems, see \cite[Conjecture 
A]{hhs2}).

\begin{proof}[Proof of Corollary~\ref{cori:cubulated}]
As shown in~\cite{hhs1}, $\cuco X^{(1)}$ with the combinatorial
metric admits an HHS structure based on the construction
in~\cite[Section 8]{hhs1}.  In particular, the hierarchy paths/rays in
$\cuco X^{(1)}$ are combinatorial geodesics, so standard
$\nu$--orthants can be taken to
be $\ell_1$--embedded copies of the standard cubical tiling of
$[0,\infty)^\nu$. By Theorem~\ref{thmi:main} we are done, if we 
choose all our $Q_{i}$ to be of the first type listed above.

To conclude, it suffices to produce $N$ so that for any
$\ell_1$--isometric embedding $o\co \prod_{i=1}^\nu\gamma_i\to\cuco X$
with $\gamma_i$ a combinatorial geodesic ray, there is a CAT(0)
orthant $o'$ with $\dist_{haus}(\image (o),o')\leq N$.  For each $i$,
let $\cuco Y_i$ be the convex hull of $\gamma_i$, i.e., the
intersection of all combinatorial half-spaces containing $\gamma_i$.
Then the hull of $\image (o)$ decomposes as $\prod_{i=1}^\nu\cuco Y_i$.
Since $\cuco Y_i$ contains a CAT(0)--geodesic ray crossing all
hyperplanes, it suffices to show that $\cuco Y_i$ lies uniformly close
to $\gamma_i$.  But if there is no such bound, then for any $m$, we
can choose $o$ so that for some $i$, we have an $\ell_1$--isometric
embedding $[0,m]^2\to\cuco Y_i$, and thus an $\ell_1$--isometric
embedding $[0,m]^2\times[0,\infty)^{\nu-1}\to\cuco X$.  Cocompactness
would then allow us to produce a $(\nu+1)$--dimensional cubical
orthant in $\cuco X$, which is impossible by our choice of $\nu$.
\end{proof}

The quasiflats in Corollary~\ref{cori:cubulated} may have
dimension strictly less than the dimension of $\cuco X$, since a cube
complex may contain cubes of high dimension that are not contained in
cubical orthants. For instance, there exist hyperbolic (and hence
rank one) cubulated groups, whose associated CAT(0) cube complexes have
arbitrarily large dimension.  

In this sense, this corollary is
stronger than the main result in~\cite{Huang:quasiflats}; our
result applies even if the dimension is larger than the rank. On the
other hand, in practice, the construction of a factor system relies on
a geometric group action, a hypothesis not needed in the
context of \cite{Huang:quasiflats}. 

Although our results, applied in the cubical case, generalize some of those of \cite{Huang:quasiflats}, our proof is 
obtained by passing from the hierarchically hyperbolic space setting to a CAT(0) cube complex where the 
dimension equals the rank and then using Huang's theorem.  Specifically, we use the \emph{cubical approximations}, discussed 
immediately below, 
to construct a CAT(0) cube complex to which we can apply Huang's result, \cite[Theorem 1.1]{Huang:quasiflats}. So, Huang's result 
is a crucial ingredient in our work.

\subsection*{Approximating with cube complexes}
A key insight in the geometry of hyperbolic spaces is 
that in certain respects, they ``coarsely look like trees''; Gromov, in his famous treatise on 
the subject, introduced a number of ways in which 
this idea can be made precise \cite{Gromov}. One such statement is: in a hyperbolic space, the coarse convex hull of any 
finite set of points can be uniformly approximated by a geodesic tree \cite[\S 6.2 Geodesic trees]{Gromov}. 

It is now well understood that CAT(0) cube complexes are a natural generalization of trees.  Two important aspects of this 
idea are:

\begin{itemize}
     \item In a simplicial tree, the midpoint of any edge separates the tree into two complementary components.  In a CAT(0) 
cube complex, the midpoint of each edge is contained in a \emph{hyperplane}, a codimension--$1$ subspace with exactly two 
complementary components.  The revolutionary work of Sageev~\cite{Sageev}, elaborated later 
in~\cite{ChatterjiNiblo,Nica,HruskaWise}, shows that very general set-theoretic data --- a \emph{wallspace}, i.e. a set 
equipped with a suitable collection of bipartitions --- determines a CAT(0) cube complex in a canonical way.  We need this 
in Section~\ref{sec:hull_level}.

\item In a simplicial tree, any three vertices determine a unique geodesic tripod consisting of three geodesics, each of 
which joins two of the given points.  The intersection of the three geodesics is a single vertex, the \emph{median} of the 
three points.  Generalizing this, one obtains the class of \emph{median graphs}, i.e. graphs where each triple of vertices 
spans at least one metric tripod, all of which have the same center.  Chepoi showed~\cite{Chepoi} that there 
is a bijective correspondence between  one-skeleta of CAT(0) cube complexes and median graphs.  The median viewpoint on 
CAT(0) cube complexes is very useful, and we adopt it in various ways in this paper.  
\end{itemize}

In Section~\ref{sec:hull_level}, we generalize  
Gromov's theorem about hyperbolic groups to the setting of hierarchically hyperbolic 
spaces. Roughly, the theorem we prove establishes that the ``convex hull'' of a finite set $A$, denoted 
$H_\theta(A)$, is approximated by a finite CAT(0) cube complex.

This result provides a new tool for studying hierarchically hyperbolic 
spaces. Indeed, it is one of the key innovations which allows us to 
apply Huang's theorem about quasi-flats in CAT(0) cube complexes 
 \cite{Huang:quasiflats} to prove Theorem~\ref{thmi:main} about quasiflats in HHSs. Further, we expect that Theorem~\ref{thm:cubulated_hulls} will 
have a number of applications beyond those of this paper. A sketch of the 
proof of this result is provided later in the introduction.

\begin{thmi}[Approximation of convex hulls in HHSs by CAT(0) cube complexes]\label{thm:cubulated_hulls}
  Let $\cuco X$ be an asymphoric HHS of rank $\nu$.  Then for any $N$
  there exists $C$ so that the following holds.  Let $A\subseteq \cuco
  X$ have cardinality at most $N$.  Then there exists a CAT(0) cube
  complex $\cuco Y$ of dimension at most $\nu$ and a $C$--quasimedian
  $(C,C)$--quasi-isometry $\mathfrak p_A\co \cuco Y\to H_\theta(A)$.
\end{thmi}

A new proof of the preceding theorem, in a slightly more general
context, was given by Bowditch~\cite{Bowditch:hulls}, motivated by an 
early version of this paper.

Any HHS is coarse median in the sense
of Bowditch~\cite{Bowditch:coarse_median}, as shown in \cite[Section 7]{hhs2}. 
In the coarse median setting, there are several interesting 
precursors to our theorem. One which was particularly inspirational 
to us was Bowditch's result that in the asymptotic cone of a finite 
rank coarse median metric space, any top-dimensional closed Euclidean 
flat is cubulated, see \cite[Proposition 1.2]{Bowditch:rigid}. 
We will use Bowditch's result about cubulating 
top-rank Euclidean subsets in a complete median space in 
order to apply our result about cubulating arbitrary finite sets in an HHS.

In the coarse median setting, one has, by definition, 
``cubical approximations of finite sets,'' and there is also a 
nice metric
approximation result given by Zeidler \cite[Theorem 6.2]{Zeidler}. The approximation given by 
Theorem \ref{thm:cubulated_hulls} has stronger properties, as it provides an
approximation of the entire convex hull, whereas 
the \emph{quasimedian map} from a finite median algebra provided by the
coarse median property can be very far from having uniformly
(hierarchically) quasiconvex image.  (To see the distinction, consider
the case where $\cuco X=\integers^2$ and $A=\{(0,0),(n,n)\}$ for some
$n\geq0$.  Then the $\cuco Y$ provided by
Theorem~\ref{thm:cubulated_hulls} is a $n$--by--$n$ square, while the
$2$--point median algebra $\{(0,0),(n,n)\}$ satisfies the requirements of the
definition of a coarse median space, and is a ``metric approximation'' in
the sense of \cite{Zeidler} when endowed with the natural metric.)

Theorem~\ref{thm:cubulated_hulls} allows us to control the rank of
$\cuco X$ as a coarse median space more precisely than we did
in~\cite{hhs2}; see Corollary~\ref{cor:rank_coarse_median}.  This also
leads to a characterization of hierarchically hyperbolic spaces which
are hyperbolic, which we establish as Corollary~\ref{cor:hyperbolic}.

\subsection*{Induced quasi-isometries on factored spaces and
quasi-isometric classification} In \cite[\S 2]{hhs3}, we introduced
the notion of \emph{factored spaces} of an HHS. These are obtained
from a given HHS by ``coning off'' a collection of product regions. Factored spaces are HHSs themselves, with respect to a 
substructure of the original HHS structure. Factored spaces are central in the proof of finite
asymptotic dimension~\cite{hhs3}. 

A notable naturally-occurring example is that the Weil-Petersson metric on 
Teichm\"uller space is quasi-isometric to a factored space of the corresponding mapping
class group. In any HHS, we proved in \cite[Corollary~2.9]{hhs3} 
that there exists a factor space which is quasi-isometric to 
$\fontact S$ for the $\nest$--maximal element $S$ (e.g., for the 
mapping class group of a surface $S$ then 
$\fontact S$ is the curve graph of $S$).

In Theorem \ref{thm:cone_bound} we use the Quasiflats Theorem as a 
starting point to show that the image of any quasiflat in a certain
factored space is bounded. For now, we just state a new result about
mapping class groups which is a special case of Theorem
\ref{thm:cone_bound}:

\begin{thmi}[Quasiflats have finite diameter $\fontact S$ projection]\label{thmi:MCG}
 Let $(\cuco X,\mathfrak S)$ be the mapping class group of a
 non-sporadic surface $S$.  Then for every $K$ there exists $L$ so
 that any $(K,K)$--quasi-isometric embedding $f\co\reals^\nu\to
 \cuco X$ satisfies $\diam_{\fontact S}(\pi_S(f(\reals^\nu)))\leq L$.
\end{thmi}

In Corollary \ref{cor:induced_qi}, we prove that any quasi-isometry 
between HHSs (satisfying a mild condition) induces a quasi-isometry of the factored
spaces obtained by coning off the standard product regions containing
top-dimensional quasiflats.  This is very important because one can
extract further information about the original quasi-isometry from the
induced quasi-isometry on factored spaces, and even take further factored
spaces for additional data.  This is totally unexplored territory, since, for
example, it provides a way to study quasi-isometries of CAT(0) cube
complexes that requires leaving the world of cube complexes.  

We expect this strategy to be crucial to prove quasi-isometric rigidity
results for, say, right-angled Artin and Coxeter groups.
We discuss this in more detail below; for now we just give
an example of two right-angled Artin groups whose quasi-isometry
classes can be distinguished using this method, but not by any other known
methods: see Figure \ref{fig:two_graphs}.  

The obstruction to their 
being quasi-isometric is that, despite having the same rank, their factored
spaces as in Corollary \ref{cor:induced_qi} have different rank (which
is a quasi-isometry invariant by Theorem \ref{thm:rank_inv}).  We note
that the graphs we chose do not fit the hypotheses of
\cite{Huang:finite,Huang:infinite}, or that of any other class of 
right-angled Artin groups which have been classified including those 
considered in \cite{BehrstockNeumann:qigraph, 
BehrstockJanuszkiewiczNeumann:highdimartin, BestvinaKleinerSageev:RAAG1}.

\begin{figure}[h]
 \includegraphics[width=0.4\textwidth]{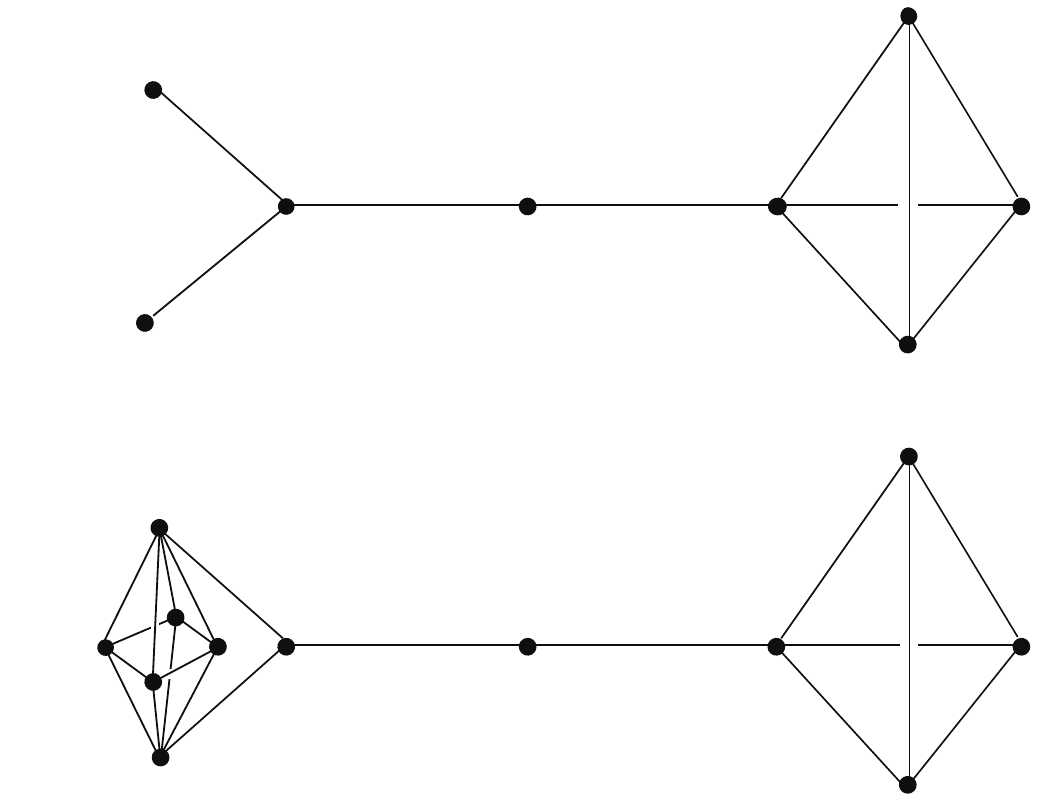}
 \caption{The right-angled Artin groups associated to the two graphs both have rank 4. However, the 4-dimensional flats get collapsed in the corresponding factored spaces, leaving only 2-dimensional flats in the case of the first RAAG, while there are 3-dimensional flats that persist in the case of the second RAAG.}\label{fig:two_graphs}
\end{figure}

\subsection*{Induced automorphisms of combinatorial data and 
quasi-isometric rigidity}

The Quasiflats Theorem provides a powerful tool for
proving quasi-isometric rigidity results for various HHSs, e.g. right-angled Artin and Coxeter groups.  In fact, the 
set of quasiflats and, more importantly, their intersection patterns, can be 
easily converted into purely combinatorial data.  

In good cases, one
can extract from the output of the Quasiflats Theorem (and with
basically no further knowledge about the geometry of the HHS) an automorphism of a
combinatorial structure encoding the data, and therefore reduce
proving quasi-isometric rigidity to proving that a certain
combinatorial structure is ``rigid''.  The kind of combinatorial
structure that the reader should keep in mind is $\mathfrak S$ endowed
with the partial order  given by nesting, $\nest$,  and the symmetric 
relation of orthogonality, $\orth$.

Rather than a general but complicated statement, we give a template
for this procedure.  In Theorem \ref{thm:clique_map} we give an
example of the combinatorial automorphism one can extract from a
quasi-isometry, under additional assumptions on the HHS. These
additional assumptions are satisfied by mapping class groups.
Accordingly, in Section~\ref {subsec:mcgqi}, we
use Theorem \ref{thm:clique_map} to give a new proof of
quasi-isometric rigidity of mapping class groups which, once we have 
established Theorem~\ref{thmi:main}, requires simpler combinatorial considerations than previous proofs, cf.\ 
\cite{BKMM,Bowditch:rigid,Hamenstadt:geometry}.

\begin{thmi}[QI rigidity for mapping class groups; \cite{BKMM}]\label{thm:mcg_rigid}
 Let $\cuco X$ be the the mapping class group of a non-sporadic
 surface $S$.  Then for any $K$ there exists $L$ so that: for each $(K,K)$--quasi-isometry $f\co\cuco X\to\cuco X$ there 
exists a mapping class $g$ so that $f$ $L$--coarsely coincides with
 left-multiplication by $g$.
\end{thmi}

Theorem \ref{thm:clique_map} applies to other spaces and groups as
well, including, for example, right-angled Artin groups with no triangles and no leaves in their
presentation graph, and fundamental groups of non-geometric graph
manifolds.  Variations of Theorem \ref{thm:clique_map} can be tailored
to treat other families of groups as well.

In the case of mapping class groups, there is no need to pass to
factored spaces, but in other contexts (e.g., the right-angled 
Artin groups in Figure \ref{fig:two_graphs}) 
the induced quasi-isometries on factored spaces provide extra
combinatorial data.  

In the study of right-angled Artin and Coxeter groups our results 
allow one to reduce the question of 
quasi-isometric rigidity to the following type of combinatorial 
problem, which we believe is of 
independent interest.

Let $\Gamma$ be a finite simplicial graph, and let
$B_\Gamma$ be either the associated right-angled Artin group or the
associated right-angled Coxeter group.  Recall from \cite[Section 8]{hhs1} that the standard hierarchically 
hyperbolic structure on such a group is obtained by setting $\mathfrak
S_{\Gamma}=\{gB_\Lambda\}/_\sim$, where $g\in B_\Gamma$ and $\Lambda$
is an induced subgraph of $\Gamma$, where $\sim$ is the equivalence
relation defined by $g B_\Lambda\sim h B_\Lambda$ if $g^{-1}h\in
B_{star(\Lambda)}$, and where $star(\emptyset)=\Gamma$ (i.e., $g^{-1}h$
commutes with each $b\in B_\Lambda$).  Declare $[gB_\Lambda]\nest
[gB_{\Lambda'}]$ if $\Lambda\subseteq \Lambda'$ and $[gB_\Lambda]\orth
[gB_{\Lambda'}]$ if $\Lambda\subseteq link(\Lambda')$.
Answers to the following can be used to obtain results on the problems of quasi-isometric 
rigidity and classification:

\begin{probi} 
Study the automorphism group $\Aut(\mathfrak S_\Gamma,\nest,\orth)$ of $(\mathfrak S_\Gamma,\nest,\orth)$. When is every 
element of $\Aut(\mathfrak S_\Gamma,\nest,\orth)$ induced by left multiplication by an element of $B_\Gamma$? When is every 
element of $\Aut(\mathfrak S_\Gamma,\nest,\orth)$ ``induced'' by an automorphism of $B_\Gamma$? (Not all automorphisms of 
$B_\Gamma$ need to ``induce'' an automorphism of $(\mathfrak S_\Gamma,\nest,\orth)$; which ones do?)
\end{probi}

Theorem~\ref{thm:clique_map} states that, under three natural
assumptions, a quasi-isometry $f\co(\cuco X,\mathfrak S)\to(\cuco
Y,\mathfrak T)$ induces a bijection from the set of \emph{hinges} of
$\cuco X$ to that of $\cuco Y$; a \emph{hinge} in $\cuco X$ is a pair
$(U,p)$ with $U\in\mathfrak S$ and $p\in\boundary \fontact U$, where $U$ has
the additional property that $U\in\{U_i\}_{i=1}^\nu$ where $\nu$ is
the rank of $\cuco X$, each $\fontact U_i$ is unbounded, and the $U_i$
are pairwise-orthogonal.

Since it preserves orthogonality, this bijection determines a
simplicial isomorphism from the union of the top-dimensional simplices
of the HHS boundary $\boundary\cuco X$ to $\boundary\cuco
Y$ (see~\cite{HHS_boundary} for more on the HHS boundary and its
simplices).  One should be able to articulate natural conditions
defining a subclass of HHSs for which one can use this map, perhaps
in conjunction with Section~\ref{sec:factored}, to pass from a
quasi-isometry to a map between HHS boundaries.

\subsection*{Sketch of the proof of the Cubulation of Hulls Theorem}
We provide here a sketch of the proof of 
Theorem~\ref{thm:cubulated_hulls} 
(Approximation of convex hulls in HHSs by CAT(0) cube complexes). This is 
one of the main tools we develop in this paper, allowing one to use ideas from the 
world of cube complexes to study HHSs. This 
result plays a crucial role in the proof of Theorem~\ref{thmi:main} (Quasiflats 
Theorem for HHSs). The full proof is given in 
Section~\ref{sec:hull_level}.

A hierarchically hyperbolic space can be roughly thought of as a subset of the product of a
(typically infinite) collection of hyperbolic spaces. This subset has 
the property that  its projection to any direct product of two factors is surjective if 
and only if those two factors are ``orthogonal.'' This allows 
one to move back and forth between properties of the HHS, $\cuco X$, and 
properties in the associated hyperbolic spaces, $\{\fontact U\}$. 
Here is a construction from \cite[\S 6]{hhs2} that exploits this point of view.   

Given a set a points in $\cuco X$, one can 
build the ``hull'' of that set by looking at the projections of that finite 
set of points to each of the associated hyperbolic spaces, $\fontact U$, taking 
coarse convex hulls in the $\fontact U$, and then looking at the points 
in $\cuco X$ that in each $\fontact U$ project close to the hull.

The \emph{realization theorem} \cite[Theorem 3.1]{hhs2} gives, roughly, a 
characterization of points in the product of the hyperbolic spaces that lie in the image of $\cuco X$, in terms of 
\emph{consistency conditions} (in the mapping class group context, one such condition is given by \cite[Theorem 
4.3]{Behrstock:asymptotic}). In this paper we rely on the construction of hulls and the realization theorem 
in an essential way. 

Following Sageev, the main method to cubulate a space is to explicitly
build \emph{walls}, that is, ``codimension-one'' subspaces which separate
the space. For hulls of a finite set of points in a hierarchically
hyperbolic space, this can be done in the following manner. Consider a finite set of points and their projections to each
hyperbolic space $\fontact U$.  By Gromov's theorem, in a hyperbolic
space the convex hull of a finite set of points can be uniformly
approximated by a geodesic tree \cite[\S 6.2 Geodesic trees]{Gromov}.

Taking an appropriately dense collection of points in each
such geodesic tree and considering their inverse images in $\cuco
X$, one obtains walls that can be used to construct a CAT(0) cube complex.  One needs to verify that this actually works as
needed to prove Theorem~\ref{thm:cubulated_hulls}. In particular, a key point is to show that any vertex in the CAT(0) 
cube complex ``corresponds'' to a point in the hull in $\cuco X$ of the finitely many points. Establishing this requires a 
careful analysis of the ``consistency conditions'', with the aim of
invoking the aforementioned realization theorem, \cite[Theorem 3.1]{hhs2}.

\subsection*{Sketch of the proof of the Quasiflats Theorem}
An important ingredient in our proof of Theorem~\ref{thmi:main}
(Quasiflats Theorem for HHSs) is Huang's result
\cite[Theorem~1.1]{Huang:quasiflats} which classifies $n$--dimensional
quasiflats in $n$--dimensional CAT(0) cube complexes (for emphasis: 
in Huang's theorem 
the dimension of the quasiflats under consideration coincides with
that of the CAT(0) cube complex).

We prove Theorem~\ref{thmi:main} by constructing an appropriate 
CAT(0) cube complex to which we can apply Huang's theorem.

The first step is to use a result of Bowditch \cite[Proposition 
1.2]{Bowditch:rigid} about ``local cubulations'' of top-dimensional flats in median metric spaces.  The median spaces in 
question are (bilipschitz equivalent to) asymptotic cones of the HHS $\cuco X$.  This allows us to show that 
any finite ball in a quasiflat is coarsely contained in the hull of a \emph{uniformly} bounded number of points.

By our cubulation of hulls theorem discussed above, Theorem~\ref{thm:cubulated_hulls}, we know that any such hull is 
uniformly quasi-isometric to a CAT(0) cube complex. Taking an ultralimit of these CAT(0) cube complexes, we obtain a finite 
dimensional CAT(0) cube complex which quasi-isometrically embeds in our HHS, and the quasiflat is contained in a bounded 
neighborhood of the image of the quasi-isometric embedding.

By construction and Theorem~\ref{thm:cubulated_hulls}, the CAT(0) cube complex we build has the same 
dimension as the quasiflat, thus allowing us to apply Huang's result \cite[Theorem~1.1]{Huang:quasiflats}. This finishes the 
proof since one can show that the orthants in the CAT(0) cube complex we construct correspond to 
standard orthants in the original HHS. 

As can be seen in this sketch, our theorem relies on 
Huang's result \cite[Theorem~1.1]{Huang:quasiflats} in an essential 
way. We also note that because the CAT(0) cube complex we construct is built using Theorem~\ref{thm:cubulated_hulls}, this 
complex always has the same dimension as its top-dimensional quasiflats. So,  
although our argument factors through Huang's theorem, our result extends Huang's in the setting of cocompact CAT(0) cube complexes 
with factor systems. For instance, our theorem applies to 
arbitrary right-angled Coxeter groups, even though the dimension of the CAT(0) cube complex associated to a right-angled Coxeter is typically
(much) larger than the dimension of the quasiflats it contains.

\subsection*{Outline}
Section~\ref{sec:background} contains background on hierarchically
hyperbolic spaces, wallspaces/cube complexes, median and coarse median spaces, and asymptotic cones.
In Section~\ref{sec:hull_level} we build walls in hulls of finite
sets, proving Theorem~\ref{thm:cubulated_hulls}.  The main goal of
Section~\ref{sec:quasiflats_and_cones} is to prove
Corollary~\ref{cor:fixed_flat_ball_hull}, showing that balls in
quasiflats in an HHS can be uniformly well-approximated by hulls of
uniformly finite sets of points.  In Section~\ref{sec:main},
we develop background on standard orthants in HHSs, and then prove
Theorem~\ref{thmi:main}. We also prove stronger versions in which we
control both the number of standard orthants (using a volume growth
argument) and the distance from the quasiflat to the approximating
orthants, in terms of the quasi-isometry constants.  In
Section~\ref{sec:standard_flat}, we impose additional assumptions on
an HHS enabling one to study the effect of quasi-isometries on the
underlying combinatorial structure; see Theorem~\ref{thm:clique_map}. It is in this section that we give a new proof of 
quasi-isometric rigidity of the mapping class group, i.e., Theorem~\ref{thm:mcg_rigid}.  Finally, in
Section~\ref{sec:factored}, we discuss factored spaces.  We first prove Theorem~\ref{thm:cone_bound} and then deduce 
Corollary~\ref{cor:induced_qi}, which is about induced quasi-isometries of factored spaces.

\subsection*{Acknowledgments}
We would like to thank an anonymous ancient Greek scholar for early 
input on this paper, in particular for 
suggesting the adjective ``asymphoric,'' derived from the word
$\sigma\upsilon\mu\phi o \rho\acute{\alpha}$ and roughly meaning 
``mishap averted.''

We thank Ela Behrstock for drawing two of the figures (Figure~\ref{fig:bridge} and 
Figure~\ref{fig:same_walls}).  We thank Jingyin Huang, for pointing out that the RAAGs
appearing in Figure~\ref{fig:two_graphs} in an earlier version of this
paper could be distinguished by using
\cite[Theorem~3.28]{Huang:finite} and a clever trick.  We 
thank Brian Bowditch, Harry Petyt, and Jacob Russell for several
helpful comments. We also thank the numerous referees for very useful comments; in 
particular, the 87 enumerated comments of one of the referees helped us to significantly improve the exposition.

The authors were supported by the National Science Foundation
under Grant No.\  DMS-1440140 at the Mathematical Sciences Research
Institute in Berkeley during Fall 2016 program in Geometric Group
Theory. 

Behrstock was supported by NSF grant DMS-1710890. Hagen was supported by EPSRC grant EP/L026481/1. 

Sisto was partially supported by the Swiss National Science Foundation (grant \#182186).

We thank the Isaac Newton 
Institute for Mathematical
Sciences, Cambridge, for support and hospitality during the program
\emph{Non-positive curvature group actions and cohomology} where work
on this paper was done; this work was partly supported by EPSRC grant
no.\  EP/K032208/1.  

\section{Background}\label{sec:background}

\subsection{Hierarchically hyperbolic spaces}
Throughout this paper, we work with a \emph{hierarchically hyperbolic space}, which is a pair $(\cuco X,\mathfrak S)$ with 
some additional extra structure described in Definition 1.1 of~\cite{hhs2}.  
Roughly, an HHS consists of:
\begin{itemize}
 \item a quasigeodesic metric space $\cuco X$;
 \item a set of uniformly hyperbolic spaces $\{\fontact U:U\in\mathfrak S\}$;
 \item uniformly coarsely-Lipschitz coarsely-surjective maps $\pi_U\co \cuco X\to\fontact U$;
 \item three relations $\nest$ (a partial order), $\orth$ (an anti-reflexive symmetric relation), $\transverse$ (the complement of $\nest$ and $\orth$) on $\mathfrak S$;
 \item a unique $\nest$--maximal element of $\mathfrak S$, and a uniform bound on the length of $\nest$--chains in $\mathfrak S$;
 \item for $U\propnest V$ or $U\transverse V$, a uniformly bounded set $\rho^U_V$;
 \item for $U\propnest V$, a coarse map $\rho^V_U:\fontact V\to\fontact U$.
\end{itemize}
Definition 1.1 of~\cite{hhs2} consists of several axioms governing
this data.  The definition and basic properties of HHSs were first laid
out in \cite{hhs1}; below we list \cite{hhs2} as the primary reference
since a few of the properties were first established there and this
provides for unified notation.  The properties of HHSs which are
central to this article are listed below.

\begin{remi}[QI invariance]
As explained in~\cite[Proposition 1.10]{hhs2}, the property of being a
hierarchically hyperbolic space is preserved under quasi-isometries.  If $(\cuco X,\mathfrak S)$ is a hierarchically hyperbolic
space and $f:\cuco X'\to\cuco X$ is a quasi-isometry, then $(\cuco
X',\mathfrak S)$ is an HHS; where the structure in $\cuco X$ is 
obtained  
by replacing each
projection $\pi_U,U\in\mathfrak S$ by $\pi_U\circ f$.
\end{remi}

The first property says that the ``coordinates'' $(\pi_U(x))_{U\in\mathfrak S}$ for some $x\in\cuco X$ cannot be arbitrary. In fact, for certain pairs $U,V$ there are conditions that need to be satisfied by $\pi_U(x),\pi_V(x)$. There is no condition for $U\orth V$, which corresponds to the fact that in this case $U,V$ should be thought of as factors of a product region, as we will see later.

\begin{axiom}[Consistency axioms]\label{axiom:consistency}
Let $(\cuco X,\mathfrak S)$ be hierarchically hyperbolic.  Then there is a constant $E=E(\cuco X,\mathfrak S)$ so that the following hold for all $x\in\cuco X$ and $U,V,W\in\mathfrak S$:
\begin{itemize}
 \item if $V\transverse W$, then $$\min\left\{\dist_{
 W}(\pi_W(x),\rho^V_W),\dist_{
 V}(\pi_V(x),\rho^W_V)\right\}\leq E;$$
 \item if $V\propnest W$, then $$\min\left\{\dist_{
 W}(\pi_W(x),\rho^V_W),\diam_{\fontact
 V}(\pi_V(x)\cup\rho^W_V(\pi_W(x)))\right\}\leq E.$$ 
\end{itemize}
Finally, if $U\nest V$, then $\dist_W(\rho^U_W,\rho^V_W)\leq E$ whenever $W\in\mathfrak S$ satisfies either $V\propnest W$ or $V\transverse W$ and $W\not\orth U$.
\end{axiom}

\begin{remi}[Consistent tuples]
The consistency axiom has a sort of converse, the \emph{realization theorem} stated below.  The idea is that the projection 
maps $\pi_U,U\in\mathfrak S$ allow us to think of points in $\cuco X$ as tuples in $\prod_{U\in\mathfrak S}\fontact U$.  The 
consistency axiom imposes conditions on which tuples can be in the image of the map $\cuco X\to\prod_{U\in\mathfrak 
S}\fontact U$ given by the projections.  The realization theorem says that the consistency conditions actually 
(coarsely) characterize tuples in the image of $\cuco X$.  The 
precise statement is Theorem~\ref{thm:realization}, which is 
formulated using the notion of a \emph{consistent tuple}, which we now define.

Fix a constant $\kappa\ge0$.  Given a tuple $(b_U)_U\in\prod_{U\in\mathfrak S}\fontact U$, we say that $(b_U)_U$ is 
\emph{$\kappa$--consistent} if it satisfies the conditions from Axiom~\ref{axiom:consistency}, except with each occurrence 
of $\pi_V(x)$ (resp. $\pi_W(x)$) replaced by $b_V$ (resp. $b_W$).  So, the axiom says that tuples in the image of $\cuco X$ 
are $E$--consistent.
\end{remi}

Now we can state the realization theorem:

\begin{thm}[Realization of consistent tuples; \cite{hhs2}]\label{thm:realization}
 For each $\kappa\geq1$ there exist $\theta_e,\theta_u\geq0$ such that
 the following holds.  Let $\tup
 b\in\prod_{W\in\mathfrak S}2^{\fontact W}$ be $\kappa$--consistent 
 (\cite[Definition 1.17]{hhs2});
 for each $W$, let $b_W$ denote the $\fontact W$--coordinate of $\tup
 b$.

 Then there exists $x\in \cuco X$ so that $\dist_{
 W}(b_W,\pi_W(x))\leq\theta_e$ for all $\fontact W\in\mathfrak S$.
 Moreover, $x$ is \emph{coarsely unique} in the sense that the set of
 all $x$ which satisfy $\dist_{ W}(b_W,\pi_W(x))\leq\theta_e$ in each
 $\fontact W\in\mathfrak S$, has diameter at most $\theta_u$.
\end{thm}

The realization theorem is one of what we see as three foundational theorems about HHSs.  The other two are closely related: 
the \emph{distance formula} and the existence of \emph{hierarchy paths}.

The distance formula provides a way to compute distances in $\cuco X$
in terms of distances in the various $\fontact U$, thereby reducing
the study of the geometry of $\cuco X$ to that of the family of
hyperbolic spaces $\{\fontact U\}_{U\in\mathfrak S}$.

We write $A\asymp_{K,C} B$ if $A/K-C\leq B\leq KA+C$. Also, we let $\ignore{A}{s}=A$ if $A\geq s$, and $\ignore{A}{s}=0$ 
otherwise. Moreover, we denote $\dist_{ W}(x,y)=\dist_{\fontact W}(\pi_W(x),\pi_W(y))$ (the distance between $x$ and $y$ from 
the point of view of $W$).

\begin{thm}[Distance Formula; \cite{hhs2}]\label{thm:distance_formula}
 Let $(X,\mathfrak S)$ be hierarchically hyperbolic. Then there exists $s_0$ such that for all $s\geq s_0$ there exist
 constants $K,C$ such that for all $x,y\in\cuco X$,
 $$\dist_{\cuco X}(x,y)\asymp_{K,C}\sum_{W\in\mathfrak S}\ignore{\dist_{ W}(x,y)}{s}.$$
\end{thm}

\begin{remi}[Uniqueness axiom]
Notice that a special case of the distance formula is that, roughly
speaking, if $x,y\in\cuco X$ are so that $\pi_U(x),\pi_U(y)$ are close
for each $U$, then $x,y$ are close in $\cuco X$.  This special case is
the \emph{uniqueness axiom}, which is part of the definition of a
hierarchically hyperbolic space~\cite[Definition 1.1.(9)]{hhs2}.
There are various places in Section~\ref{sec:hull_level} where we
apply the distance formula, but could probably get away with just
using the uniqueness axiom.  In fact, since we initially posted this
paper, Bowditch has given an independent proof of
Theorem~\ref{thm:cubulated_hulls}, not using the distance formula.
One can then deduce the distance formula from
Theorem~\ref{thm:cubulated_hulls}, which Bowditch does in
\cite{Bowditch:hulls}.
\end{remi}

\emph{Hierarchy paths} are quasi-geodesics in the HHS whose 
projections to 
each associated hyperbolic space
are 
(coarsely) monotone.  Any two
points can be joined by a hierarchy path:

\begin{thm}[Existence of Hierarchy Paths; \cite{hhs2}]\label{thm:monotone_hierarchy_paths}
Let $(\cuco X,\mathfrak S)$ be hierarchically hyperbolic. Then there exists $D$ so that any $x,y\in\cuco X$ are 
joined by a $D$-hierarchy path, i.e., a $(D,D)$--quasi-geodesic projecting to an unparameterized 
$(D,D)$--quasi-geodesic between $\pi_U(x)$ and $\pi_U(y)$ in $\fontact U$ for each $U\in\mathfrak S$.  
\end{thm}

The following says that when moving along a hierarchy path $\gamma$,
in order to change projection to $\fontact U$, when $U\propnest V$,
one must pass close in $\fontact V$ to a specific point, namely
$\rho^U_V$.  The first assertion is the \emph{bounded geodesic image}
axiom for an HHS~\cite[Definition 1.1.(6)]{hhs2} and the second
assertion follows easily from the first, together with
Axiom~\ref{axiom:consistency}; for ease of reference we record this 
here as:

\begin{lem}(Bounded geodesic image)\label{lem:BGI}
 Let $\cuco X$ be a hierarchically hyperbolic space. There exists $B$ so that the following holds. Let $W\in\mathfrak S$, $V\propnest W$.  Suppose that $\gamma$ is a geodesic in $\fontact W$ with $\gamma\cap\neb_B(\rho^V_W))=\emptyset$.  Then $\diam_{\fontact V}(\rho^W_V(\gamma))\leq B$.
 
 Moreover, suppose $x,y\in \cuco X$ and that there exists a geodesic $\gamma$ in $\fontact W$ from $\pi_W(x)$ to $\pi_W(y)$ so that $\gamma\cap\neb_B(\rho^V_W))=\emptyset$. Then $\dist_V(x,y)\leq B$.
\end{lem}

Another part of the definition of a hierarchically hyperbolic space is the \emph{large links axiom} (Definition 1.1.(7) 
in~\cite{hhs2}).  It says roughly that, if $x,y\in\cuco X$ and $V\in\mathfrak S$, then the number of $U\propnest V$ on which 
$x,y$ have very different projections, and $U$ is $\nest$--maximal with those properties, can be bounded in terms of 
$\dist_V(x,y)$.  Typically, one does not apply the large links axiom directly.  Instead, one uses a consequence, Lemma 2.5 
of~\cite{hhs2}, which we call ``passing up large projections.''  We will use a variant of that lemma, which we state 
presently (it is applied in an essential way in the proof of Lemma~\ref{lem:same_walls}, which is part of the proof of 
Theorem~\ref{thm:cubulated_hulls}).

For $V\in\mathfrak S$, we denote $\mathfrak S_V=\{U\in\mathfrak S: U\nest V\}$.

\begin{lem}[Passing large projections up the $\nest$--lattice]\label{lem:passing_up}
 There exists $E$ with the following property. For every $C\geq0$ there exists $N_0=N_0(C)$ with the following property.  Let 
$V\in\mathfrak S$, let $x,y\in\cuco X$, and let $\{V_i\}_{i=1}^{N_0}\subseteq \mathfrak S_V$ be distinct and satisfy 
$\dist_{V_i}(x,y)\geq E$. Then there exists $W\in\mathfrak S_V$ and $i,j$ so that $V_i,V_j\propnest W$ and 
$\dist_{W}(\rho^{V_i}_W,\rho^{V_j}_W)\geq C$.
\end{lem}

\begin{exmp}\label{exmp:raag_pass_up}
Since the statement of the preceding lemma is somewhat opaque, we now give an example before proceeding to the proof.  Let 
$\cuco X$ be the Cayley graph of the free group on generators $a,b$.  We can make $\cuco X$ an HHS by taking $\mathfrak S$ 
to consist of all left cosets of all subgroups generated by subsets of $\{a,b\}$.  The space $\fontact \langle a\rangle$ is 
just $\reals$, and similarly for $\fontact\langle b\rangle$.  The space $\fontact\langle a,b\rangle$ is obtained from $\cuco 
X$ by coning off each coset in $\mathfrak S$.  

Consider the path $w=(a^Eb^E)^N$, for 
some $E\geq 1$.  Then there are $N$ cosets of $\langle a\rangle$ and $\langle b\rangle$ on which the endpoints of the above 
path have projections lying at distance $E$.  For any $C$, by making $N$ sufficiently large, we see that the coset $\langle 
a\rangle$ and the coset $wb^{-E}\langle b\rangle$ are at least $C$--distant in $\fontact\langle a,b\rangle$ and hence 
satisfy the conclusion of the lemma.      
\end{exmp}

\begin{proof}[Proof of Lemma~\ref{lem:passing_up}]
First of all, we choose constants. Let $B\geq 1$ be the constant from Lemma~\ref{lem:BGI}, and suppose that $B$ is also an upper bound on the diameter of $\rho^U_V$ for any $U\propnest V$. Moreover, supposed $B\geq D$, for $D$ as in Theorem \ref{thm:monotone_hierarchy_paths}, and moreover that $(D,D)$--quasi-geodesics in a $\delta$--hyperbolic space stay $B$--close to geodesics with the same endpoints, where $\delta$ is a hyperbolicity constant for all the $\fontact U$.

If $U\in\mathfrak S$ is $\nest$--minimal, we say that its \emph{level} is 1. Inductively, $U\in\mathfrak S$ has level $k$ if it is $\nest$--minimal among all $V\in\mathfrak S$ not of level $\leq k-1$. The proof is by induction on the level $k$ of a $\nest$-minimal $V\in\mathfrak S$ into which each $V_i$ is nested, with $E=100kB$. The base case $k=1$ is empty.
  Suppose that the statement holds for a given $N=N(k)$ when the level of $V$ as above is at most $k$. Suppose instead that $|\{V_i\}|\geq N(k+1)$ (where $N(k+1)$ is a constant much larger than $N(k)$ that will be determined shortly) and there exists a $\nest$-minimal $V\in\mathfrak S$ of level $k+1$ into which each $V_i$ is nested.  There are two cases.

If $\max_{i,j}\{\dist_{V}(\rho^{V_i}_V,\rho^{V_j}_V)\}\geq C$, then we are done. Hence, suppose not. All the $\rho^{V_i}_V$ lie $B$--close to a geodesic $[\pi_V(x),\pi_V(y)]$ by bounded geodesic image, and by the assumption they all lie close to a sub-geodesic of length $C+10B$. Hence, we can replace $x,y$ with suitable $x',y'$ on a hierarchy path from $x$ to $y$ chosen so that
\begin{itemize}
 \item $\dist_V(x',y')\leq C+100B$,
 \item  $\pi_V(x'),\pi_V(y')$ lie $B$--close to a geodesic $[\pi_V(x),\pi_V(y)]$, and
 \item  the geodesics $[\pi_V(x),\pi_V(x')]$, $[\pi_V(y),\pi_V(y')]$ do not pass $B$--close to any $\rho^{V_i}_V$.
\end{itemize}

 By Lemma~\ref{lem:BGI}, $\dist_{V_i}(x',y')\geq 100kB$, since $\dist_{V_i}(x',y')$ is approximately equal to $\dist_{V_i}(x,y)$.

The large link axiom (\cite[Definition 1.1.(6)]{hhs2}) implies that there exists $K=K(C+100B)$ and $T_1,\dots,T_K$,  each 
properly nested in $V$ (thus of level strictly less than $k+1$), so that any $V_i$ is nested in some $T_j$. In particular, if 
$N(k+1)\geq KN(k)$, there exists $j$ so that $\geq N(k)$ elements of $\{V_i\}$ are nested into $T_j$. By the induction 
hypothesis, we are done.
\end{proof}

\subsubsection{A few more basic HHS notions}\label{subsubsec:basic_hhs_notions} 
We now collect a few more basic notions about HHSs that will be used throughout the paper.

First, each of the HHS axioms (and their variants stated above) involves some constants, which are taken to be part of the 
HHS structure $(\cuco X,\mathfrak S)$.  For the sake of sanity, where possible, we can assume these constants are all the 
same:

\begin{notation}[Naming constants]
 In the remainder of the paper, following \cite[Remark 1.6]{hhs2}, 
 we fix a constant $E$ larger than each of the constants in \cite[Definition 1.1]{hhs2} and also satisfying the conclusion of 
Lemma~\ref{lem:passing_up}, Lemma~\ref{lem:BGI}, and Axiom~\ref{axiom:consistency}.
\end{notation}

Given $x,y\in\cuco X$, it is convenient to consider the subset of $\mathfrak S$ on whose associated hyperbolic spaces $x,y$ 
project far apart, where ``far'' is determined by some threshold, generally specified in advance independently of $x,y$:

\begin{defn}[Relevant]\label{defn:relevant}
Given points $x,y\in\cuco X$, we say that $U\in\mathfrak S$
is \emph{relevant} (with respect to $x,y$ and a constant $\theta>0$) if
$\dist_U(x,y)>\theta$.  Denote by $\relevant_{\theta}(x,y)$ the
set of relevant elements.  Note that, for all sufficiently large $\theta$, the distance formula implies that 
$\relevant_\theta(x,y)$ is finite.  In fact, using Lemma~2.5 of~\cite{hhs2}, one can bound its cardinality in terms of 
$\theta,E$, and $\dist_{\cuco X}(x,y)$ without using the distance formula.
\end{defn}

The notion of the rank of $(\cuco X,\mathfrak S)$ is easy to define, 
but it is of significant importance in the present paper:

\begin{defn}[Rank]
The \emph{rank} $\nu=\nu(\cuco X,\mathfrak S)$ of the HHS $(\cuco
X,\mathfrak S)$ is the maximal $n$ so that there exist pairwise
orthogonal $U_1,\dots,U_n\in\mathfrak S$ for which $\pi_{U_i}(\cuco X)$
is unbounded for all $i$.
\end{defn}

The rank is closely related to \emph{standard product regions} in $\cuco X$, which are a useful tool whose 
construction we now review; see also 
\cite[Section~13]{hhs1} and 
\cite[Section 5]{hhs2}. These products are built out of 
the following two spaces, which we define abstractly, but often 
implicitly identify with their images as subsets of $\cuco X$.

\begin{defn}[Nested partial tuples]\label{defn:nested_partial_tuple}
Recall that $\mathfrak S_U=\{V\in\mathfrak S \mid V\nest U\}$.  Fix
$\kappa\geq E$ and let $\FU U$ be the set of
$\kappa$--consistent tuples in $\prod_{V\in\mathfrak S_U}2^{\fontact
V}$.
\end{defn}

\begin{defn}[Orthogonal partial tuples]\label{defn:orthogonal_partial_tuple}
Let $\mathfrak S_U^\orth=\{V\in\mathfrak S\mid V\orth U\}$. Fix $\kappa\geq E$ and let $\EU U$ be the set of
$\kappa$--consistent tuples in $\prod_{V\in\mathfrak
S_U^\orth}2^{\fontact V}$.
\end{defn}

\begin{defn}[Standard product regions in $\cuco X$]\label{const:embedding_product_regions}
Given $\cuco X$ and $U\in\mathfrak S$, there are coarsely well-defined
maps $\phi^\nest,\phi^\orth\co\FU U,\EU U\to\cuco X$ 
which extend to a coarsely
well-defined map $\phi_U\co\FU U\times \EU U\to\cuco X$. 
Indeed, for each $(\tup a,\tup b)\in
\FU U\times \EU U$, and each $V\in\mathfrak S$, the
coordinate $(\phi_U(\tup a,\tup b))_V$ is defined as follows.  If $V\nest U$,
then $(\phi_U(\tup a,\tup b))_V=a_V$.  If $V\orth U$, then
$(\phi_U(\tup a,\tup b))_V=b_V$.  If $V\transverse U$, then
$(\phi_U(\tup a,\tup b))_V=\rho^U_V$.  Finally, if $U\propnest V$, let $(\phi_U(\tup a,\tup b))_V=\rho^U_V$. We refer to 
$\FU U\times \EU U$ as a \emph{standard product region}, whose image in $\cuco X$ we also call a \emph{standard product 
region} and denote by $P_U$.
\end{defn}

The image of $F_U$ in $\cuco X$ is again a hierarchically hyperbolic
space, with index set $\mathfrak S_U$ and hyperbolic spaces and
projections inherited from those in $\mathfrak S$.  The same is true
for $E_U$, although one must replace $\mathfrak S_U$ with the set of
$V\in\mathfrak S$ such that $V\orth U$, together with some element
$A\in\mathfrak S$ into which each such $V$ is nested (such an $A$ is
provided by the HHS axioms).  We won't have much need for this here,
and refer the interested reader to \cite[Section 5]{hhs2} for details.

\subsubsection{Hierarchically hyperbolic groups}
A finitely generated group $G$ is a \emph{hierarchically hyperbolic group (HHG)} if some (hence any) 
Cayley graph of $G$ is an HHS, and the HHS structure is $G$--invariant.  Specifically, an HHG is a finitely 
generated group $G$, equipped with a specific word metric, so that there is an HHS $(G,\mathfrak S)$ where:
\begin{itemize}
     \item $G$ acts cofinitely on $\mathfrak S$, preserving each relation $\nest,\orth,\transverse$;
     \item for each $U\in\mathfrak S$ and $g\in G$, there is an 
	 isometry $g\colon\fontact U\to\fontact gU$, and if $h\in G$, then 
the isometry $gh\colon\fontact U\to\fontact ghU$ is the same as the composition $\fontact U\stackrel{h}{\longrightarrow}\fontact 
hU\stackrel{g}{\longrightarrow}\fontact ghU$;
\item for each $U\in\mathfrak S$ and $g,x\in G$, the points $g\circ \pi_U(x)$ and $\pi_{gU}(gx)$ uniformly coarsely 
coincide;
\item for each $U,V\in\mathfrak S$ such that $U\transverse V$ or $U\propnest V$, and each $g\in G$, we have 
$\rho^{gU}_{gV}=g(\rho^U_V)$.
\end{itemize}

Examples of hierarchically hyperbolic groups include mapping class groups of finite-type orientable surfaces and fundamental 
groups of compact special cube complexes, see \cite{hhs1,hhs2} for 
details and additional examples.

The only property of HHGs that we use in this paper is immediate from
the definition, in particular from the 
property that $G$ acts cofinitely on $\mathfrak S$: 
there exists $C\geq0$ such that for all $U\in\mathfrak
S$, either $\diam(\fontact U)\leq C$, or $\fontact U$ has unbounded
diameter.

\subsubsection{Rank as a quasi-isometry invariant}

We now introduce a technical assumption on the HHS that we will assume 
throughout the paper.  This condition is satisfied by all HHGs; it is 
also satisfied for all naturally occurring examples
of HHSs. We impose it in order to rule out product regions with
bounded but arbitrarily large factors.  Our results likely have analogues that hold in the absence of this
hypothesis, but would require custom-tailoring to the situation at 
hand.

 \begin{defn}[Asymphoric]\label{defn:asymphoric}
 We say that the HHS $(\cuco X,\mathfrak S)$ of rank $\nu$ is
 \emph{asymphoric} if there exists a constant $C$ with the property
 that there does not exist a set of $\nu+1$ pairwise orthogonal
 elements $U$ of $\mathfrak S$ where each $\fontact U$ has diameter at
 least $C$.  In this case, without loss of generality, we assume that
 $E$ is chosen to be at least as large as $C$.
\end{defn}

For completeness, we remark that a result from \cite{hhs1}  
implies that the rank is a quasi-isometry invariant of asymphoric HHSs:

\begin{thm}[Quasi-isometry invariance of rank]\label{thm:rank_inv}
Let $(\cuco X,\mathfrak S)$ be an asymphoric HHS. Then the rank $\nu$
of $\cuco X$ coincides with the maximal $n$ for which there exists $K$
and $(K,K)$--quasi-isometric embeddings $f\co(B_R(0)\subseteq
\reals^n)\to \cuco X$ for all $R\geq 0$.  In particular, the rank is a
quasi-isometry invariant of asymphoric HHS.
\end{thm}

\begin{proof}
 It is easy to construct a quasi-isometric embeddings of balls in
 $\reals^n$ starting from $n$ pairwise orthogonal elements $U$ of
 $\mathfrak S$ with unbounded $\pi_U(\cuco X)$.  Hence, we have to
 show that if there exist quasi-isometric embeddings as in the
 statement, then $n$ is at most the rank.  This is because, by
 \cite[Theorem 13.11.(2)]{hhs1}, there exists an asymptotic cone
 $\seq{\cuco X}$ where a copy of the unit ball in $\mathbb R^n$ is
 contained in an ultralimit of standard boxes.  These are products of
 intervals contained in a subspace coarsely decomposing as product whose factors are various subspaces $F_U$, 
so that any ultralimit of
 standard boxes in $\seq{\cuco X}$ is homeomorphic to a subset of
 $\reals^\nu$ because $\cuco X$ is asymphoric.  Hence, $n\leq \nu$, as
 required.
\end{proof}

\subsection{Hulls and gates}

Sets in an HHS have \emph{hulls}, built from coarse convex hulls in 
hyperbolic spaces:

\begin{defn}[Hull of a set; Section 6 of \cite{hhs2}]\label{defn:hull}
 For each $A\subset \cuco X$ and $\theta\geq 0$, let the \emph{hull}, 
 $H_{\theta}(A)$,  be the set of all $p\in\cuco X$ so that, for each
 $W\in\mathfrak S$, the set $\pi_W(p)$ lies at distance at most
 $\theta$ from $\hull_{\fontact W}(A)$, the coarse convex hull of $A$ in the 
 hyperbolic space $\fontact W$ (that is to say, the union of all geodesics in $\fontact W$ joining points of $A$). 
 Note that $A\subset H_{\theta}(A)$.
\end{defn}

Hulls are examples of \emph{hierarchically quasiconvex} subspaces of $\cuco X$.  The other notable examples are standard 
product regions.  The idea behind hierarchical quasiconvexity is to simultaneously capture (in our coarse setting) various 
notions of (coarse) convexity:
\begin{itemize}
     \item First, hierarchical quasiconvexity directly generalizes the usual notion of quasiconvexity in a hyperbolic space: 
when $\cuco X$ is a hyperbolic HHS, the two notions coincide.  More generally, hierarchical quasiconvexity of a subspace 
$\cuco Y\subset\cuco X$ requires that $\cuco Y$ has uniformly quasiconvex projections to all hyperbolic spaces $\fontact 
U$ for $U\in\mathfrak S$.
\item Second, hierarchical quasiconvexity imitates, in the HHS setting, the notion of a convex subcomplex $A$ of a 
CAT(0) cube complex $M$.  That notion has many equivalent formulations; one of them says that $M$ is convex provided that 
the median of $x,y,z$ lies in $A$ whenever at least two of the 
vertices $x,y,z$ lie in $A$.  This generalizes naturally to a 
notion of coarse median convexity in Bowditch's \emph{coarse median spaces}~\cite{Bowditch:coarse_median}, which are 
discussed in more detail below.  It was verified in~\cite[Section 7]{hhs2} that HHSs are coarse median spaces (we rely 
heavily on this 
fact in the rest of the paper) and that hierarchically quasiconvex subspaces are coarsely median convex.  Recently, 
Russell-Spriano-Tran have proved the converse~\cite{RST}.
\item From a point of view that emphasizes paths rather than coarse medians or projections, hierarchically quasiconvex 
subspaces are ``quasiconvex with respect to hierarchy paths'': if $\cuco Y$ is hierarchically quasiconvex, then any 
hierarchy path with endpoints in $\cuco Y$ stays close to $\cuco Y$.
\end{itemize}

A subset $\cuco Y\subset\cuco X$ is \emph{hierarchically quasiconvex} if it has quasiconvex projections to 
the various hyperbolic spaces, and coarsely contains all realization points for tuples whose $U$--coordinate lies in 
$\pi_U(\cuco Y)$ for all $U\in\mathfrak S$.  More precisely:

\begin{defn}[Hierarchical quasiconvexity, Definition 5.1 of~\cite{hhs2}]\label{defn:hierarchical_quasiconvexity}
Let $(\cuco X,\mathfrak S)$ be a hierarchically hyperbolic space.  
Then $\cuco Y\subseteq\cuco X$ is \emph{$k$--hierarchically quasiconvex}, for some $k\co[0,\infty)\to[0,\infty)$, if the following hold:
\begin{enumerate}
 \item For all $U\in\mathfrak S$, the projection $\pi_U(\cuco Y)$ is 
 a $k(0)$--quasiconvex subspace of the $\delta$--hyperbolic space $\fontact U$.
 \item For all $\kappa\geq0$ and $\kappa$-consistent tuples $\tup
 b\in\prod_{U\in\mathfrak S}2^{\fontact U}$ with $b_U\subseteq\pi_U(\cuco Y)$ for all $U\in\mathfrak S$, each point
 $x\in\cuco X$ for which $\dist_{
 U}(\pi_U(x),b_U)\leq\theta_e(\kappa)$ (where $\theta_e(\kappa)$ is as in Theorem~\ref{thm:realization}) satisfies $\dist(x,\cuco Y)\leq
 k(\kappa)$.
\end{enumerate}
\end{defn}

As one might expect, hulls of arbitrary sets are hierarchically quasiconvex, although in this paper we mainly consider 
hulls of finite sets:

\begin{prop}\cite[Lemma 6.2]{hhs2}\label{prop:coarse_retract}
There exists $\theta_0$ so that for each $\theta\geq \theta_0$ there
exists $\kappa\co\reals_+\to\reals_+$ so that for each $A\subset 
\cuco X$ the set $H_\theta(A)$ is
$\kappa$--hierarchically quasiconvex.
\end{prop}

\begin{rem}Whenever we are working with a fixed HHS $(\cuco
X,\mathfrak S)$, the notation $\theta_0$ will refer
to the constant from Proposition~\ref{prop:coarse_retract}, and we fix
once and for all a constant $\theta\geq\theta_0$.\end{rem}

\subsubsection{The gate map to a hierarchically quasiconvex subspace, and the bridge lemma}
We now recall a construction from Section~5 of~\cite{hhs2}, namely the \emph{gate map} to a hierarchically quasiconvex 
subspace, and prove some additional facts about it.   (The terminology is inspired by the similarity with the notion of a 
gate map to a convex subspace of a median space; see Section~\ref{subsec:median}.)  

Fix a hierarchically hyperbolic space 
$(\cuco X,\mathfrak S)$.  

Let $A\subset\cuco X$ be $\kappa$--hierarchically quasiconvex.  
Recall, this implies that for each $U\in\mathfrak S$, the set $\pi_U(A)$ is
$\kappa(0)$--quasiconvex in $\fontact U$ and there is thus a coarse
closest-point projection $p_{U,A} \co \fontact U \to\pi_U(A)$.  Define a 
\emph{gate 
map} $\gate_A \co \cuco X\to A$ as follows: given $x\in\cuco X$, for each
$U\in\mathfrak S$ let $b_U=p_{U,A}(x)$.  In~\cite[Section
5]{hhs2} we show that the tuple $(b_U)_{U\in\mathfrak S}$ is uniformly
(depending on $\kappa(0)$) consistent, so
Theorem~\ref{thm:realization} and hierarchical quasiconvexity of $A$
produce a coarsely unique point $\gate_A(x)\in A$ such that
$\pi_U(\gate_A(x))$ uniformly coarsely coincides with $b_U$ for all
$U\in\mathfrak S$.

Intuitively, the gate map $\gate_A$ takes $x$ to some realization point for the tuple whose $U$--coordinate, 
for each $U$, is a closest point to $\pi_U(x)$ in the quasiconvex subspace $\pi_U(A)$.

The following lemma, Lemma~\ref{lem:parallel_gates} (``the bridge lemma''), contains a lot of information about the gates of 
a 
hierarchically quasiconvex sets $A,B$. It essentially describes a ``bridge'' of the form $\gate_A(B)\times 
H_{\theta}(\{a,b\})$, for suitable $a\in A,b\in B$, that connects the two. An efficient way to go from $a'\in 
A$ to $b'\in B$ is to start at $a'$, get to the bridge, cross it, and then go to $b'$.

The lemma collects more information than we will need in this paper, for future reference. The proof can be 
safely skipped on first reading.  Before we state it, we give some intuition coming from CAT(0) cube complexes:

\begin{remi}
The bridge lemma is well-illustrated by an analogy to CAT(0) cube 
complexes, where the notion was introduced by Behrstock--Charney 
\cite{BehrstockCharney}.  In the 
analogy, let $P,Q$ be convex subcomplexes of a CAT(0) cube complex $\cuco M$, and let $\gate_P,\gate_Q:\cuco 
M\to P,Q$ be cubical closest-point projection; on the $0$--skeleton, these are the usual gate maps in the median space 
sense.  (So, a hyperplane separates $\gate_P(x)$ from $\gate_P(y)$ if 
and only if it crosses $P$ and separates $x,y$.)  Then $\gate_P(Q)$ is a convex subcomplex of $P$ crossed by 
exactly those hyperplanes that cross $P$ and $Q$, and $\gate_P(Q)$ is a convex subcomplex of $Q$ crossed by 
the same hyperplanes.  The convex hull of $\gate_P(Q)\cup\gate_Q(P)$ is crossed by the above hyperplanes, 
together with the hyperplanes that separate $P$ from $Q$.  Hyperplanes of the latter type cross hyperplanes of 
the former type, and so the convex hull decomposes as a product, which one can view as a ``bridge'' between 
$P$ and $Q$, in the sense that combinatorial geodesics from $P$ to $Q$ travel through the bridge.  

Lemma~\ref{lem:parallel_gates} is analogous, 
except we have replaced the CAT(0) cube complex with an HHS, replaced convexity with hierarchical 
quasiconvexity, and replaced the cubical closest-point projection with the gate map.

Lemma~\ref{lem:parallel_gates} will be important later on in the paper.  We use it in 
Section~\ref{sec:quasiflats_and_cones} to study boxes in asymptotic cones of an HHS; we use it in 
Section~4 to study coarse intersections between standard orthants, the key point being that if $A,B$ are 
hierarchically quasiconvex, then the ``coarse intersection'' of $A$ and $B$ coarsely coincides with 
$\gate_A(B)$. Also, this lemma is useful for obtaining simplifications of the 
distance formula in various instances, see for instance 
Corollary~\ref{cor:gate_in_product} where we obtain a formula for the 
distance between a point and a product region. We note that another 
inspiration for this lemma is its analogue in the mapping class 
group, as developed in \cite[Section~3]{BKMM}.
\end{remi}

 \begin{figure}[h]
 \begin{overpic}[width=0.35\textwidth]{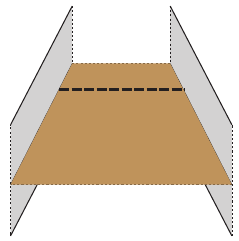}
 \put(22,77){$A$}
 \put(88,65){$B$}
 \put(6,40){\rotatebox{61}{$\gate_A(B)$}}
 \put(80,55){\rotatebox{-59}{$\gate_B(A)$}}
 \put(20,62){$a$}
 \put(77,62){$b$}
 \put(35,66){$H_\theta(\{a,b\})$}
 \end{overpic}
 \caption{The bridge between quasi-convex sets $A$ and $B$}\label{fig:bridge}
 \end{figure}

\begin{lem}[Bridge lemma]\label{lem:parallel_gates}
 For every $\kappa:[0,\infty)\to[0,\infty)$ and all $K_0\geq 10\kappa(0)E$, the following holds.   There exists a function $\kappa'$ and constants 
$K_1=K_1(\kappa,E,K_0)$ and $K_2=K_2(\kappa,E,K_0)$ and $K_3=K_3(\kappa,E,K_1)$ such that for all $\kappa$--hierarchically 
quasiconvex sets $A,B$, we have:
 \begin{enumerate}
  \item\label{item:gate_hq} $\gate_{A}(B)$ is
  $\kappa'$--hierarchically quasi-convex.  
  
  \item\label{item:double_gate_surjective}  The composition $\gate_A\circ\gate_B|_{\gate_A(B)}$ is bounded distance from the identity $\gate_A(B)\to\gate_A(B)$.
  
  \item\label{item:bridge}
  For any $a\in \gate_A(B),b=\gate_B(a)$, we have a $(K_1,K_1)$--quasi-isometric
  embedding $f\co\gate_A(B)\times H_\theta(\{a,b\})\to\cuco X$ with
  image $H_\theta(\gate_A(B)\cup\gate_B(A))$, so that 
  $f(\gate_A(B)\times\{b\})$ $K_1$--coarsely coincides with $\gate_B(A)$.
  \end{enumerate}
  
  Let $\mathcal H=\{U\in\mathfrak S: \diam(\gate_A(B))> K_0\}$.  Let $a,b,f$ be as above.
  \begin{enumerate}\setcounter{enumi}{3}
  \item \label{item:df_bridge_0}For each $p,q\in \gate_A(B)$ and $t\in H_\theta(\{a,b\})$, we have 
$\relevant_{K_2}(f(p,t),f(q,t))\subseteq \mathcal H$.
  \item\label{item:df_bridge}  For each $p\in \gate_A(B)$ and $t_1,t_2\in H_\theta(\{a,b\})$, we have 
$\relevant_{K_2}(f(p,t_1),f(p,t_2))\subseteq \mathcal H^\perp$.
  \item\label{item:take_the_bridge} For each $p\in A,q\in B$ we have $$\dist(p,q)\asymp_{K_3,K_3} 
\dist(p,\gate_A(B))+\dist(q,\gate_B(A))+\dist(A,B)+\dist(\gate_{\gate_B(A)}(p),\gate_{\gate_B(A)}(q)).$$
 \end{enumerate}
\end{lem}

The reader is referred to Figure~\ref{fig:bridge} for a heuristic picture of the content of the 
lemma.

\begin{proof}[Proof of Lemma~\ref{lem:parallel_gates}]
We start with a definition and an observation.

\textbf{The sets $\mathcal V,\mathcal H$:} Let $\mathcal V$ be the set of $V\in\mathfrak S$ with
$\dist_V(A,B)\geq 100E\kappa(0)$.  Fix $K_0\geq10E\kappa(0)$ and let $\mathcal H_{K_0}$ be the set of $H\in\mathfrak S$ with 
$\dist_H(a,a')>K_0$ for some $a,a'\in\gate_A(B)$, say $a=\gate_A(b),a'=\gate_A(b')$ for some $b,b'\in B$.  

The following claim can be proved using standard thin quadrilateral arguments in the hyperbolic space $\fontact 
V$ for each $V\in\mathcal V$:

\begin{claim}\label{claim:bounded_V}
$\pi_V(\gate_A(B))$
and $\pi_V(\gate_B(A))$ have diameter $\leq 10E\kappa(0)$ for
$V\in\mathcal V$. 

For $U\in\mathfrak S-\mathcal V$ and $x\in\gate_A(B)$, $\dist_U(x,\gate_B(x))\leq 10E\kappa(0)$.
\end{claim}

The next auxiliary claim is a sufficient condition for orthogonality between $H\in\mathcal H_{K_0}$ and $V\in\mathcal V$:

\begin{claim}\label{lem:four_points}
Let $C\geq E$ and let $a,b,a',b'\in\cuco X$ and suppose that $H,V\in\mathfrak S$ satisfy 
\begin{itemize}
 \item $\dist_V(a,a'),\dist_V(b,b')\leq C$;
 \item $\dist_V(a,b)>10C$;
 \item $\dist_H(a,b),\dist_H(a',b')\leq C$;
 \item $\dist_H(a,a')>10C$;
\end{itemize}
Then $H\orth V$.
\end{claim}

\begin{proof}[Proof of Claim~\ref{lem:four_points}]
	To establish that $H\orth V$ we must show that 
	$H$ and $V$ are  
	not related by either the transversality or the nesting 
	relation. Our proof is by contradiction.
	
Suppose $V\transverse H$. First, assume that we are in the case that 
$\dist_V(a,\rho^H_V)\leq E$. We then have that 
$\dist_V(\rho^H_V,b)>8C$ and thus $\dist_V(\rho^H_V,b')>6C$. 
Then, by consistency $\rho^V_H$ lies
$E$--close to both $\pi_H(b),\pi_H(b')$, which is impossible since
$d_H(b,b')>6C$.  It remains to consider the case where 
$\dist_V(a,\rho^H_V)>E$. Here, by consistency, we have that 
$\dist_H(a,\rho^V_H)\leq E$.  Hence $\dist_H(a',\rho^V_H)\geq5E$, and 
so, by consistency, we have $\dist_V(a',\rho^H_V)\leq E$. This case 
now reduces to the first case, with $a'$ replacing $a$, again 
yielding a contradiction.

Suppose $V\propnest H$.  Since $\dist_H(a,a')>10C$ and
$\dist_H(b,b')>6C$, at least one of the pairs $a,b$ or $a',b'$ has the property that geodesics in $\fontact H$ connecting the 
corresponding projection points
are $E$--far from $\rho^V_H$.  By the bounded geodesic
image axiom, we have, say, $\dist_V(a,b)\leq E$, a contradiction.  
The same argument rules out $H\propnest V$.  

Since we have ruled out nesting and transversality, we thus have $H\orth V$.
\end{proof}

The preceding two claims imply that $V\orth H$ for all $V\in\mathcal V$ and $H\in\mathcal H_{K_0}$.  We now proceed to the 
proofs of the enumerated assertions.\\

 \textbf{Assertion~\eqref{item:gate_hq} and Assertion~\eqref{item:double_gate_surjective}:} First we claim that $\pi_U(\gate_A(B))$ is uniformly quasiconvex for all $U\in\mathfrak S$.  Observe that $\pi_U(\gate_A(B))$ uniformly coarsely coincides with $p_{U,A}(\pi_U(B))$.  On the other hand, (uniform) quasiconvexity of $\pi_U(B)$ and a thin quadrilateral argument show that $p_{U,A}(\pi_U(B))$ is uniformly quasiconvex, as required.

We now verify that $\gate_A(B)$ satisfies the second part of the definition of hierarchical quasiconvexity. To that end, let $(t_U)_{U\in\mathfrak S}$ be a consistent tuple so that 
$t_U=p_{U,A}(b_U)$ for some $b_{U}\in \pi_{U}(B)$ for each $U\in\mathfrak S$.
Theorem~\ref{thm:realization} and hierarchical quasiconvexity of $A$
provide a realization point $x\in A$ for $(t_U)$.  

To complete the proof of hierarchical quasiconvexity, we must show that in fact $x$
lies uniformly close to $\gate_A(B)$.  Let $y=\gate_A(\gate_B(x))$.  Since $y\in\gate_A(B)$,  it suffices to show that $x$ 
and $y$ are uniformly close.  To do so, we show that $\pi_U(x),\pi_U(y)$ are uniformly close for each $U\in\mathfrak S$, but 
this follows by considering the two possibilities for $U$ covered by Claim~\ref{claim:bounded_V}.  This proves 
Assertion~\eqref{item:gate_hq}.

For $b\in B$, Claim~\ref{claim:bounded_V} can be applied as above to show that $\pi_U(\gate_A(\gate_B(\gate_A(b))))$ uniformly coarsely
coincides with $\pi_U(\gate_A(b))$ for each $U\in\mathfrak S$, and
hence $\gate_A(\gate_B(\gate_A(b)))$ uniformly coarsely coincides with
$\gate_A(b)$ for all $b\in B$, thus proving 
Assertion~\eqref{item:double_gate_surjective}.\\

\textbf{Defining $f$:} Fix $a\in\gate_A(B)$.  Choose $b''\in B$ so that $a=\gate_A(b'')$, and
let $b=\gate_B(a)$.  Note that
$100E\kappa(0)\leq\dist_V(a,b) \leq \dist_V(A,B)+20E\kappa(0)$ 
for
$V\in\mathcal V$; the second inequality here follows from Claim~\ref{claim:bounded_V}.  
Since $a\in A$ and $b\in B$ we also have 
$\dist_V(A,B)\leq\dist_V(a,b)$.  

Let $a'\in\gate_A(B)$.  Assertion~\ref{item:double_gate_surjective} implies that, up to uniformly bounded distance, 
$a'=\gate_A(b')$ for some $b'\in\gate_B(A)$.  For each $U\in\mathfrak S-\mathcal V$, set $b_U=\pi_U(a')$. 
For each $V\in\mathcal V$, let $\gamma_V$ be a geodesic from $\pi_V(a)$ to $\pi_V(b)$ and, for a 
fixed $h\in H_\theta(\{a,b\})$, set $b_V=\pi_V(h)$, which lies $\theta$--close to $\gamma_V$.

\begin{claim}\label{claim:consistent}
For each $a',h$ as above, the associated tuple $(b_W)_{W\in\mathfrak S}$ defined above is $20K_0$--consistent.
\end{claim}

\begin{proof}[Proof of Claim~\ref{claim:consistent}]
If $W,W'\in\mathfrak S-\mathcal V$, or if $W,W'\in\mathcal V$, then $b_W,b_{W'}$ satisfy any consistency inequality involving $W,W'$, since $b_W,b_{W'}$ coincide with the projections to $\fontact W,\fontact W'$ of a common point in those cases.

If $W\in\mathfrak S-\mathcal V$ and $V\in\mathcal V$, then either
\begin{itemize}
 \item $W\in\mathcal H_{K_0}$, or
 \item  $\diam_W(\pi_W(\gate_A(B)))\leq K_0$ and
$\dist_W(a,b)\leq 100E\kappa(0)$.
\end{itemize}

 In the first case, $V\orth W$ by Claim~\ref{lem:four_points}, so
there is no consistency inequality to check.

In the second case, if
$W\notorth V$, then a $200E\kappa(0)$--consistency inequality 
holds, as we now show. 
Indeed, if $W\transverse V$, then $\pi_W(a'),\pi_W(b')$ coarsely coincide,  as do $\pi_V(a),\pi_V(a')$ and 
$\pi_V(b),\pi_V(b')$.  At least one of $\pi_V(a')$ or $\pi_V(b')$ is $E$--far from $\rho^W_V$, so either $\pi_W(a')$ or 
$\pi_W(b')$ is uniformly close to $\rho^V_W$, but these two points coarsely coincide, so $\pi_W(a')=b_W$ is uniformly 
close to $\rho^V_W$.  The nested cases are similar.
\end{proof}

\textbf{Assertion~\eqref{item:bridge}:}
Given the consistent tuple provided by Claim~\ref{claim:consistent}, 
the realization theorem, Theorem~\ref{thm:realization}, then provides a coarsely
unique $x\in\cuco X$ realizing $(b_W)$, and we let $f(a',h)=x$.  This gives a map $f\co\gate_A(B)\times
H_\theta(a,b)\to\cuco X$, and one can see using the distance formula that there exists $K_1=K_1(\kappa,E)$ so that $f$ is a 
$(K_1,K_1)$--quasi-isometric embedding.  
In the next claims, we check that $f$ satisfies the remaining 
properties of Assertion~\eqref{item:bridge}.\\

\begin{claim}\label{claim:coarse_contained_1}
$f(\gate_A(B)\times H_\theta(\{a,b\}))$ is coarsely contained in $H_\theta(\gate_A(B)\cup\gate_B(A))$.
\end{claim}

\begin{proof}[Proof of Claim~\ref{claim:coarse_contained_1}]
Let $h\in H_\theta(\{a,b\})$.  Let $c\in B$ and let $x=f(\gate_A(c),h)$.  Let $U\in\mathfrak S$.  If $U\in\mathcal V$, then 
$\pi_U(x)$ uniformly coarsely coincides with $\pi_U(h)$, which in turn lies $\theta$--close to the geodesic $\gamma_U$ in 
$\fontact U$ from $\pi_U(a)$ to $\pi_U(b)$, by the definition of a $\theta$--hull.  

If $U\in\mathfrak S-\mathcal V$, then $\pi_U(x)$ lies uniformly close to $\pi_U(\gate_A(c))$.  In either case, $\pi_U(x)$  
lies uniformly close to a geodesic starting and ending in $\pi_U(\gate_A(B)\cup\gate_B(A))$, so $x$ lies uniformly close to 
$H_\theta(\gate_A(B)\cup\gate_B(A))$.
\end{proof}

\begin{claim}\label{claim:coarsely_contained_2}
$H_\theta(\gate_A(B)\cup\gate_B(A))$ is coarsely contained in the image of $f$.
\end{claim}

\begin{proof}[Proof of Claim~\ref{claim:coarsely_contained_2}]
Suppose that $x\in H_\theta(\gate_A(B)\cup\gate_B(A))$.  Let
$y=f(\gate_{\gate_A(B)}(x),\gate_{H_\theta(\{a,b\})}(x))$.  We claim that
$\pi_U(y)$ coarsely coincides with $\pi_U(x)$ for all $U\in\mathfrak
S$, and hence $x$ coarsely coincides with $y$.  Indeed, suppose that
$U\in\mathcal V$.  By Claim~\ref{claim:bounded_V}, we have that 
$\pi_U(\gate_A(B)),\pi_U(\gate_B(A))$ are
uniformly bounded; thus 
$\pi_U(H_\theta(\gate_A(B)\cup\gate_B(A)))$ coarsely coincides with
$\pi_U(H_\theta(\{a,b\}))$.  Hence, since $x\in
H_\theta(\gate_A(B)\cup\gate_B(A))$, we have $\pi_U(x)$ coarsely
coincides with $\pi_U(\gate_{H_\theta(\{a,b\})}(x))$.  By definition, this
coarsely coincides with $\pi_U(y)$.

Suppose that $U\in\mathfrak S-\mathcal V$.  Then $\pi_U(\gate_A(B))$ coarsely coincides with $\pi_U(\gate_B(A))$ and hence $\pi_U(H_\theta(\gate_A(B)\cup\gate_B(A)))$ coarsely coincides with $\pi_U(\gate_A(B))$.  Hence, since $x\in H_\theta(\gate_A(B)\cup\gate_B(A))$, we have $\pi_U(x)$ coarsely coincides with $\pi_U(\gate_{\gate_A(B)}(x))$, which coarsely coincides with $\pi_U(y)$ by definition.  
\end{proof}

\begin{claim}\label{claim:hull_last}
$\gate_B(A)$ coarsely coincides with $f(\gate_A(B)\times\{b\})$.
\end{claim}

\begin{proof}[Proof of Claim~\ref{claim:hull_last}]
By Claim~\ref{claim:coarsely_contained_2}, $\gate_B(A)$ is coarsely contained in the image of $f$.   Moreover, if 
$x\in\gate_B(A)$, then $\pi_V(x)$ coarsely coincides with $\pi_V(b)$ for all $V\in\mathcal V$, since $b\in\gate_B(A)$ and 
$\pi_V(\gate_B(A))$ is bounded by Claim~\ref{claim:bounded_V}.  Hence $\gate_B(A)$ is coarsely contained in 
$f(\gate_A(B)\times\{b\})$.  

Conversely, for any $a'\in\gate_A(B)$, $f(a',b)$ coarsely coincides with $\gate_B(a')$.   Indeed, for $V\in\mathcal V$, 
$\pi_V(f(a',b))$ coarsely coincides with $\pi_V(b)$ by definition.  But $\pi_V(b)\in\pi_V(\gate_B(A))$, by the choice of $b$. 
 Since $\pi_V(\gate_B(A))$ is uniformly bounded, $\pi_V(\gate_B(a'))$ coarsely coincides with $\pi_V(b)$ and hence 
$\pi_V(f(a',b))$.  

Let $H\in\mathfrak S-\mathcal V$.  Since $\dist_H(A,B)\leq100E\kappa(0)$, we have that $\pi_V(\gate_B(a'))$ coarsely 
coincides with $\pi_V(a')$.    By definition $\pi_V(f(a',b))$ coarsely coincides with $\pi_V(a')$.  Hence 
$f(\gate_A(B)\times\{b\})$ is coarsely contained in $\gate_B(A)$.
\end{proof}

\textbf{Assertions~\eqref{item:df_bridge_0},\eqref{item:df_bridge}:}  Let $p,q\in \gate_A(B)$ and $t_1,t_2\in 
H_\theta(\{a,b\})$.   Then there exists $K_2$, depending on $\kappa,K_1,E$ such that the following hold by the construction 
of $f$.  First, if $H\in\relevant_{K_2}(f(p,t_1),f(q,t_1))$, then $H\in\mathcal H_{K_0}$.  

Second, if $V\in\relevant_{K_2}(f(p,t_1),f(p,t_2))$, then $V\in\mathcal V$, so $V\in\mathcal H^\perp_{K_0}$ by 
Lemma~\ref{lem:four_points}, as explained above.\\

\textbf{Assertion~\eqref{item:take_the_bridge}:} Let
$F=H_\theta(\gate_A(B)\cup\gate_B(A))$, and consider $p\in A$ and 
$q\in B$. Assertion~\eqref{item:bridge} and Lemma~\ref{lem:distance_between_sets} provides $K_4=K_4(\kappa,\cuco X)$ so that
$$\dist(\gate_F(p),\gate_F(q))\asymp_{K_4,K_4}\dist(A,B)+\dist(\gate_{\gate_B(A)}(p),\gate_{\gate_B(A)}(q)),$$
so it suffices to compare $\dist(p,q)$ with
$\dist(p,\gate_F(p))+\dist(\gate_F(p),\gate_F(q))+\dist(q,\gate_F(q))$.
The upper bound is just the triangle inequality.  For $U\in\mathfrak
S$, examining a thin quadrilateral shows
\begin{eqnarray*}
\dist_U(p,q)&\geq&\dist_U(p,p_{U,F}(\pi_U(p)))+\dist_U(p_{U,F}(\pi_U(p)),p_{U,F}(\pi_U(q)))+\dist_U(q,p_{U,F}(\pi_U(q)))-T\\
&\geq&\dist_U(p,\gate_F(p))+\dist_U(\gate_F(p),\gate_F(q))+\dist_U(q,\gate_F(q))-10T
\end{eqnarray*}
for some uniform $T$.  Given $L\geq0$, let
$\sigma_L(p,q)=\sum_{U\in\mathfrak S}\ignore{\dist_U(p,q)}{L}$.

By the distance formula (Theorem~\ref{thm:distance_formula}),
$\dist(p,q)\geq K_3^{-1}\sigma_{10T}(p,q)-K_3$ for some $K_3$.  Since, 
$10 \sigma_{10T}(p,q)
\geq \sigma_{100T}(p,\gate_F(p))+\sigma_{100T}(\gate_F(p),\gate_F(q))+\sigma_{100T}(\gate_F(p),q)$,
the claim follows from another use of the distance formula (on the
right, with threshold $100T$).
\end{proof}

The next lemma is used in the proof of the final assertion of Lemma~\ref{lem:parallel_gates}, but it is also interesting in its 
own right, since it says in particular that $\gate_A(a)$ is the ``coarsely closest point'' of the hierarchically quasiconvex 
set $A$ to the (arbitrary) point $a\in\cuco X$.

\begin{lem}\label{lem:distance_between_sets}
Let $A,B\subset\cuco X$ be $\kappa$--hierarchically quasiconvex sets.
Then there exists $K=K(\kappa,\cuco X,\mathfrak S)$ so that 
for all
$a\in\cuco X$ we have $\dist(a,B)\asymp_{K,K}\dist(a,\gate_B(a)).$ 
Moreover, for any $a\in A$:
$$\dist(A,B)\asymp_{K,K}\dist(\gate_B(a),\gate_A(\gate_B(a))).$$ 
\end{lem}

\begin{proof}
First let $a\in\cuco X$ and $b\in B$.  Recall that for $U\in\mathfrak
S$, the map $p_{U,B}\co\fontact U\to\pi_U(B)$ is coarsely the
closest-point projection.  For any $U\in\mathfrak S$, we have
$\dist_U(a,p_{U,B}(\pi_U(a)))\leq \dist_U(a,b)+1$.  By the definition
of the gate, and the distance formula, we thus have $K'$, depending on
$\kappa$, so that $\dist(a,\gate_B(a))\leq K'\dist(a,b)+K'$.  Since
this holds for any $b\in B$, this proves the first assertion.

Now let $a\in A$ and let $U\in\mathfrak S$.  Then
$p_{U,A}(p_{U,B}(\pi_U(a)))$ lies uniformly close to any $\fontact
U$--geodesic from $\pi_U(a)$ to $p_{U,B}(\pi_U(a))$, so by the
distance formula and the definition of the gate,
$\dist(a,\gate_B(a))\geq \dist(\gate_B(a),\gate_A(\gate_B(a)))/K'-K'$
for $K'$ depending only on $\cuco X,\mathfrak S$, and $\kappa$.

Choose $a\in A$ so that $\dist(A,B)\geq \dist(a,B)-1$.  Then
$\dist(A,B)\geq K'\dist(a,g_B(a))/K'-K'-1$, by the first 
assertion and the choice of $a$.  As above,
$\dist(a,\gate_B(a))\geq\dist(\gate_B(a),\gate_A(\gate_B(a)))/K'-K'$.
Combining these facts shows that, up to uniform constants,
$\dist(A,B)$ is bounded below by
$\dist(\gate_B(a),\gate_A(\gate_B(a)))$, as required.
\end{proof}

Although we will not use it in the rest of the paper, we note the following interesting corollary, which is useful elsewhere:

\begin{cor}\label{cor:gate_in_product}
Let $(\cuco X,\mathfrak S)$ be an HHS.  Then for all sufficiently large $s$, there exists $K$ such that the following 
holds.  Let $U\in\mathfrak S$.  Let $P_U$ be a corresponding standard product region and let $x\in\cuco X$.  Let $\mathcal 
R$ be the set of $V\in\mathfrak S$ such that $U\propnest V$ or $U\transverse V$ and $\dist_V(x,\rho^U_V)>s$.  Then 
$$\dist(x,P_U)\asymp_{K,K}\sum_{V\in\mathcal R}\dist_V(x,\rho^U_V).$$ 
\end{cor}

\begin{proof}
By construction, $P_U$ is $\kappa$--hierarchically quasiconvex, where $\kappa$ depends only on $E$.  
Lemma~\ref{lem:distance_between_sets} provides $K_0$ such that $\dist(x,P_U)\asymp_{K_0}\dist(x,\gate_{P_U}(x))$.  Now, by 
the definition of $\gate_{P_U}$, the projections $\pi_V(x)$ and $\pi_V(\gate_{P_U}(x))$ uniformly coarsely coincide unless 
$U\propnest V$ or $U\transverse V$.  In the latter case, $\gate_{P_U}(x)$ projects uniformly close to $\rho^U_V$, by 
Definition~\ref{defn:nested_partial_tuple} and Definition~\ref{defn:orthogonal_partial_tuple}.  The claim now follows from 
the distance formula, Theorem~\ref{thm:distance_formula}.     
\end{proof}

\subsection{Wallspaces}
\emph{Wallspaces} were introduced by
Haglund--Paulin~\cite{HaglundPaulin} and then further developed by 
Hruska--Wise in~\cite{HruskaWise}; there are now numerous variants of
the notion.  Here, we recall the 
relevant 
definitions for Section~\ref{sec:hull_level}.  See,
e.g., \cite{HruskaWise} for more background on CAT(0) cube complexes.

\begin{defn}[Wallspace, coherent orientation]\label{defn:wallspace}
A \emph{wallspace} $(\mathcal S,\mathcal W)$ consists of a set $\mathcal S$ and a collection $\mathcal W=\{(\OL W,\OR 
W)\}$ of partitions of $\mathcal S$; each such partition is called a \emph{wall}.  The subsets $\OL W,\OR 
W\subset\mathcal S$ are the \emph{halfspaces associated to} $(\OL W,\OR W)$.  A \emph{orientation} $x$ of $\mathcal 
W$ is a map $\mathcal W\ni (\OL W,\OR W)\mapsto x(\OL W,\OR W)\in\{\OL W,\OR W\}$.  The orientation $x$ is 
\emph{coherent} if $x(\OL W,\OR W)\cap x(\OL W',\OR W')\neq\emptyset$ for all $(\OL W',\OR W'),(\OL W,\OR 
W)\in\mathcal W$.  The orientation $x$ is \emph{canonical} if there exists $s\in\mathcal S$ so that $s\in x(\OL 
W',\OR W')$ for all but finitely many $(\OL W',\OR W')\in\mathcal W$.  When $\mathcal W$ is finite, as it will always be the case in this 
paper, any orientation is canonical.
\end{defn}

\begin{defn}[Dual cube complex]\label{defn:dual cube complex}
The \emph{dual cube complex} $C=C(\mathcal S,\mathcal W)$ associated to the wallspace $(\mathcal S,\mathcal W)$ is 
the CAT(0) cube complex whose $0$--cubes are the coherent, canonical orientations of $\mathcal W$, with two 
$0$--cubes joined by a $1$--cube if the corresponding orientations differ on exactly one wall.  The resulting graph 
is  median, as was proven independently by Chatterji-Niblo~\cite{ChatterjiNiblo} and Nica~\cite{Nica}, building on 
work of Sageev~\cite{Sageev}. Thus this graph is the $1$--skeleton of a uniquely determined CAT(0) cube 
complex, by a theorem of Chepoi~\cite{Chepoi}; we call this complex $C$.  Note that, given a CAT(0) cube complex $C$, each 
hyperplane $W$ yields a 
wall in $C^{(0)}$ by partitioning $C^{(0)}$ into the vertex sets of the two components of $C-W$.  The CAT(0) cube complex 
dual to the resulting wallspace is exactly $C$.
\end{defn}

\begin{defn}[Hyperplane, crossing]
A \emph{hyperplane} in $C$ is a connected subspace whose intersection with each cube $c=[-1,1]^n$ is either $\emptyset$ or a subspace obtained by restricting exactly one coordinate to $0$.   
\end{defn}

The hyperplanes in $C(\mathcal S,\mathcal W)$ correspond bijectively to the walls in $\mathcal W$.  Moreover, two hyperplanes have nonempty intersection if and only if the corresponding walls \emph{cross} in the sense that all four possible intersections of associated halfspaces are nonempty.  It follows that the dimension of $C$ is equal to the largest cardinality of a subset of $\mathcal W$ consisting of pairwise-crossing walls.

We occasionally use the \emph{convex hull} of a set $A\subset C(\mathcal S,\mathcal W)$: this is the largest subcomplex contained in the intersection of all halfspaces containing $A$.

Finally, we need the notion of a \emph{cubical orthant}.  Let $C$ be a CAT(0) cube complex.  Let $n\geq1$ and let $\mathbf 
R$ be the standard tiling of $[0,\infty)$ by $1$--cubes.  A \emph{cubical $n$--orthant} is a copy of the CAT(0) cube complex 
$\mathbf R^n$ with the obvious product cubical structure.  A cubical $n$--orthant \emph{in $C$} is a subcomplex $O$ of $C$ 
that is isomorphic to $\mathbf R^n$ and has the property that $O\hookrightarrow C$ is an isometric embedding (and in 
particular a median homomorphism) on the $0$--skeleton.

\subsection{Ultralimits and asymptotic cones}
We now recall the definitions of ultralimits and asymptotic cones of metric spaces.  A more detailed discussion can be 
found, for example, in the book~\cite{DrutuKapovich} or in~\cite{Drutu}; we recall just the notions we need.

Let $(M, d)$ be a metric space and let $\omega\subset 2^{\naturals}$ be
a non-principal ultrafilter on $\naturals$.  Given a sequence
$m=(m_n\in M)_{n\in\naturals}$ of \emph{observation points} and a positive
sequence $s=(s_n)_{n\in\naturals}$ with
$s_n\stackrel{n}{\longrightarrow}\infty$,
the \emph{asymptotic cone} $\mathbf M$ is the ultralimit of the
based metric spaces
$\lim_{\omega}(M,m_n,\frac{d}{s_n})$: define a
pseudometric $\mathbf d$ on $\prod_n M$ by $\mathbf d(y,z)=\lim_{\omega}\frac{d(y_n,z_n)}{s_n}$,
and consider the induced pseudometric on the component containing $m$, i.e.,
\[\widehat M=\left\{(y_n)_{n\in\naturals}\in\prod_n(M,\frac{d}{d_n})\,:\,\mathbf
d(y,m)<\infty\right\}.\]
Then $\mathbf M$ is the associated quotient metric space,
obtained from $\widehat M$ by identifying points $y$ and $z$ for which $\mathbf
d(y,z)=0$.  

More generally, given a sequence $(M_n,d_n)$ of metric spaces, with a basepoint $m_n\in M_n$ for each $n$, we 
define the ultralimit as follows.  Given $x=(x_n),y=(y_n)\in\prod_nM_n$, let $\mathbf d(x,y)=\lim_\omega 
d_n(x_n,y_n)$.  We identify $(x_n),(y_n)$ when $\mathbf d(x,y)=0$, and restrict ourselves to 
points $(x_n)$ for which $\lim_\omega d(x_n,m_n)<\infty$.  The resulting based space is the ultralimit 
$\lim_\omega (M_n,d_n,m_n)$.  Note that the asymptotic cone $\mathbf M$ defined above is just 
$\lim_\omega(M_n,d_n/s_n,m_n)$.  When taking ultralimits of a sequence of spaces without rescaling, we will 
emphasize this by saying ``non-rescaled''.

We will adopt the following notational conventions for asymptotic 
cones.  We let $\omega$ denote a 
non-principal ultrafilter on $\naturals$, fixed once and for all.  Given a sequence $(M_i)_{i\in\naturals}$ of 
based metric spaces, we denote by $\mathbf M$ the corresponding ultralimit.  Given $m\in\mathbf M$, a 
representative of $m$ is a sequence $(m_i\in M_i)_{i\in\naturals}$, and, when there is no possibility of 
confusion, we use a boldface letter to denote this representative, viz. $\mathbf m=(m_i)$.  

We also denote by $\,^\omega\reals_+$ the ultrapower of the set
$\reals_+$ of nonnegative reals.  Given $\lambda\in\,^\omega\reals_+$,
we sometimes use the notation, e.g., $\mathbf r$ to denote a sequence
$(r_m)_{m\in\naturals}$ representing $\lambda$.

\subsection{Median, coarse median, quasimedian}\label{subsec:median}
We recall some background on median spaces and coarse median spaces.  The latter were introduced by 
Bowditch~\cite{Bowditch:coarse_median} and we refer the reader to~\cite{Bowditch:coarse_median,Bowditch:rigid} for
a more detailed discussion of both concepts.

The discussion of coarse median spaces in
\cite{Bowditch:coarse_median} is given in terms of \emph{(finite)
median algebras}.  For concreteness, we first consider only the
following example of a (finite) median algebra: let $\cuco Y$ be a
CAT(0) cube complex (with finitely many $0$--cubes).  Recall 
 that there exists a \emph{median} map $\mu\co (\cuco
Y^{(0)})^3\to\cuco Y^{(0)}$ with the property that, for all
$x_1,x_2,x_3\in\cuco Y^{(0)}$, the $0$--cube $\mu(x_1,x_2,x_3)$ lies
on a combinatorial geodesic from $x_i$ to $x_j$ for all distinct
$i,j\in\{1,2,3\}$, see e.g., \cite{Chepoi}.  This $0$--cube with the 
given property is unique.

\begin{rem}[Median and walls]\label{rem:walls_and_media}
Let $\cuco Y$ be a CAT(0) cube complex and let $x,y,z$ be $0$--cubes.
The \emph{median}, $\mu=\mu(x,y,z)$, can be described in terms of orientations
of walls as follows.  If $W$ is a wall in $\cuco Y$ so that some
associated halfspace $W^+$ contains $x,y,z$, then $\mu$ orients $W$
toward $W^+$.  Otherwise, $W$ has two associated halfspaces $W^\pm$ so
that $W^+$ contains exactly two of the points $\{x,y,z\}$ and $W^-$
contains exactly one of these points.  Then $\mu$ orients $W$ toward
$W^+$.  This choice of orientation of all walls is coherent and easily
verified to yield a $0$--cube which is the median of $x,y,z$.
\end{rem}

The above discussion provides the basis for the definition of a coarse median space.

\begin{defn}[Coarse median space; \cite{Bowditch:coarse_median}]\label{defn:coarse_median}
Let $(\cuco L,\dist)$ be a metric space and let $\mu\co \cuco
L^3\to\cuco L$ be a ternary operation.  We say that $\cuco L$,
equipped with $\mu$, is a \emph{coarse median space} if there exists a
constant $k$ and a map $h\co\naturals\to[0,\infty)$ so that the
following hold:
\begin{itemize}
 \item For all $x,y,z,x',y',z'\in\cuco L$, $$\dist(\mu(x,y,z),\mu(x',y',z'))\leq k(\dist(x,x')+\dist(y,y')+\dist(z,z'))+h(0).$$
 \item For all $p\in\naturals$ and $A\subseteq\cuco L$ with $|A|\leq 
 p$, there is a CAT(0) cube complex $\cuco Y_A$ with finite 
 $0$--skeleton and median map $\mu_A$, and maps $f\co A\to\cuco 
 Y_A^{(0)}$ and $g\co\cuco Y_A^{(0)}\to A$ so that the following hold: 
 \begin{itemize}
  \item $\dist(\mu(g(x),g(y),g(z)),g(\mu_A(x,y,z)))\leq h(p)$ for all $x,y,z\in\cuco Y_A^{(0)}$;
  \item $\dist(a,g(f(a)))\leq h(p)$ for all $a\in A$.
 \end{itemize}
\end{itemize}
The \emph{coarse median rank} $\nu$ of $\cuco L$ is the smallest integer $\nu$ so that $\cuco Y_A$ can be taken to have dimension $\leq\nu$ for all finite $A$.
\end{defn}

It was shown in~\cite{hhs2} that every hierarchically hyperbolic space
is a coarse median space; we refer the reader there for details of the
construction. The property of coarse medians we need in this paper is 
that, given an HHS $(\cuco X,\mathfrak S)$, there exists a
constant $\ell$, depending only on the HHS constant $E$, so that the following holds.
Given $x,y,z\in\cuco X$ and letting $m\in\cuco X$ be their coarse 
median, then for all $U\in\mathfrak S$, the point $\pi_U(m)$ lies
$\ell$--close to any geodesic in $\fontact U$ joining $a,b$, where
$a,b\in\{\pi_U(x),\pi_U(y),\pi_U(z)\}$ are distinct.

\begin{defn}[Quasimedian map]\label{defn:quasimedian}
Let $\cuco Y$ be a CAT(0) cube complex with median map $\mu_{\cuco Y}$
on its $0$--skeleton.  Let $(\cuco L,\mu,\dist)$ be a coarse median
space.  Let $h\geq0$.  An \emph{$h$--quasimedian map} is a map
$q\co\cuco Y\to\cuco L$ for which
$$\dist(\mu(q(x),q(y),q(z)),q(\mu_{\cuco Y}(x,y,z)))\leq h$$ for all
$x,y,z\in\cuco Y$.
\end{defn}

Note that quasimedian maps are referred to by 
\cite{Bowditch:coarse_median} as ``quasimorphisms,'' but we use a 
different terminology 
to avoid any confusion with other uses of that word.

When studying asymptotic cones of HHSs, it isn't sufficient to restrict oneself to finite median algebras/CAT(0) cube 
complexes.  So, we need a few more standard notions about general median algebras and median metric spaces.

Recall that a set $\mathcal M$ equipped with a ternary
operation $\mu\co\mathcal M^3\to\mathcal M$ is a \emph{median algebra}
if for all finite $A\subset\mathcal M$, there is a finite
$B\subset\mathcal M$ so that $A\subseteq B$, and $B$ is closed under
$\mu$, and $(B,\mu)$ is a finite median algebra in the above sense
(i.e., we can identify its elements with points in a finite CAT(0) cube
complex in such a way that $\mu$ coincides with the cubical median).
The \emph{rank} of a median algebra is defined as in
Definition~\ref{defn:coarse_median} in terms of the dimensions of the
cube complexes approximating finite sets.

Given $a,b\in \mathcal M$, the \emph{interval} $[a,b]$ is the set of
$c\in \mathcal M$ with $\mu(a,b,c)=c$, and $\mathcal N\subset\mathcal
M$ is \emph{median convex} if $[a,b]\subseteq\mathcal N$ whenever
$a,b\in\mathcal N$.

If $\mathcal M$ is also a Hausdorff topological space, and $\mu$ is continuous, then $(\mathcal M,\mu)$ is a \emph{topological median algebra}.  We consider the following special case.  Let $(M,\dist)$ be a metric space.  For any $a,b\in M$, let $[a,b]$ be the set of $c\in M$ for which $\dist(a,b)=\dist(a,c)+\dist(c,b)$.  If $M$ has the property that for all $a,b,c\in M$, the intersection $[a,b]\cap[b,c]\cap[c,a]$ consists of a single point $\mu(a,b,c)$, then the map $(a,b,c)\mapsto\mu(a,b,c)$ makes $(M,\dist)$ a topological median algebra.  In this situation, we say $M$ is a \emph{median (metric) space}.  The metric notion of an interval agrees with the median notion discussed above.

Bowditch showed, in \cite[Theorem~2.3]{Bowditch:coarse_median}, that any asymptotic cone of a coarse median 
space of rank $\nu$ is a topological median algebra of rank $\nu$, where the median of points represented by 
sequences $(x_n),(y_n),(z_n)$ is represented by a sequence whose $n^{th}$ term is the coarse median of 
$x_n,y_n,z_n$.  Moreover, Bowditch showed in \cite[Proposition 2.4]{Bowditch:rigid} (see also Theorem 6.9 of 
the same paper) that any asymptotic cone of a coarse median space is bilipschitz homeomorphic to a metric 
median space, where the median is as just described.  When we work with asymptotic cones of HHSs (recall that 
each HHS is coarse median of finite rank), we will only be interested in the bilipschitz homeomorphism class, 
and will therefore assume that the asymptotic cone, with the given median, is a median metric space.

We collect the above in the following proposition, which plays an important role throughout 
Section~\ref{sec:quasiflats_and_cones}:

\begin{prop}[Asymptotic cones of HHS are median metric spaces]\label{prop:cone_median}
Let $(\cuco X,\mathfrak S)$ be a hierarchically hyperbolic space.  Let $\mathbf{\cuco X}$ be an asymptotic 
cone of $\cuco X$.  Let $\mu:\cuco X^3\to\cuco X$ be the coarse median map, and let 
$\boldsymbol{\mu}:\mathbf{\cuco X}^3\to\mathbf{\cuco X}$ be the map sending $\mathbf x=(x_n),\mathbf 
y=(y_n),\mathbf z=(z_n)$ to the point represented by $(\mu(x_n,y_n,z_n))$.  Then:
\begin{itemize}
 \item $\boldsymbol{\mu}$ makes $\mathbf{\cuco X}$ into a topological median space of finite rank.  If $(\cuco 
X,\mathfrak S)$ is asymphoric and has rank $\nu$, then the median space $\mathbf{\cuco X}$ has rank at 
most $\nu$.
\item $\mathbf{\cuco X}$, equipped with the median $\boldsymbol{\mu}$, is bilipschitz equivalent to a median 
metric space.
\end{itemize}
\end{prop}

\begin{proof}
It is shown in~\cite[Section 7]{hhs2} that $\cuco X$ is a coarse
median space.  From Theorem~\ref{thm:cubulated_hulls} 
it follows that the rank of $\cuco X$ as a
coarse median space is bounded above by the maximal cardinality $m$ of
collections $\{U_i\}\subset\mathfrak S$ of pairwise orthogonal
elements.  (So, in general, the coarse median rank of $\cuco X$ is
bounded only by the complexity of $\mathfrak S$.)

The bound on the coarse median rank in the asymphoric case is
Corollary~\ref{cor:rank_coarse_median} below.
 
The first assertion now follows 
from \cite[Theorem~2.3]{Bowditch:coarse_median} which provides the 
median structure and implies that $\nu$ is an upper bound on the rank of the median space $(\seq{\cuco 
X},\boldsymbol{\mu})$.  

The second assertion follows from \cite[Theorem~6.9]{Bowditch:rigid}.  More specifically, being an 
asymptotic cone of a quasigeodesic space, $\seq{\cuco X}$ is a 
complete geodesic space (see e.g., \cite[Proposition 10.70]{DrutuKapovich}).  The proof of 
\cite[Proposition 9.1]{Bowditch:coarse_median} shows that $\seq{\cuco X}$, with the given median, 
satisfies the hypotheses of \cite[Proposition~2.4]{Bowditch:rigid}, which then yields the second assertion. 
\end{proof}

Finally, we conclude with some background about the notion of gate maps in a median space, and the notion of a block; these 
again are vital in Section~\ref{sec:quasiflats_and_cones}.  

\begin{defn}[Block, median gate]\label{defn:block_gate}
Let $(M,\mu,\dist)$ be a median metric space.  A
\emph{$n$--block} in $M$ is a median convex subspace isometric to the  
product of $n$ nontrivial compact intervals in $\reals$, endowed with the $\ell_1$
metric.

Recall that the (median) interval in $M$ between points $m$ and $n$ is the set $[m,n]$ of all $m'$ such that 
$\mu(m,m',n)=m'$.

If $N\subset M$ is a closed median convex subset, a \emph{median gate map} 
$\gate_N\co M\to N$ is a map such that $\gate_N(m)\in[m,n]$ for all $m\in M,n\in N$.  

Closed convex subsets of a complete median space always admit a unique gate map (see e.g.~\cite[Lemma 6.26]{DrutuKapovich}). 
 If $N,N'$ are  median convex, then $\gate_N(N')$ is again median convex; see~\cite{Bowditch:rigid}.
\end{defn}

\subsection{Identifying hierarchy paths}\label{subsec:quasimedian_hierarchy_ray}
We now prove a sufficient condition for a path in the HHS $(\cuco X,\mathfrak S)$ to be a hierarchy ray.  It is 
straightforward, but it will play a role in Section~\ref{sec:main}.

In the lemma, ``quasimedian'' will mean with respect to the coarse median 
on $\cuco X$ and the usual median on $\reals$, i.e., $\gamma:\reals\to \cuco X$ is quasimedian if whenever 
$r,s,t\in\reals$ satisfy $r<s<t$, then the coarse median of $\gamma(r),\gamma(s),\gamma(t)$ is $\lambda$--close to 
$\gamma(s)$.

\begin{lem}\label{lem:quasimedian_hierarchy_ray}
Let $(\cuco X,\mathfrak S)$ be an HHS.  Then for all $\lambda\geq1$, there exists $D=D(\lambda)$ such that the following 
holds.  Let $I\subset\reals$ be a (possibly unbounded) subinterval 
and let $\gamma\colon I\to\cuco X$ be a $\lambda$--quasimedian 
$(\lambda,\lambda)$--quasi-isometric embedding.  Then $\gamma$ is a $(D,D)$--hierarchy path.     
\end{lem}

\begin{proof}
The path $\gamma$ is a $(\lambda,\lambda)$--quasigeodesic by 
hypothesis, so to show that it is a hierarchy path we only
need to prove that there exists a constant $D$ so that, for each $U\in\mathfrak S$, the composition of $\gamma$ with
$\pi_U$ is an unparameterized $(D,D)$--quasigeodesic
in $\fontact U$. In order to do so, it suffices to show that there exists a constant $D'$ so that for each $r,s,t\in I$ with $r<s<t$, we have that $\pi_U(\gamma(s))$ lies $D'$--close to a geodesic from $\pi_U(\gamma(r))$ to $\pi_U(\gamma(t))$.

Let $r,s,t\in I$ satisfy $r<s<t$.  Let $m$ be the coarse median of
$\gamma(r),\gamma(s),\gamma(t)$.  Since $\pi_U$ is $E$--coarsely
Lipschitz and $\gamma$ is $\lambda$--quasimedian, we have
$\dist_U(\pi_U(\gamma(s)),\pi_U(m))\leq E\lambda+E$.  By the
definition of the coarse median, there exists 
$\lambda'=\lambda'(E,\lambda)$ such that
$\dist_U(\gamma(s),m_U)\leq\lambda'$, where $m_U$ is the coarse median
in the hyperbolic space $\fontact U$ of the three points 
$\pi_U(\gamma(r)),\pi_U(\gamma(s)),\pi_U(\gamma(t))$. The distance from $m_U$ to any geodesic $[\pi_U(\gamma(r)),\pi_U(\gamma(t))]$ is bounded in terms of the hyperbolicity constant of $\fontact U$, so we are done.
\end{proof}

\section{Cubulation of hulls}\label{sec:hull_level}
Fix a hierarchically hyperbolic space $(\cuco X,\mathfrak S)$.  In
this section, we prove that the hull of any finite set $A\subset \cuco
X$ can be cubulated, i.e., approximated by a finite CAT(0) cube
complex in such a way that both distances and (coarse) medians are
coarsely preserved.

We achieve the cubulation of $H_\theta(A)$ by constructing finitely many walls in $H_\theta(A)$ and then passing to the dual 
cube complex, using the work of Chatterji--Niblo, Nica, and Sageev mentioned in Definition~\ref{defn:dual cube complex}.  

In the case where $(\cuco X,\mathfrak S)$ is a rank-one HHS --- which, as we will see below, is equivalent to being hyperbolic 
--- the cubulation of the hull of $A$ reduces to a classical fact about hyperbolic spaces: the coarse convex hull of any 
finite collection of points can be approximated by a quasi-isometrically embedded $1$--dimensional CAT(0) cube 
complex, i.e., a tree~\cite{Gromov:hyperbolic}.    

We exploit this special case to build walls in $H_\theta(A)$, roughly as follows.  We consider $U\in\frak S$ and consider 
a tree which approximates the coarse convex hull of $\pi_{U}(A)$ in the hyperbolic space $\contact U$.  We then find an
appropriate separated net in this tree and, for each point in this 
net, we use $\pi_{U}^{-1}$ of a connected component of the 
complement as one of our walls.  

The fact that the quality of the tree approximation in $\fontact U$ depends on $|A|$ is the 
most obvious way in which the dependence of the quality of our cubical approximation on $|A|$ makes itself felt. However, there 
are also several other essentially different ways in which $|A|$ influences the quality of the approximation. First, it does so in 
our choice of separated nets (roughly, the larger the total branching of a tree is, the harder it is to approximate the tree with a separated net), and the other two are in 
Lemma~\ref{lem:diambv_bounded} and Lemma~\ref{lem:same_walls}.

We now turn to the formal statement of the cubulation of hulls theorem (which is Theorem~\ref{thm:cubulated_hulls} of the 
introduction):

 \begin{thm}[Cubulation of hulls]\label{thm:cubulated_hull}
  Let $(\cuco X,\mathfrak S)$ be a hierarchically hyperbolic space and let $k\in\naturals$.  Then there exists $M_0$ so that for all $M\geq M_0$ there is a constant $C_1$ so that for any $A\subset\cuco X$ of cardinality $\leq k$, there is a $C_1$--quasimedian
  $(C_1,C_1)$--quasi-isometry $\mathfrak p_A\co\cuco Y\to H_\theta(A)$, where $\cuco Y$ is a CAT(0) cube complex.
  
  Moreover, let $\mathcal U$ be the set of $U\in\mathfrak S$ so that $\dist_U(x,y)\geq M$ for some $x,y\in A$.  Then $\dimension\cuco Y$ is equal to the maximum cardinality of a set of pairwise-orthogonal elements of $\mathcal U$.
  
  Finally, there exist $0$--cubes $y_1,\ldots,y_{k'}\in\cuco Y$ so that $k'\leq k$ and $\cuco Y$ is equal to the convex hull in $\cuco Y$ of $\{y_1,\ldots,y_{k'}\}$.
 \end{thm}
 
 \begin{remi}
Since we posted an earlier version of this paper, Bowditch has given
a new proof of this theorem under somewhat
more general hypotheses (very similar to, but strictly weaker than, the HHS axioms); see Theorem 1.3 
in~\cite{Bowditch:hulls}.
 \end{remi}

The proof is carried out over the next several subsections.  We fix once and for all $(\cuco X,\mathfrak S)$, some $k\in\naturals$, and a subset $A=\{x_1,\dots,x_k\}\subseteq\cuco X$. 

\subsection{The candidate finite CAT(0) cube complex}\label{subsec:cubulated_hull}
Fix $U\in\mathfrak S$.  For each $x_j\in A$, recall that $\pi_U(x_j)$ is a subset of the $\delta$--hyperbolic 
space $\fontact U$ of diameter at most $E$; for each $j$, choose $\ell_j^U\in\pi_U(x_j)$, to obtain $k$ points 
$\ell_1^U,\ldots,\ell_k^U\in\fontact U$.  As shown by Gromov, there exists $C=C(k,\delta)$ so 
that there is a finite tree $T_U$ and an embedding $T_U\hookrightarrow\fontact U$, sending edges to geodesics 
of $\fontact U$, such that:
\begin{itemize}
 \item $\dist_U(p,q)\leq\dist_{T_U}(p,q)\leq\dist_U(p,q)+C$ for all $p,q\in T_U$;
 \item $\ell_j^U\in T_U$ for $1\leq j\leq k$;
 \item each leaf of $T_U$ lies in $\{\ell_1^U,\ldots,\ell_k^U\}$.
\end{itemize}
This is the usual spanning tree of a finite subset of a hyperbolic
space; see~\cite{Gromov:hyperbolic}.  The given properties of $T_U$ ensure that,
up to increasing $C$ uniformly, $\dist_{Haus}(T_U,\hull_{\fontact
U}(\pi_U(A)))\leq C$.

Our choice of $T_U$ ensures that, for each $x_j\in A$, every leaf of $T_U$ is contained in $\pi_U(x_j)$ for 
some $x_j\in A$ and each $\pi_U(x_j)$ contains a point of $T_U$.

Let $M$ be a (large) constant to be specified below.  We will point
out the conditions that $M$ must satisfy as we proceed.  

Let $\mathcal
U$ be the set of all $U\in\mathfrak S$ with
$\diam(\pi_U(A))\geq 100Mk$.

Let $\mathcal U_1\subseteq \mathcal U$ be
the set of $\nest$--minimal elements of $\mathcal U$. 
Given $\mathcal U_{n-1}$, let $\mathcal U_n\subseteq \mathcal U$ be
 the set of all $\nest$--minimal elements of $\mathcal U- \mathcal
 U_{n-1}$.  Finite complexity ensures that there is some $s$ so that
 $\bigcup_{n=1}^s\mathcal U_s=\mathcal U$.  For each

 $U\in\mathcal U$, 
 let $\mathcal U^{\nest, U}=\{V\in\mathcal
 U:V\propnest U\}$. 
 For each $V\in \mathcal U^{\nest, U}$, choose
  $r^V_U\in T_U$ closest to $\rho^V_U$; the set of choices is bounded 
  diameter (moreover, in Lemma
  \ref{lem:rho_tree_co-nest}, we prove that $r^V_U$ is $100EC$--close to
  $\rho^V_U$).

\begin{remi}[Finiteness of $\mathcal U$]
For sufficiently large $M$ (in terms of the threshold constant in the distance formula), Theorem~\ref{thm:distance_formula} 
implies that $\mathcal U$ is finite.  One could also deduce this from the large link axiom (Definition 1.1.(7) 
in~\cite{hhs2}), avoiding use of the distance formula.  Finiteness of $\mathcal U$ is used below; it will ensure that the 
wallspace we construct has finitely many walls, so that the $0$--cubes of $\cuco Y$ correspond bijectively to coherent 
orientations of the walls (recall Definition~\ref{defn:dual cube complex}); in other words, we don't have to worry about the 
``canonical orientation'' condition from Definition~\ref{defn:dual cube complex}, because we will be dealing with a finite 
wallspace.
\end{remi}

We now proceed to the construction of the walls.
  
Starting with each $U\in\mathcal U_1$ and then repeating for 
$\mathcal U_2$ up to $\mathcal U_s$, we choose a finite set of 
elements $p_i^U\in T_U$ satisfying the following 
conditions (which implies that the $p_i^U$ together with the $r^{V}_{U}$ 
provide a $10Mk$--net which is $M$--separated):
\begin{enumerate}
 \item $\dist_U(p_i^U,x_j)\geq M$,
 \item $\dist_U(p_i^U,p_j^U)\geq M$,
 \item $\dist_U(p_i^U,r^V_U)\geq M$ for each $V\in  \mathcal U^{\nest,
 U}$ (when $U\in \mathcal U_1$, there are no such $V$), and  
 \item each component of $T_U- \Big(\{p_i^U\}\cup \{r^V_U\}_{V\in
  \mathcal U^{\nest, U}})\Big)$ has diameter at most $10Mk$ (when $U\in \mathcal U_1$, there are no such $V$, 
so 
 the criterion is only about complements of the $\{p_i^U\}$).
 \end{enumerate}
 
 The existence of such a net is justified as follows.  Fix $U\in\mathcal U$.  For each $x_j\in A$, choose 
$y_j\in T_U$ lying in $\pi_U(x_j)$.

For each $j\leq k$, let $T_U^j$ be the subtree of $T_U$ spanned by $y_1,\ldots,y_j$.  Consider the geodesic 
$T_U^2$.  Let $a_1,\ldots,a_\ell$ be the points $a_s$ on $T_U^2$ such that there is a 
(possibly trivial) geodesic in $T_U$ that intersects $T_U^2$ at $a_s$ and joins $a_s$ to a point in 
$\{y_3,\ldots,y_k\}$.  Note that $\ell\leq k-2$.

Note that $T_U$ is the union of $T_U^2$ along with $\ell$ subtrees $C_s$, each of which intersects $T_U^2$ at 
a point $a_s,\ s\leq\ell$.  Choose a (possibly empty) $M$--separated set of points $p^U_i$ in $T_U^2$ so that 
each $p^U_i$ is $M$--far from each $a_s$, and $M$--far from $y_1,y_2$, and $M$--far from each 
$\rho^V_U,V\in\mathcal U^{\nest,U}$ belonging to $T_U^2$.  Any collection that is maximal with these properties 
has the property that $\{p^U_i\}\cup\{\rho^V_U\in T_U^2\}$ is an $M(\ell+2)$--net in 
$T_U^2$.   If $k=2$, then $\{a_s\}=\emptyset$ and we are done.

Otherwise, each tree $C_s$ contains at most $k-2$ of the points $y_j$, and exactly one of the points 
$a_1,\ldots,a_\ell$, namely $a_s$.  So, by induction, $C_s$ contains an $M$--separated collection of points 
$\{p^U_i(s)\}_i$ that are $M$ far from any $\rho^V_U\in C_s$, and $M$--far from any $y_j\in C_s$, and $M$--far 
from $a_s$, such that $\{p^U_i(s)\}_i\cup \{r^V_U\in C_s\}$ is an $M(k-1)$--net in $C_s$.  Observe that the 
set of $p^U_i$, together with the union over $s$ of the $\{p^U_i(s)\}_i$, has the properties listed above.

(Since the points $p_i^U$ are $M$--far from each $y_j$, they are $(M-E)$--far from $\pi_U(x_j)$, and so we 
rename $M-E$ to $M$ to see that the first property on the list holds for $\pi_U(x_j)$, not just $y_j$.  With 
the renamed constant, we now have an $(M+E)k$--net, and in particular a $10Mk$--net, provided $M\geq E/9$.  We 
assume this just to simplify computations later.) 
 
 \begin{defn}[Walls in $H_\theta(A)$]\label{defn:walls_in_hull}
 Given $U\in\mathcal U$ and $\{p_i^U\}$ as above, for each $i$ we define a partition $H_{\theta}(A)=\OL W^U_i\sqcup \OR W^U_i$ of $H_\theta(A)$ as follows. 
 Choose a component $T'_U$ of $T_U-\{p_i^U\}$ and let $\OL W^U_i=\beta_U^{-1}(T'_U)\cap H_{\theta}(A)$, and set $\OR W^U_i=H_\theta(A)-(\OL W^U_i)$.  Let $\mathcal L^U_i=(\OL W^U_i,\OR W^U_i)$.
 \end{defn}
 
 Observe that the (finite) set of walls in $H_{\theta}(A)$ specified in Definition~\ref{defn:walls_in_hull} 
depends on our choice of $M$ (since that determines $\mathcal U$) and on our choice of the $p_i^U$ (which is 
also constrained by the choice of $M$ and the number of points $x_j$).  Let $\cuco Y$ be the CAT(0) cube 
complex dual to the wallspace just defined.  Since the set of walls is finite, there is exactly one $0$--cube 
in $\cuco Y$ for each coherent orientation of all the walls (recall that a coherent orientation is a choice of 
halfspace for each wall such that, for any two walls, the chosen halfspaces have nonempty intersection).

The cubes in $\cuco Y$ are closely related to the standard product regions in $\cuco X$.  Specifically, each 
cube corresponds to a collection of pairwise-crossing walls.  Each wall was determined by some $U\in\mathcal 
U$, namely the $U$ for which the halfspaces in the wall are preimages of complementary components of some 
$p_i^U$.  We think of $U$ as labeling the wall.  Now, each edge of $\cuco Y$ is labeled by the same $U$ that 
labels its dual wall.  Below, we will map $\cuco Y$ to $\cuco X$ in such a way that an edge $e$ labeled $U$ 
has the property that its endpoints are sent to points in $\cuco X$ that project uniformly close on $\fontact 
V$ to $\rho^U_V$ whenever $U\propnest V$ or $U\transverse V$.  In this sense, we will be mapping cubes to 
standard product regions.

 \subsection{Lemmas supporting consistency of certain tuples}\label{subsec:consistency_lemmas} 
 The map $\cuco Y\to H_\theta(A)$ will be constructed roughly as follows.  For each $0$--cube $p\in\cuco Y$, we will 
construct a tuple $(b_V)\in\prod_V\fontact V$.  In Lemma~\ref{lem:consistent_walls}, we will verify that this tuple is 
consistent, and this will require the following technical lemmas, which are essentially just applications of consistency 
(Axiom~\ref{axiom:consistency}) and bounded geodesic image (Lemma~\ref{lem:BGI}).

The content of the lemmas is the following.  Given distinct, non-orthogonal $U\in\mathcal U,V\in\mathfrak S$, there are 
three possibilities: we can have $U\transverse V,U\propnest V$, or $V\propnest U$.  In the first two cases, the coarse point 
$\rho^U_V$ lies close to $T_V$ in $\fontact V$.  In the second case, for any $x\in T_U$ far from $\rho^V_U$, the coarse 
point $\rho^U_V(x)$ lies close to $T_V$.

  \begin{lemma}[$\rho^U_V$ close to $T_V$, transverse case]\label{lem:rho_tree}
  For all $M>10E$, the following holds.  Let $U\in\mathcal U$ and $V\in\mathfrak S$. If $U\transverse V$ then $\rho^U_V$ is $E$--close to some $\pi_V(x_i)$, and hence $2E$--close to $T_V$.   
  \end{lemma}
  
  \begin{proof}
  Since $U\in\mathcal U$, we have $\diam_{\fontact U}(\pi_U(A))\geq 100Mk\geq 100M>10^3E$.  Hence we can 
choose $x_i\in A$ so that $\dist_U(x_i,\rho^V_U)>E$.  Consistency yields $\dist_V(x_i,\rho^U_V)\leq E$.  Since 
$\pi_V(x_i)$ has diameter $\leq E$ and contains a point of $T_V$, we have $\dist_V(T_V,\rho^U_V)\leq 2E$.
  \end{proof}
  
  \begin{lemma}[$\rho^U_V$ close to $T_V$, nested case]\label{lem:rho_tree_co-nest}
  For any $M> 10E$, the following holds. Let $U\in\mathcal U,V\in\mathfrak S$, with $U\propnest V$. Then $\dist_V(\rho^U_V,T_V)\leq 100EC$.
  \end{lemma}

  \begin{proof}
   Suppose that $\dist_V(\rho^U_V,T_V)>100EC$.  Then, since $T_V$
   $C$--coarsely coincides with $\hull_{\fontact V}(A)$, and the
   latter is $5E$--quasiconvex, we have that $\rho^U_V$ lies at
   distance greater than $E$ from any geodesic joining points in $\pi_V(A)$.
   Hence, by consistency and bounded geodesic image, any such geodesic
   projects to a geodesic in $\fontact U$ of diameter at most $E$, 
   i.e., $\pi_U(A)$ has diameter bounded by $10E$.  This contradicts
   $U\in\mathcal U$, provided $M>10E$.
  \end{proof}
 
 \begin{lemma}[$\rho^U_V(x)$ close to $T_V$]\label{lem:rho_tree_nest}
 For any $M>10EC$ the following holds.  Consider $U\in\mathcal U$ and
 any $V\in\mathfrak S$ with $V\propnest U$.  Then for each $x\in T_U- \neb_M(\rho^V_U)$
 there exists $x_j\in A$ with $\dist_V(\rho^U_V(x),x_j)\leq 2E$ (in
 particular, $\rho^U_V(x)$ is $10E$--close to $T_V$).
 \end{lemma}

\begin{proof}
 There exists a leaf of $T_U$, contained in $\pi_U(x_j)$ for some
 $x_j\in A$, in the same connected component of $T_U-
 \neb_{M/2}(\rho^V_U)$ as $x$.  Geodesics from $x$ to $\pi_U(x_j)$ thus
 stay $E$--far from $\rho^V_U$, so that the desired conclusion
 follows from bounded geodesic image (and consistency, which says
 $\diam_V(\pi_V(x_j)\cup\rho^U_V(\pi_U(x_j)))\leq E$).
\end{proof}

 \subsection{The proof of Theorem~\ref{thm:cubulated_hull}}\label{subsec:proof_of_cubulated_hull_theorem}
 We now prove Theorem~\ref{thm:cubulated_hull}.  Some auxiliary lemmas appear immediately below the proof, organized according to which part of the proof they support.
 
 \begin{proof}[Proof of Theorem \ref{thm:cubulated_hull}]
 We break the proof into several parts.
 
 \textbf{Definition of $\mathfrak p_A$:}  We first define $\mathfrak 
 p_A\co\cuco Y\to \cuco X$, noting that it suffices to define $\mathfrak p_A$ on the $0$--skeleton of $\cuco Y$.  Let $p\in\cuco Y^{(0)}$; we view $p$ as a coherent orientation of the walls $\mathcal L^U_i$ provided by Definition~\ref{defn:walls_in_hull}.
  
  For $U\in\mathcal U$, $V\in\mathfrak S$ and each $p^U_i$ (which we 
  recall gives a pair $\{\OL
  W^U_i,\OR W^U_i\}$),  we can consider $\overline W_i(U)\in \{\OL
  W^U_i,\OR W^U_i\}$ which is the halfspace given by the orientation $p$,
  namely $p(\overleftarrow W_i(U),\overrightarrow W_i(U))$. We 
  let
  $S_{U,i,V}(p)\subseteq T_V$ be the convex hull in $T_V$ of
  $\beta_V(\overline W_i(U))$, where, as above, $\beta_{V}$ is the composition 
  of projection to $\fontact V$ and the closest point projection to 
  $T_{V}$.

  By the definition of a coherent orientation, for any $U,i,U',i'$, we have $\beta_V(\overline W_i(U))\cap \beta_V(\overline W_{i'}(U'))\neq\emptyset$, whence $S_{U,i,V}(p)\cap S_{U',i',V}(p)\neq\emptyset$.  The Helly property for trees thus ensures that $\bigcap_{U,i} S_{U,i,V}(p)\neq\emptyset$ for each $V\in\mathfrak S$, and we let $b_V=b_V(p)=\bigcap_{U,i} S_{U,i,V}(p)$.
  Lemma~\ref{lem:diambv_bounded}, below, proves that $\diam(b_V)$ are 
  uniformly bounded. Lemma~\ref{lem:consistent_walls}, below, shows the 
  $(b_{V})$ are $\eta$--consistent, where $\eta=\eta(M,k,\cuco X)$.

  We can now define $\mathfrak p_A(p)\in \cuco X$ to be a realization point 
  associated to 
  $(b_U)$ via Theorem~\ref{thm:realization}.  Specifically, there
  exists $\xi=\xi(\eta,E)$ so that for all $U\in\mathfrak S$, we have
  $\dist_U(\pi_U(\mathfrak p_A(p)),b_U)\leq\xi$.
  
  \textbf{The image of $\mathfrak p_A$ coarsely coincides with $H_\theta(A)$:}  Let $x\in
  H_\theta(A)$.
  
  For each wall in $H_\theta(A)$ (the walls are those from Definition~\ref{defn:walls_in_hull}), choose the 
halfspace containing $x$; there is exactly one such halfspace since a wall is, by construction, a partition of 
$H_\theta(A)$ into two halfspaces.  Now, any two of the chosen halfspaces contain $x$, so by 
Definition~\ref{defn:wallspace}, this orientation is coherent, and it is a canonical orientation simply because there are 
only finitely many walls.

So, this orientation of all walls determines a $0$--cube $p\in\cuco Y$, by Definition~\ref{defn:dual cube 
complex}.  Now, by construction, the tuple $(b_V(p))$ has the property that, for all $U\in\mathcal U$, we have 
$\beta_U(x)\in b_U(p)$.  Since $\dist_U(\pi_U(x),\beta_U(x))\leq\theta$, because $x\in H_\theta(A)$, we see 
that $\dist_U(x,b_U(p))\leq\theta$.  Now, $\dist_U(b_U(p),\mathfrak p_A(p))\leq\xi$, so $\dist_U(x,\mathfrak 
p_A(p))\leq \xi+\theta$ for all $U\in\mathcal U$.  If $U\not\in\mathcal U$, then $\dist_U(x,\mathfrak 
p_A(p))\leq \theta+100Mk$.  So, by the uniqueness axiom for HHS (Definition 1.1.(9) in~\cite{hhs2}), or simply 
by Theorem~\ref{thm:distance_formula}, we have  $\dist_{\cuco X}(x,\mathfrak
  p_A(p))\leq C_1'$, where $C_1'=C_1'(M,k,\cuco X,\theta)$.  Hence $H_\theta(A)$ lies in a uniform 
neighborhood of $\image\mathfrak p_A$.

On the other hand, if $p\in\cuco Y$, then $\pi_U(\mathfrak p_A(p))$ lies uniformly close (in terms of $\xi$) 
to $\hull(\pi_U(A))$ for all $U\in\mathfrak S$.  The definition of hierarchical 
quasiconvexity, together with the fact that $H_\theta(A)$ is hierarchically quasiconvex, ensures that $\mathfrak 
p_A(p)$ lies uniformly close to $H_\theta(A)$, i.e., $\image\mathfrak p_A$ lies in
a uniform neighborhood of $H_\theta(A)$.

After enlarging $C'_1$ if necessary, we thus see that there exists $C_1'=C_1'(M,k,\cuco X,\theta)$ such that 
$H_\theta(A)$ and $\image\mathfrak p_A$ lie at Hausdorff distance at most $C_1'$.
  
  \textbf{Distance estimates:} For $p\in\cuco Y$, we say $p^U_i$ is a \emph{separator} for $p$ if
  $p^U_i$ separates $\beta_U(\mathfrak p_A(p))$ from $b_U(p)$.  We call
  $U$ the \emph{support} of the separator. In 
  Lemma~\ref{lem:same_walls} we produce a constant $T=T(M,k,\eta,\xi,\mathfrak S)$, so 
  that for each $p\in \cuco Y$ there are at most $T$ separators for 
  $p$.
  
  We first relate the number of walls separating a pair of points in 
  $\cuco Y$ to the number of points separating their images under 
  $\mathfrak p_{A}$.
  
  Specifically, let $p,q\in\cuco Y$.  By the definition of distance in a CAT(0) cube complex, $\dist_{\cuco 
Y}(p,q)$ is the  number of walls separating $p$ and $q$.  Let $\mathcal L^V_i$ be a wall separating $p$ from 
$q$.  Then, by the construction of the tuples $b_V(p)$ ad $b_V(q)$, the subtrees $b_V(p)$ and $b_V(q)$ of $T_V$ 
lie on opposite sides of the wall in $T_V$ determined by $p^V_i$.  Conversely, if $b_V(p)$ and $b_V(q)$ are 
separated by the partition of $T_V$ determined by some $p^V_i$, then $\mathcal L^V_i$ corresponds to a 
wall in $\cuco Y$ separating $p$ from $q$.

Hence $\dist_{\cuco Y}(p,q)$ is the sum of the
  numbers of $p^V_i$ separating $b_V(p)$ from $b_V(q)$, as $V$ varies.  Now, $\mathcal L^V_i$ separates 
$b_V(p)$ from $b_V(q)$ but fails to separate $\beta_V(\mathfrak p_A(p))$ from $\beta_V(\mathfrak 
  p_A(q))$ only if $\mathcal L^V_i$ is a separator for $p$ or for $q$.  Similarly, $\mathcal L^V_i$ separates 
$\beta_V(\mathfrak p_A(p))$ from $\beta_V(\mathfrak 
  p_A(q))$ but fails to separate $b_V(p),b_V(q)$ only if $\mathcal L^V_i$ is a separator for $p$ or for $q$.
  
  Lemma~\ref{lem:same_walls} shows that $p$ has at most $T$ separators and $q$ has at most $T$ separators.  
Let $Q(p,q)$ be the sum over all $V$ of the number of $p^V_i$ separating $\beta_V(\mathfrak p_A(p))$ from 
$\beta_V(\mathfrak p_A(q))$.  The preceding discussion shows that $|\dist_{\cuco Y}(p,q)-Q(p,q)|\leq 2T$.

  Observe that: if, for some $V$, there
  exist distinct $p^V_i,p^V_{i'}$ separating $\beta_V(\mathfrak
  p_A(p))$ from $\beta_V(\mathfrak p_A(q))$, then $V$ contributes to
  the distance formula sum between $\mathfrak p_A(q)$ and $\mathfrak
  p_A(p)$, at some fixed threshold $L$ chosen in terms of 
$E$ and $M$.  Moreover, $V$ also contributes to the distance formula sum in the case where $\beta_V(\mathfrak
  p_A(p))$ and $\beta_V(\mathfrak p_A(q))$ are both $M/2$--close to $\pi_V(A)$
  and there exists at least one $p^V_i$ separating $\beta_V(\mathfrak
  p_A(p))$ from $\beta_V(\mathfrak p_A(q))$.  
  
  Applying Lemma
  \ref{lem:count_separators} and Lemma \ref{lem:Ramsey}, we have 
  $$\dist_{\cuco X}(\mathfrak p_A(p),\mathfrak p_A(q))\asymp \sum_{U\in\mathfrak S} 
  \ignore{\dist_U(\mathfrak p_A(p),\mathfrak p_A(q))}{L}\geq 
  Q(p,q)-100EC\theta N,$$
  where $N$ is the constant from Lemma~\ref{lem:Ramsey}.  Hence there exists $C_1''=C_1''(M,\cuco 
X,\mathfrak S,k)$ so that $\dist_{\cuco X}(\mathfrak p_A(p),\mathfrak p_A(q))\geq \dist_{\cuco 
Y}(p,q)/C_1''-C_1''$ for $p,q\in\cuco Y$.
  
  \textbf{$\mathfrak p_A$ is coarsely Lipschitz:} Crossing one hyperplane of $\cuco Y$ corresponds to changing
  only one coordinate $(b_U)$ as above by a bounded amount, so there
  exists $C_1'''=C_1'''(M,k,\cuco X)$ so that $\mathfrak p_A$ is
  $(C_1''',C_1''')$--coarsely Lipschitz.  
  
  \textbf{Dimension:}  The assertion about dimension follows from Lemma~\ref{lem:cross} and the well-known fact that any finite set of $n$ pairwise crossing hyperplanes in a CAT(0) cube complex intersect in the barycenter of some $n$--cube.
  
  \textbf{Convex hull:}  For each $x_j\in A$, let $y_j$ be the orientation of the walls in $H_\theta(A)$ 
obtained by choosing, for each wall $(\OL W^U_i,\OR W_i^U)$, the halfspace containing $x_j$.  This orientation 
is coherent by definition, so it determines a $0$--cube of $\cuco Y$, which we also denote $y_j$.  By 
construction, each wall separates two elements of $A$, so every hyperplane of $\cuco Y$ separates two of the 
chosen $0$--cubes $y_i,y_j$.  Thus no intersection of combinatorial halfspaces properly contained in $\cuco Y$ 
contains all of the $y_j$, so $\cuco Y$ is the convex hull in $\cuco Y$ of the set of $y_j$.  
  
  \textbf{Conclusion:}  Lemma~\ref{lem:quasimedian} provides $C_1''''$ so that $\mathfrak p_A$ is $C_1''''$--quasimedian,  so 
the proof is complete once we take $C_1=\max\{C_1',C_1'',C_1''',C_1''''\}$.   
 \end{proof}

 \subsubsection{Lemmas supporting realization}\label{subsubsec:realization}
 
 The two lemmas below are used to construct a point in $\cuco X$ via 
 realization, given the tuple $(b_V(p))=(b_V)$ associated to a $0$--cube $p\in\cuco Y$ (which we fix for the 
purposes of the next two lemmas).  The first lemma shows that $b_V$ is a uniformly bounded set in each 
$\contact V$, and the second verifies that the tuple $(b_V)$ is $\eta$--consistent (and bounds $\eta$).

The realization theorem (Theorem~\ref{thm:realization}) then provides a point $\mathfrak p_A(p)\in\cuco X$ 
that projects $\xi$--close to $b_V$ in each $\fontact V$, where $\xi$ just depends on $E$ and $\eta$.  This is 
how we defined the map $\mathfrak p_A:\cuco Y\to\cuco X$ in the proof of Theorem~\ref{thm:cubulated_hull}.
 
 \begin{lem}\label{lem:diambv_bounded}
 There exists $\tau=\tau(M,k)>0$ (independent of $V$) so that $\diam(b_V(p))\leq\tau$ for all $p\in\cuco Y$.
 \end{lem}
 
 \begin{proof}
 Fix $p\in\cuco Y$ and write $b_V=b_V(p)$.
 
 If $V\in\mathfrak S-\mathcal U$, then $\diam(b_V)\leq\diam(T_V)\leq 100M$.  Hence suppose that $V\in\mathcal 
U$.

There exists $\tau=\tau(M,k)\geq 50Mk(k-2)$ such that the following holds.  Suppose that $x,y\in\cuco X$ 
satisfy $\dist_V(x,y)>\tau$.  Then there exists $\alpha\in \{p_i^V\}_i\cup\{r^W_V\}_{W\in \mathcal U_1\cap 
V^{\nest,\mathcal V}}$ so that $\alpha$ is $10M$--far from $\beta_V(x),\beta_V(y)$ and from all points of 
$T_V$ of valence larger than $2$, and separates $\beta_V(x)$ from $\beta_V(y)$.  Indeed, there are at most 
$k-2$ points of valence larger than $2$, since each leaf of $T_V$ belongs to $\pi_V(A)$ and $|A|=k$.  So the 
geodesic from $\beta_V(x)$ to $\beta_V(y)$ has a sub-segment of length at least $50Mk$ avoiding the points of 
valence more than $2$.  This subsegment contains a point $\alpha$ that necessarily separates $\beta_V(x)$ from 
$\beta_V(y)$ and either lies in $\{p^V_i\}$ or $\{r^W_V\}_{W\in\mathcal V^{\nest,V}}\}$, because such points 
form a $10Mk$--net.  The restriction to $\mathcal U_1$ is justified by the fact that 
for $W'\propnest W\propnest U$, we have that $\rho^{W'}_V$ coarsely coincides with $\rho^W_V$, so we can 
assume each $r^W_V$ as above coincides with $r^{W'}_V$ for some $W'\in\mathcal U_1$ nested in $V$.
 
 Choose any $x,y\in\cuco X$ projecting $M$--close to $b_V$, and suppose by contradiction that $\dist_V(\beta_V(x),\beta_V(y))>\tau$. Let $\alpha$ be as above.
 
 If $\alpha=p_i^V$, then we clearly have a contradiction since $b_V$ is contained in one of the connected 
components of $T_V-\{p^V_i\}$.  If $\alpha=r^W_V$, then we write $A\cup\{x,y\}=A'\sqcup A''$, where we group 
together all elements of $A\cup\{x,y\}$ corresponding to a point of $T_V$ in a given connected component of 
$T_V-\{r^W_V\}$. By bounded geodesic image and the fact that $r^W_V$ is close to $\rho^W_V$ (Lemma 
\ref{lem:rho_tree_co-nest}), $\pi_W(A')$ and $\pi_W(A'')$ are uniformly bounded, so that $T_W$ consists of two 
uniformly bounded sets, respectively containing $\pi_W(A')$ and $\pi_W(A'')$, that are joined by a segment in 
$T_W$ which is a geodesic $\gamma$ of $\fontact W$ containing no  vertex of valence more than $2$.  Moreover, 
this geodesic has $\beta_W(x),\beta_W(y)$ uniformly close to its endpoints.
 
Since $W\in\mathcal U_1$, there exists some $p^W_i$ in $T_W$. Let us show that $S_{W,i,V}(p)$ is far from one 
of $\beta_V(x)$ or $\beta_V(y)$, which is a contradiction. If there is a $p^W_i$ in $T_W$, then since $p^W_i$ 
was chosen far from the leaves of $T_W$, we have that $p_i^W\in\gamma$, lying at distance $M/2$ from 
$\beta_W(x)$ and from $\beta_W(y)$.
 
 Let $\overline{T}$ be one of the two connected components of $T_W-\{p^W_i\}$. Then $\beta_W^{-1}(\overline T)$ cannot contain points $x',y'$ with $\beta_V(x'),\beta_V(y')$ far from $r^W_V$ and in different components of $T_V- \{r^W_V\}$, which is the required property of $S_{W,i,V}(p)$. Indeed, otherwise bounded geodesic image would imply that $x',y'$ project respectively close to $\pi_W(A')$ and $\pi_W(A'')$, thus on opposite sides of $p^W_i$.
 \end{proof}

 \begin{lemma}\label{lem:consistent_walls}
  There exists $\eta=\eta(M,k,\cuco X)$ such that the following holds.  Let $p\in\cuco Y$.  Then the tuple 
$(b_V(p))$ is $\eta$--consistent.
 \end{lemma}
 
 \begin{proof}  
 Let $U\transverse V$. If $U,V\in\mathfrak S-\mathcal U$, we are done because the corresponding coordinates 
$b_U,b_V$ $(100Mk+E)$--coarsely coincide with those of, say, $x_1$. If $U\in\mathcal U$ and $V\in\mathfrak 
S-\mathcal U$, then any point in $T_V$, whence also any point in $b_V(p)$, is $(100Mk+C+2E)$--close to $\rho^U_V$ by Lemma \ref{lem:rho_tree} and the definition of $\mathcal U$, 
so we are done. 

Now suppose that $U,V\in\mathcal U$. 
Let $c_U$ be a point in $T_U$ which is $10E$--close to $\rho^V_U$, and define $c_V$ similarly ($c_U$ and $c_V$ 
are provided by Lemma \ref{lem:rho_tree}). If both $b_U$ and $b_V$ are $100Mk$--far from $\rho^V_U$ and 
$\rho^U_V$ respectively, then there are $S_{W,i,U}(p),S_{W',i',V}(p)$ containing $b_U,b_V$ but far from 
$c_U,c_V$. There cannot be $q\in \cuco X$ with $\beta_U(q)\in S_{W,i,U}(p),\beta_V(q)\in S_{W',i',V}(p)$ by 
consistency, implying that the intersection of the halfspaces chosen from $\mathcal L^W_i, \mathcal 
L^{W'}_{i'}$ is empty. This contradicts the coherence of the orientation defining $p$.

 \begin{figure}[h]
 \begin{overpic}[width=0.5\textwidth]{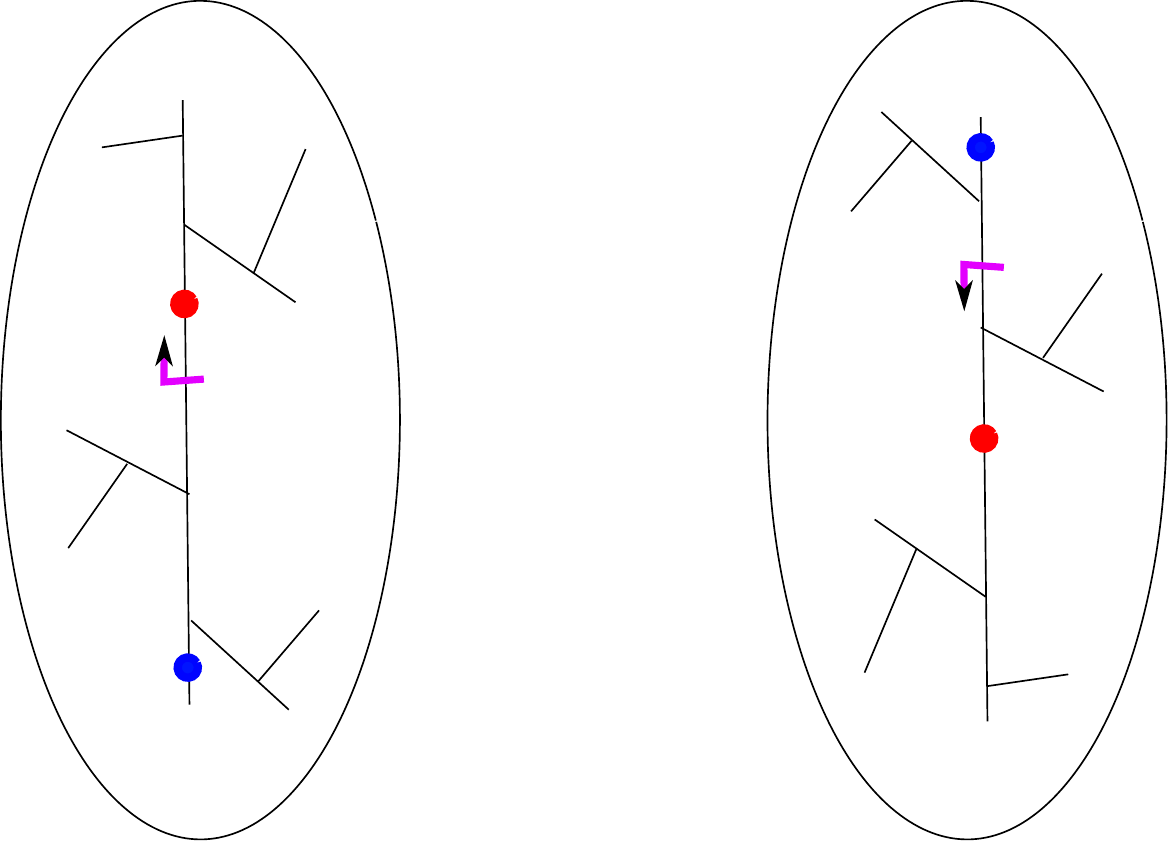}
 \put(29,65){$\fontact U$}
 \put(65,65){$\fontact V$}
 \put(17,45){$b_U$}
 \put(86,33){$b_V$}
 \put(86,60){$\rho^U_V$}
 \put(-14,12){$\rho^V_U$ or $\rho^V_U(b_V)$}
 \put(43,40){$\transverse$ or $\propnest$}
 \end{overpic}
 \caption{Proof of Lemma~\ref{lem:consistent_walls}. $S_{W,i,U}(p),S_{W',i',V}(p)$ are shown as oriented halfspaces in the trees $T_U,T_V$.}\label{fig:consistent_walls}
 \end{figure}

 Let $U\propnest V$. If $V\in\mathfrak S-\mathcal U$, then by Lemma \ref{lem:rho_tree_co-nest} we have that 
$\rho^U_V$ is $100EC$--close to $b_V$. Hence, we can assume $V\in\mathcal U$. If $U\in\mathfrak S-\mathcal U$, 
similarly, the corresponding coordinates $b_U,b_V$ coarsely coincide with those of a point in $H_{\theta}(A)$ 
that projects close to $b_V$ in $\fontact V$.
 
 Finally, suppose $U,V\in\mathcal U$.  The argument is very similar
 to the final argument in the transverse case above.  Let $c_V=r^U_V$
 (which is $10E$--close to $\rho^U_V$ by Lemma
 \ref{lem:rho_tree_co-nest}); and, as given by Lemma
 \ref{lem:rho_tree_nest}, we 
 let $c_U$ be a point in $T_U$ which is 
 $100EC$--close to $\rho^V_U(b_V)$.  If both $b_U$ and $b_V$ are $100Mk$--far
 from the corresponding $\rho$, then there exist
 $S_{W,i,U}(p),S_{W',i',V}(p)$ containing $b_U,b_V$ but far from
 $c_U,c_V$.  By the bounded geodesic image axiom, $\rho^V_U(S_{W',i',V}(p))$
 has uniformly bounded diameter.  Hence, there cannot be $q\in \cuco
 X$ with $\beta_U(q)\in S_{W,i,U}(p),\beta_V(q)\in S_{W',i',V}(p)$ by
 consistency, implying that the intersection of the halfspaces chosen
 from $\mathcal L^W_i, \mathcal L^{W'}_{i'}$ is empty.  This
 contradicts the coherence of the orientation defining $p$.
 \end{proof}
 
 \subsubsection{Lemmas supporting the distance estimate}\label{subsubsec:separator_control} The next three 
lemmas support the distance estimate in the proof of Theorem~\ref{thm:cubulated_hull}.  

The first lemma bounds projection distances from below in terms of the walls; it was used above to give a 
lower bound on $\dist_{\cuco X}(\mathfrak p_A(p),\mathfrak p_A(q))$ in terms of the distance in the cube 
complex $\cuco Y$ between $p$ and $q$.

 \begin{lem}\label{lem:count_separators}
Let $U\in\mathcal U$.  For each $x,y\in H_\theta(A)$, we have $\dist_U(x,y)+50EC\theta\geq |\{i: p^U_i\in 
[\beta_U(x),\beta_U(y)]\}|$. Moreover, if $\pi_U(x),\pi_U(y)$ are both $C$--close to $\pi_U(A)$, then 
$\dist_U(x,y)\geq |\{i: p^U_i\in [\beta_U(x),\beta_U(y)]\}|$.
\end{lem}

\begin{proof}
Let $x,y\in H_\theta(A)$.  Recall that $\diam(\pi_U(x)\cup\beta_U(x))\leq 10(E+C+\theta)$, so $\dist_U(x,y)\geq 
\dist_U(\beta_U(x),\beta_U(y))-20(E+C+\theta)$.  Hence $\dist_U(x,y)\geq 
\dist_{T_U}(\beta_U(x),\beta_U(y))-40EC\theta$.  Therefore, $\dist_U(x,y)\geq |\{i: p^U_i\in 
[\beta_U(x),\beta_U(y)]\}|-40EC\theta-1$, as required.  The ``moreover'' statement follows in a similar way 
using the fact that the $p^i_U$ are $M$--far from leaves of $T_U$.
\end{proof}

The next lemma is a simple application of Ramsey theory and the consistency property of an HHS.  This lemma 
is used in tandem with the one above.  It is also used below to control the number of separators associated to 
$p\in\cuco Y$.  Recall that $p^U_i$ is said to be a separator for $p$ if $p^U_i$ separates $b_U(p)$ from 
$\beta_U(\mathfrak p_A)$ in the tree $T_U$.

 \begin{lemma}\label{lem:Ramsey}
   There exists $N=N(\cuco X)\geq0$ so that for each $x\in H_\theta(A)$ there are at most $N$ elements 
$U\in\mathcal U$ so that $\dist_U(\beta_U(x),\pi_U(A))>100E$.
  \end{lemma}
  
  \begin{proof}
   One axiom of an HHS is that there is a  
   bound, $c$,  on the cardinality of subsets of $\mathfrak S$
   whose elements are pairwise $\nest$--comparable. 
   By \cite[Lemma~2.1]{hhs2}, $c$ also bounds the maximum 
   cardinality of a set of pairwise orthogonal elements.  Given
   $x\in H_\theta(A)$, consider the set of $U\in\mathfrak S$ such that
   $\dist_U(x,A)>100E$.  Ramsey's theorem provides $N$ (the Ramsey
   number $R(c,c)$) for which either there are at most $N$ such $U$, or
   there exist $U_1,U_2$ with $U_1\transverse U_2$ and
   $\dist_{U_l}(x,A)>100E$ for $l=1,2$.  By Lemma~\ref{lem:rho_tree},
   $\rho^{U_1}_{U_2}$ is $10E$--close to an element of
   $\pi_{U_2}(A)$~and 
   thus $90E$--far from $\pi_{U_2}(x)$.  The same
   holds with $U_1$ and $U_2$ reversed, contradicting consistency.
  \end{proof}

  The next lemma bounds the number of separators in terms of $M,k$, and the constants $\xi,\tau$.  Since the proof is 
somewhat technical, we first give a heuristic discussion.  We first show that if $p\in\cuco Y$ has, say, $T'$ separators, 
then there are at least $T'M/\xi$ elements $U\in\mathcal U$ that support separators (this is achieved by bounding the maximal number of separators supported on any given $U\in\mathcal U$).  Lemma~\ref{lem:Ramsey} shows that, for ``most'' 
such $U$, the point $\mathfrak p_A(p)$ projects in $T_U$ close to some $\pi_U(x_j)$.  So, if $T'$ is too large, there is a 
specific pair $x_j,x_k$ such that, in many $U$ as above, $\mathfrak p_A(p)$ projects close to $\pi_U(x_j)$ and far from 
$x_k$.  Lemma~\ref{lem:passing_up} then provides $U_1,U_2$ with these properties, both nested into some $V\in\mathfrak S$, 
such that $\rho^{U_1}_V,\rho^{U_2}_V$ are very far in $\fontact V$ 
(in terms of $M,\xi,\tau$).  Applications of bounded 
geodesic image, consistency, and coherence of the orientation of walls corresponding to the $0$--cube $p$ allow us to 
conclude that $\dist_V(\pi_V(\mathfrak p_A(p)),b_V)>\xi$, which contradicts how the point $\mathfrak p_A(p)$ was chosen, 
namely as a realization point with constant $\xi$.  For the last part of the argument, the reader will find it helpful to 
consult Figure~\ref{fig:same_walls}.
 
 \begin{lem}\label{lem:same_walls}
  There exists $T=T(M,k,\xi,\tau,\cuco X,\mathfrak S)$ such that for any $p\in\cuco Y$ there exist at most
  $T$ separators for~$p$.
 \end{lem}
 
 \begin{proof}
 We fix $p\in\cuco Y$, after which, we can simplify our notation by 
 writing $b_V$ to mean $b_V(p)$.  Recall from 
Lemma~\ref{lem:consistent_walls} that $(b_V)$ is an $\eta$--consistent tuple, where $\eta$ depends on $M$, 
$k$, and the global HHS constants, but is independent of $p$.  Recall that the realization point $\mathfrak 
p_A(p)\in\cuco X$ provided by Theorem~\ref{thm:realization} is 
characterized by the property that 
$\dist_V(b_V,\mathfrak p_A(p))\leq \xi$ for all $V\in\mathfrak S$, where $\xi$ depends on $\eta$ and the 
global HHS constants, but is independent of $p$.

First, for each $V\in\mathcal U$, we bound the number of separators
$p^V_i$.  Since each $p^V_i$ separates $b_V$ from $\beta_V(\mathfrak
p_A(p))$, and the set of $p^V_i$ in $T_V$ is $M$--separated by
construction, there are at most $\xi/M$ separators with support $V$.
Let $\mathrm{Sep}$ be the set of $V\in\mathcal U$ that support a
separator for $p$.

By the previous paragraph, it remains to bound the cardinality $|\mathrm{Sep}|$ of $\mathrm{Sep}$.  Now, 
Lemma~\ref{lem:Ramsey} provides a uniform constant $N$ so that there are at most $N$ elements of $\mathcal U$ 
where $\mathfrak p_A(p)$ projects $100E$--far from every element of $A$.

Suppose that $|\mathrm{Sep}|> N+N_0k(k-1)$, where $N_0=N_0(M,\xi,\cuco X,\mathfrak S)$ will be chosen 
momentarily.  (If the preceding inequality does not hold, then we have bounded $|\mathrm{Sep}|$ 
independently of $p$, as required.)  This lower bound implies that there are at least $N_0k(k-1)$ elements $V\in\mathrm{Sep}$ where 
$\beta_V(\mathfrak p_A(p))$ is $100E$--close to $\pi_V(A)$.  Hence, there exists $x_j\in A$ so that there are 
at least $N_0(k-1)$ elements $V\in\mathrm{Sep}$ where $\beta_V(\mathfrak p_A(p))$ is $100E$--close to 
$\pi_V(x_j)$.

Now, for each such $V$, we have a separator $p^V_i$ separating $\beta_V(\mathfrak p_A(p))$ from $b_V$, and 
necessarily lying $M$--far from $\pi_V(x_j)$, because of how our net in $T_V$ was chosen.  Hence $b_V$ 
separates $\pi_V(x_j)$ from some $\pi_V(x_\ell)$.  Hence there is a pair $x_j,x_\ell$ and at least $N_0$ 
elements $U\in\mathcal U$ such that:
\begin{itemize}
 \item $\beta_U(\mathfrak p_A(p))$ is $100E$--close to $\pi_U(x_j)$;
 \item there exists a separator $p_i^U$ for $p$, with support $U$, separating $\beta_U(x_j)$ from 
$\beta_U(x_\ell)$.
\end{itemize}

Now, let $L=1000(M+\xi+\tau)$.  Suppose we chose $N_0=N_0(L)$, the constant from 
Lemma~\ref{lem:passing_up}.  So, if $|\mathrm{Sep}|> N_0k(k-1)+N$, then Lemma~\ref{lem:passing_up} provides 
some $V\in\mathfrak S$ and two 
elements $U_1,U_2\in\mathrm{Sep}$ with the above two listed properties, such that $U_1\propnest V$, and 
$U_2\propnest V$, and $\dist_V(r^{U_1}_V,r^{U_2}_V)>10E+\xi$.  

For $t=1,2$, there exists $p^{U_t}_{i_t}$ separating $\beta_{U_t}(\mathfrak p_A(p))$ (which is $100E$--close 
to $\pi_{U_t}(x_j)$) from $\pi_{U_t}(x_\ell)$, so $\dist_{U_t}(x_j,x_\ell)>M$.

By bounded geodesic image, the geodesic in $T_V$ from $\beta_V(x_j)$ to $\beta_V(x_k)$ must pass $E$--close to 
$r^{U_1}_V$ and $r^{U_2}_V$.  (So $V$ is necessarily in $\mathcal U$.)

For concreteness, suppose $U_1,U_2$ are labeled so that $\neb_E(\rho^{U_1}_V)$ separates $\pi_V(x_j)$ from 
$\rho^{U_2}_V$ and $\pi_V(x_\ell)$.

Bounded geodesic image and consistency imply that $\beta_V(\mathfrak p_A(p))$ lies $E$--close to the 
connected component of $T_V-\neb_E(r^{U_t}_V)$ containing $\pi_V(x_j)$ for $t=1,2$.  Indeed, this holds 
because $\beta_{U_t}(\mathfrak p_A(p))$ is $(M-100E)$--far in $T_{U_t}$ from $\pi_{U_t}(x_\ell)$.

We now analyze two cases, according to how close $b_V$ lies to the component $\Pi$ of 
$T_V-\neb_E(\rho^{U_2}_V)$ containing $\pi_V(x_k)$.

Recall that $b_V$ is the intersection of various subtrees $S_{W,s,V}$, each of which coarsely coincides with 
the projection to $\fontact V$ of a halfspace belonging to the coherent orientation $p$.  

The first case is where every such 
subtree lies $E$--close to $\Pi$. In this case, the intersection of the subtrees --- which is by definition $b_V$ --- lies 
$10E$--close to $\Pi$.  Hence, by Lemma~\ref{lem:diambv_bounded}, $b_V$ is contained in the 
$(10E+\tau)$--neighborhood of $\Pi$.  This implies that $\dist_V(b_V,\rho^{U_1}_V)\geq L-(10E+\tau)$.  Since 
we saw that $\beta_V(\mathfrak p_A(p))$ is $E$--close to the component of $T_V-\neb_E(r^{U_1}_V)$ containing 
$\pi_V(x_j)$, we have $\dist_V(\mathfrak p_A(p),b_V)>L-(10E+\tau)-E$. By our choice of $L$, this quantity exceeds 
$\xi$, which contradicts  the definition of $\mathfrak p_A(p)$.

So we must be in the second case, where some halfspace $H$ belonging to the coherent orientation $p$ projects 
to a tree $S$ in $T_V$ --- necessarily containing $b_V$ --- that is $E$--far from $\Pi$.  Hence, $S$ and 
$\pi_V(x_j)$ lie in the same component of $T_V-\neb_E(\rho^{U_2}_V)$.
So, by the consistency and bounded geodesic
image axioms, $\pi_{U_2}(H)$ is contained in the $E$--neighborhood in $T_{U_2}$ of $\pi_{U_2}(x_j)$.  But since 
$p^{U_2}_{i_2}$ is $M$--far from $\pi_{U_2}(x_j)$ and separates $\beta_{U_2}(\mathfrak p_A(p))$ (which is 
$100E$--close to $\pi_{U_2}(x_j)$) from $b_{U_2}$, we see that $p^{U_2}_{i_2}$ separates $\pi_{U_2}(H)$ from 
$b_{U_2}$.  But the coherence of the orientation $p$ and the definition of $b_{U_2}$ requires $b_{U_2}$ to be 
contained in $\pi_{U_2}(H)$, which gives a contradiction.

\begin{figure}[h]
\begin{overpic}[width=0.55\textwidth]{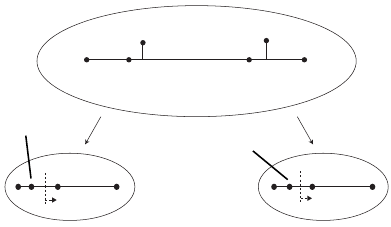}
\put(95,40){$\fontact V$}
\put(36,10){$\fontact U_1$}
\put(57,10){$\fontact U_2$}
\put(17,40){$x_j$}
\put(80,40){$x_\ell$}
\put(34,50){$r^{U_1}_V$}
\put(65,50){$r^{U_2}_V$}
\put(30,37){$\beta_V(\mathfrak p_A(p))$}
\put(64,37){$b_V$}
\put(3,6){$x_j$}
\put(28,6){$x_\ell$}
\put(68,6){$x_j$}
\put(93,6){$x_\ell$}
\put(-4,24){$\beta_{U_1}(\mathfrak p_A(p))$}
\put(52,22){$\beta_{U_2}(\mathfrak p_A(p))$}
\put(15,12){$b_{U_1}$}
\put(81,12){$b_{U_2}$}
\end{overpic}
\caption{The proof of Lemma~\ref{lem:same_walls}}.\label{fig:same_walls}
\end{figure}

We conclude that $|\mathrm{Sep}|\leq N_0(L)k(k-1)+N$, which is independent of $p$.  Hence there are at most 
$\xi(N_0(L)k(k-1)+N)/M$ separators for $p$, which is again independent of $p$ because the realization constant 
$\xi$ depends on $p$ only to the extent that it depends on the consistency constant $\eta$ for $(b_V(p))$, 
which was shown in Lemma~\ref{lem:consistent_walls} to be independent of $p$.
\end{proof}

\subsubsection{Walls cross if and only if orthogonal}
We now check that the walls $\mathcal L^U_i,\mathcal L^V_j$ cross if and only if $U\orth V$.  One direction, 
done in the first lemma, is essentially just the partial realization axiom for HHS.  The other direction, 
which is the second lemma, relies on our specific choice of walls.
 
 \begin{lemma}\label{lem:walls_cross}
  Suppose $U,V\in\mathcal U$ and $U\orth V$, and fix any $p\in \hull_{\fontact U}(A)$, $q\in \hull_{\fontact V}(A)$. Then there exists $x\in H_{\theta}(A)$ that coarsely projects to $p$ in $\fontact U$ and to $q$ in $\fontact V$.
 \end{lemma}

\begin{proof}
By partial realization (Definition~1.1.(8) in~\cite{hhs2}), there exists $x'\in\cuco X$ projecting $E$--close to $p$ in 
$\fontact U$ and $q$ in $\fontact V$.  Up to replacing $E$ with a uniform constant depending on $\theta$, the projection 
$\gate_{H_\theta(A)}(x')$ to $H_\theta(A)$ has the same property, as required.
\end{proof}

  \begin{lemma}[Cross iff orthogonal]\label{lem:cross}
  The walls $\mathcal L^U_i$ and $\mathcal L^V_j$ cross if and only if $U\orth V$.  
 \end{lemma}
 
 \begin{proof}
 If $U\orth V$, then $\mathcal L_i^U$ crosses $\mathcal L_j^V$ (recall that this means that
 each of the four possible intersections of
 halfspaces, one associated to each wall, is nonempty) by
 Lemma~\ref{lem:walls_cross}.
 
 Conversely, suppose $U\not\orth V$.  We claim $\mathcal
 L^U_i$ and $\mathcal L^V_j$ do not cross.  First, suppose
 $U\transverse V$.  Then, by
 Lemma~\ref{lem:rho_tree}, $\rho^V_U$ and $\rho^U_V$ are uniformly
 close to the image of $A$ in each of the corresponding trees $T_U,T_V$ and hence far from 
 $p^V_j, p^U_i$. Thus, we can
 choose a halfspace from $\mathcal L^U_i$ (resp.  $\mathcal L^V_j$)
 so that all its points project far from $\rho^V_U$ (resp.
 $\rho^U_V$).  The chosen halfspaces are disjoint by consistency.
 
 Second, suppose $U\propnest V$.  By construction, $p_V^j$ is $M$--far from $\rho^{U}_V$, so we can choose a 
halfspace $H$ associated to $\mathcal L^V_j$ such that $\pi_V(H)$ contains a point $\pi_V(x_\ell)$ and is 
disjoint from $\neb_E(\rho^U_V)$.  Consistency and bounded geodesic image imply that $\pi_U(H)$ is $E$--close 
to $\pi_U(x_\ell)$ and hence $M$--far from $p^U_i$.  Thus we can choose a halfspace $H'$ for $\mathcal L^U_i$ 
such that $\pi_U(H)\cap \pi_U(H')=\emptyset$, so $H\cap H'=\emptyset$, and hence $\mathcal L^U_i$ and 
$\mathcal L_j^V$ do not cross.
\end{proof}

\subsubsection{The map $\mathfrak p_A$ is quasimedian}\label{subsubsec:quasimedian}
The next lemma is used to show that $\mathfrak p_A$ is quasimedian, i.e., that it takes medians in the cube complex 
$\cuco Y$ (whose $1$--skeleton is necessarily a median graph) close to coarse medians in $\cuco X$.

\begin{lem}\label{lem:quasimedian}
There exists $C_1''''=C_1''''(\cuco X, k, M)$ so that $\mathfrak p_A$ is $C_1''''$--quasimedian.
\end{lem}

\begin{proof}
Let $\mu\co \cuco X^3\to\cuco X$ be the coarse median map.  Recall from~\cite{hhs2} that $\mu$ is 
characterized by the following property: for all $a,b,c\in\cuco X$ and all $U\in\mathfrak S$, 
$\pi_U(\mu(a,b,c))$ uniformly coarsely coincides with a coarse median point in $\fontact U$ for 
$\pi_U(a),\pi_U(b),\pi_U(c)$.  

Let $x,y,z\in\cuco Y$, and let $m$ be their median.  By Remark~\ref{rem:walls_and_media}, $m$ corresponds to 
the following orientation of the walls of $\cuco Y$: for each wall $W$, $m(W)$ is the halfspace which contains 
at least two of $x,y,z$.  In other words, for each $U\in\mathcal U$ and each $p_i^U\in T_U$, the orientation 
that $m$ assigns to $\{\OL W_i(U),\OR W_i(U)\}$ is the halfspace $\overline W_i(U)$ assigned by at least two of the 
orientations $x,y,z$. 

By definition, for any $V\in\mathfrak S$, we have $b_V(m)=\bigcap_{U\in\mathcal U,i}S_{U,i,V}(m)$, where, for 
each $U,i$, we have that $S_{U,i,V}(m)$ coincides with at least two of 
$S_{U,i,V}(x),S_{U,i,V}(y),S_{U,i,V}(z)$.  

In particular, for each $V\not\in\mathcal U$, we have that $b_V(m)$ coarsely coincides with each of 
$\beta_V(x),\beta_V(y),\beta_V(z)$.  

Also, for each $U\in\mathcal U$ and each $p_i^U$, we have that $b_U(m)$ lies in the same $p_i^U$--halfspace of 
$T_U$ as at least two of the points $b_U(x),b_U(y),b_U(z)$.  Hence $b_U(m)$ lies in the same $p_i^U$--halfspace 
of $T_U$ as $m_U$, where $m_U$ is the median of $b_U(x),b_U(y),b_U(z)$ in the tree $T_U$.  We have shown that 
no $p_i^U$ separates $b_U(m)$ from $m_U$, for any $U\in\mathcal U$.

Our $(1,C)$--quasi-isometrically embedded choice of $T_U$ ensures that $m_U$ is, up to uniformly bounded error, 
a coarse median point for  the images in $\fontact U$ of $\mathfrak p_A(x),\mathfrak p_A(y),\mathfrak p_A(z)$.  
In other words, $\mu(\mathfrak p_A(x),\mathfrak p_A(y),\mathfrak p_A(z))$ is a realization point for 
$(m_V)_{V\in\mathfrak S}$.  As shown earlier in the proof of Theorem~\ref{thm:cubulated_hull}, the image of 
$\mathfrak p_A$ coarsely coincides with $H_\theta(A)$, which is hierarchically quasiconvex by 
Proposition~\ref{prop:coarse_retract}.   Hence $\mu(\mathfrak p_A(x),\mathfrak p_A(y),\mathfrak p_A(z))$ 
uniformly coarsely coincides with $\mathfrak p_A(q)$ for some $q\in\cuco Y$.

Hence there exists $q\in\cuco Y$ such that
$$\dist_{\cuco X}(\mathfrak p_A(m),\mu(\mathfrak
p_A(x),\mathfrak p_A(y),\mathfrak p_A(z)))\asymp\dist_{\cuco
X}(\mathfrak p_A(m),\mathfrak p_A(q))\asymp\dist_{\cuco Y}(m,q)$$ and 
establishes that this distance can 
be bounded in terms of the number of walls separating $m,q$.  Up to
additive error, this is just the sum over $U\in\mathcal U$ of the
number of $p^U_i$ separating $b_U(m)$ from $m_U$, which we established
above was $0$, as required.
\end{proof}

At this point, we have proved all of the lemmas supporting Theorem~\ref{thm:cubulated_hull}.
 
\subsection{Application to coarse median rank and hyperbolicity}\label{subsec:coarse_median}
In~\cite[Theorem 7.3]{hhs2}, we showed that any HHS is a coarse median
space (in the sense of~\cite{Bowditch:coarse_median}) of rank bounded
by the complexity.  In the asymphoric case, the following strengthens 
that result.

The following corollary bounds the median space rank of any asymptotic cone of $\cuco X$; see 
Proposition~\ref{prop:cone_median}.

\begin{cor}\label{cor:rank_hull}
  Suppose that $\cuco X$ is asymphoric. Then any CAT(0) cube complex $\cuco Y$ from Theorem~\ref{thm:cubulated_hull} satisfies $\dimension\cuco Y\leq\nu$, where $\nu$ is the rank of $\cuco X$.
 \end{cor}

\begin{cor}\label{cor:rank_coarse_median}
If $\cuco X$ is an asymphoric HHS of rank $\nu$, then $\cuco X$ is coarse median of rank $\nu$. 
\end{cor}

\begin{proof}[Proof of Corollary~\ref{cor:rank_hull} and Corollary~\ref{cor:rank_coarse_median}]
Choose $M$ as in the proof of Theorem~\ref{thm:cubulated_hull}; since 
$M>E$, in particular $M$ exceeds the asymphoricity constant.
For any finite $A\subset\cuco X$, let
$\cuco Y$ be the cube complex and $\cuco Y\to H_\theta(A)$ be the
$C_1$--quasimedian $(C_1,C_1)$--quasi-isometry provided by
Theorem~\ref{thm:cubulated_hull}.  By Lemma~\ref{lem:cross},
$\dimension\cuco Y$ is equal to the maximal cardinality of sets of
pairwise-orthogonal elements of $\mathcal U$.  But since elements of
$\mathcal U$ have associated hyperbolic spaces of diameter $\geq M$, such subsets
have cardinality bounded by $\nu$.  This proves
Corollary~\ref{cor:rank_hull}.  Moreover, $\cuco Y^{(0)}\to
H_\theta(A)$ is a quasimedian map from a finite median algebra
satisfying the condition $(C2)$ from the definition of a coarse median
space in~\cite[Section 8]{Bowditch:coarse_median}.  The rank of this
median algebra is, by definition, $\dimension\cuco Y\leq\nu$.  Hence
$\cuco X$ is coarse median of rank $\nu$.
\end{proof}

We can also use the proof of Corollary~\ref{cor:rank_coarse_median} 
to characterize hyperbolic HHS.  We say that a quasi-geodesic metric space $X$ is \emph{hyperbolic} if there exist $D$ and 
$\delta$ so that
\begin{itemize}
 \item any pair of points of $X$ is joined by a $(D,D)$--quasi-geodesic, and
 \item $(D,D)$--quasi-geodesic triangles are $\delta$--thin.
\end{itemize}

For us, the distinction between hyperbolic geodesic spaces and hyperbolic quasi-geodesic spaces does not matter.  Indeed, any quasi-geodesic metric space $X$ is quasi-isometric to a geodesic metric space $Y$ (in fact, a graph). If, in addition, $X$ is hyperbolic then $Y$ is hyperbolic (in the usual sense). There is a number of ways to see this, one of which is the ``guessing geodesics'' criterion for hyperbolicity from \cite[Section 3.13]{MasurSchleimer}\cite[Proposition 3.1]{Bowditch:uniform}.  It thus follows from~\cite[Theorem 2.1]{Bowditch:coarse_median} that a coarse median quasigeodesic space is hyperbolic if and only if it has rank at most $1$.

We thus get a characterization of HHSs which are hyperbolic, which we use below in the proof of 
Lemma~\ref{lem:quasi_median_top_dimensional}:

\begin{cor}\label{cor:hyperbolic}
Let $(\cuco X,\mathfrak S)$ be an HHS.  Then the following are equivalent:
\begin{itemize}
 \item $\cuco X$ is coarse median of rank $\leq 1$, and is thus hyperbolic;
 \item (Bounded orthogonality) There exists $q\in\reals$ so that $\min\{\diam(\fontact U),\diam(\fontact V)\}\leq q$ for all $U,V\in\mathfrak S$ satisfying $U\orth V$.
\end{itemize}
\end{cor}

\begin{proof}
 The fact that hyperbolicity implies bounded orthogonality easily
 follows from the construction of standard product regions.  The
 reverse implication follows from Corollary
 \ref{cor:rank_coarse_median}, with $\nu=1$, and the aforementioned
 \cite[Theorem 2.1]{Bowditch:coarse_median}.
\end{proof}

\begin{rem}
 One can prove that bounded orthogonality implies hyperbolicity using
 the guessing geodesics criterion instead of the coarse median rank.
 More specifically, triangles of hierarchy paths are thin because any
 such triangle is contained in the hull of the vertices, which is
 quasi-isometric to a $1$--dimensional CAT(0) cube complex, i.e., a tree.
\end{rem}

\section{Quasiflats and asymptotic cones}\label{sec:quasiflats_and_cones}

Fix an asymphoric hierarchically hyperbolic space $(\cuco X,\mathfrak
S)$ of rank $\nu$ and let $\seq{\cuco X}$ be an asymptotic cone of
$\cuco X$.  According to Proposition~\ref{prop:cone_median}, the
coarse median map on $\cuco X$ limits to a median map on $\seq{\cuco X}$
making it into a topological median space of rank at most $\nu$.  By
the same proposition, after changing the metric on $\seq{\cuco X}$
within its bilipschitz equivalence class, we can assume that
$\seq{\cuco X}$, with its given median, is a median metric space.

With this setup in mind, we now outline this section.  First, the goal
is to show that given a quasiflat in $\cuco X$, there are arbitrarily
large balls contained in a uniform neighborhood of the hull of
boundedly many points; this is made precise in Corollary
\ref{cor:fixed_flat_ball_hull}, and this is what will allow us to
apply Huang's quasiflat theorem for CAT(0) cube complexes
\cite{Huang:quasiflats} to describe quasiflats in HHSs.  Subsection
\ref{subsec:ultralim_median} contains preliminary lemmas that relate ultralimits of objects in $\cuco X$ defined in terms of 
the HHS structure to objects in $\seq{\cuco X}$ defined in terms of the median structure.

The content of Lemma \ref{lem:squish} and
Proposition \ref{prop:thanks_Brian} is best explained in reversed
order: In Proposition \ref{prop:thanks_Brian}, we argue that there are
balls of large radius $R$ in quasiflats in $\cuco X$ that stay
$\epsilon R$--close to hulls of finitely many points, for a fixed small
$\epsilon>0$.  Taking ultralimits, this gives a bilipschitz flat in an
asymptotic cone that stays within bounded distance of a certain median
convex subspace, and Lemma \ref{lem:squish} says that this means that, 
in fact, the flat is contained in the convex subspace.  This
corresponds to an improvement from ``$\epsilon R$'' to ``$o(R)$''.  We
then need to further improve this to ``$O(1)$'', which is performed in
Proposition \ref{prop:thanks_Mauro} by shrinking the previously found
balls.  Further discussion of the various statements and the
corresponding proofs can be found below.

\subsection{Ultralimits of hulls and some median preliminaries}\label{subsec:ultralim_median}
Given $m,m'$ in a median space $M$, we let $\hull(m,m')$ denote the set of $z\in M$ for which the median of $m,m',z$ 
is $z$.  (Note that $\hull(m,m')=[m.m']$, where $[m.m']$ is the \emph{median interval}, defined just as before.  The term 
``median interval'' is more standard, but we think of median intervals as convex hulls of pairs of points, which explains 
our choice of notation.)

Fix a hierarchically quasiconvex subspace $A\subseteq \cuco X$ and points $p,q\in A$, $x\in\cuco X$. Note that the coarse 
median of 
$(p,q,x)$ lies uniformly close to $A$ (see e.g., \cite[Section 
7]{hhs2} or \cite[Section 5]{RST}) --- this easily yields the 
first 
assertion of the following lemma, which we use freely throughout this 
section.

\begin{lemma}
 For any $\kappa$, the ultralimit of any sequence of
 $\kappa$--hierarchically quasiconvex subspaces is median convex.  Moreover, if $(A_n)$ is a sequence of 
$\kappa$--hierarchically quasiconvex subspaces and $\seq{A}\subset\seq{\cuco X}$ is their ultralimit, then the 
maps $\gate_{A_n}\colon \cuco X\to A_n$ limit to the median gate map 
$\gate\colon \seq{\cuco X}\to\seq{A}$.
\end{lemma}

\begin{proof}
We prove the assertion about gates, as the other facts are already
established, as noted above.  Fix $\seq{x}\in\seq{\cuco X}$,
represented by a sequence $(x_n)$ in $\cuco X$.  Fix
$\seq{a}\in\seq{A}$, represented by a sequence $(a_n)$.  For each $n$,
let $b_n=\gate_{A_n}(x_n)$, and let $\seq b$ be represented by $(b_n)$.

By the definition of the gate and the coarse
median, the coarse median of $a_n,b_n,x_n$ is uniformly close to
$b_n$, so the median of $\seq a,\seq b,\seq x$ is $\seq b$.  Hence the
median interval between $\seq x$ and any point in $\seq A$ contains
$\seq b$; it follows immediately from the definition of gate that $\seq b=\gate(\seq x)$.
\end{proof}

We will also tacitly use the next lemma throughout this section.  It 
states that the (median) convex hull of a pair of points 
in an asymptotic cone of $\cuco X$ arises as a limit of $\theta$--hulls of pairs of points in $\cuco X$.

\begin{lemma}\label{lem:ultralimit_of_hulls}
 Let $\seq x,\seq y\in\seq{\cuco X}$. Then $\hull(\{\seq x,\seq y\})=\lim_\omega H_{\theta}(\{x_n,y_n\})$.
\end{lemma}

\begin{proof}
	If $z_n\in H_{\theta}(x_n,y_n)$ then $m(x_n,y_n,z_n)$ coarsely
	coincides with $z_n$, which yields $$\lim_\omega
	H_{\theta}(x_n,y_n)\subseteq \hull(\seq x,\seq y).$$ 
	
	To prove the
	other containment, suppose $\seq z'\in \hull(\seq x,\seq y)$ (and 
	whence, by definition of the hull, that $\seq z'=m(\seq x,\seq y,\seq z')$), 
	and let $(z'_{n})$ be a representative for $\seq z'$.  Let
	$z_n=m(x_n,y_n,z'_n)\in H_{\theta}(x_n,y_n)$ and note that this 
	implies $\seq z=m(\seq x,\seq y,\seq {z'})$, where $\seq z$ is 
	the point represented by $(z'_{n})$.
	Since $\seq{\cuco X}$ is a 
	median space, the median of a triple is unique and thus $\seq
	z'=\seq z$; whence $\seq z'\in \lim_\omega H_{\theta}(x_n,y_n)$.
\end{proof}

\subsection{Bilipschitz flats in asymptotic cones}

The next lemma relies on results of Bowditch about \emph{cubulated} subsets of median metric spaces~\cite{Bowditch:rigid}.  
The import of the lemma is the following. Consider a top-rank bilipschitz flat $\seq F$ in $\seq{\cuco X}$ and a 
median-convex subspace $\seq H$ arising as an ultralimit of hierarchically quasiconvex subspaces of $\cuco X$. If $\seq F$ lies in a uniform neighborhood of $\seq H$, then it must actually be contained in $\seq 
H$.  This will be applied in the proof of Proposition~\ref{prop:thanks_Mauro} in the case where $\seq H$ is a limit of 
$\theta$--hulls of finite sets in $\cuco X$ of bounded cardinality.

Roughly, the idea of proof is as follows.  If the bilipschitz flat
$\mathbf F$ was median convex, we would have gate maps on both
$\mathbf F$ and $\seq H$, and $\mathbf F$ could only stay close to
$\seq H$ around $\gate_{\mathbf F}(\seq H)$, which then needs to be
the whole of $\mathbf F$.  Since $\mathbf F$ is top-dimensional, and
$\gate_{\mathbf F}(\seq H)$ is one of the factors of a product
subspace of $\seq{\cuco X}$ (in view of
Lemma~\ref{lem:parallel_gates}), the other factor has to be trivial.
On the other hand, the other factor being trivial is the same as
$\mathbf F$ being contained in $\seq H$.  Now, $\mathbf F$ need not be
median convex, and to deal with this we rely on results from
\cite{Bowditch:rigid} that, roughly, give us a decomposition of (large
portions of) the quasiflat into ``blocks,'' each of which is median
convex, and we then consider chains of such blocks.

\begin{lem}[Close to convex implies contained in convex]\label{lem:squish}
 Let $\seq{\cuco X}$ be an asymptotic cone of $\cuco X$ and let
 $\mathbf F\subseteq \seq{\cuco X}$ be a bilipschitz  $\nu$--flat.
 Let
 $\seq H$ be an ultralimit of uniformly hierarchically quasiconvex
 subsets of $\cuco X$ and suppose that $\mathbf F$ is contained in a
 neighborhood of $\seq H$ of finite radius.  Then $\mathbf F\subseteq
 \seq H$.
\end{lem}

\begin{proof}
Suppose by contradiction that there exists some $p\in \mathbf F- \seq H$.

By \cite[Proposition 1.2]{Bowditch:rigid}, $\seq F$ is
\emph{cubulated in the sense of~\cite{Bowditch:rigid}}, which means 
that there are arbitrarily large balls in $\seq F$ each of which is
contained in a finite union of blocks.  By
\cite[Proposition~3.3]{Bowditch:rigid}, this implies that there are
arbitrarily large balls $B$ in $\mathbf F$ with the following
property: $B$ is contained in a subset of $\mathbf F$ which is a union
of blocks whose pairwise intersections are each either empty or a common face.  We let
$\mathbf F'$ be such a union of blocks which contains a ball around
$p\in \mathbf F$ of radius $r$ much larger than $\sup_{x\in \mathbf
F}\dist(x,\seq H)$.
 
  After possibly subdividing the cubulation of $\mathbf F'$, there is
  a $\nu$--block $B_0$ of $\mathbf F'$ containing $p$ and disjoint
  from $\seq H$.  After subdividing, we can assume that each side of $B_0$ has length bounded by some $\ell$ 
much smaller than $r$.   
  
  Being a block, $B_0$ is the median interval between a pair of opposite corners of $B_0$.  So, by Lemma 
\ref{lem:ultralimit_of_hulls}, $B_0$ is the ultralimit $\seq H_0$ of a sequence $(H_\theta(c_n,d_n))$ of $\theta$--hulls of 
pairs of points.  
 
 As noted in Definition~\ref{defn:block_gate}, $\gate_{\subseq H_0}(\seq H)$ is a  median convex subspace.  So, 
$\gate_{\subseq H_0}(\seq H)$ is a sub-block $B'$ of $B_0$. 

On the other hand, by hypothesis, $\seq H$ is the limit of uniformly hierarchically quasiconvex subspaces $H_n$ of $\cuco 
X$.  By Lemma~\ref{lem:ultralimit_of_hulls}, $\gate_{\subseq H_0}(\seq H)$ is the limit of the subspaces 
$\gate_{H_\theta(c_n,d_n)}(H_n)$.  By Lemma~\ref{lem:parallel_gates}.\eqref{item:bridge}, $\gate_{H_\theta(c_n,d_n)}(H_n)$ 
is coarsely contained in a quasi-isometrically embedded copy of $\gate_{H_\theta(c_n,d_n)}(H_n)\times I_n$, where $I_n$ is a $\theta$--hull.  So, taking limits, we see that, if $B'$  has dimension $i$, then there is an
  $(i+1)$--dimensional topologically embedded copy of $[0,1]^{i+1}$ in
  $\seq{\cuco X}$.  This implies $i<\nu$.
 
 For any codimension--$1$ face $B_2$ of $B_0$ not intersecting $B'$, there exists a
 block $B'_1$ whose intersection with $B_0$ is $B_2$. So,
 $B_1=B_0\cup B'_1$ is a block by \cite[Lemma~3.2]{Bowditch:rigid}.  We claim
 $\gate_{B_1}(\seq H)=\gate_{B_0}(\seq H)$, which implies that $B_1$
 is also disjoint from $\seq H$.  
 
 To prove the claim, note that
 $B'=\gate_{B_0}(\seq H)=\gate_{B_0}(\gate_{B_1}(\seq H))$.  (It is a general fact about median metric spaces, following 
directly from the definition of a gate, that if $A,C$ are median-convex closed subspaces and $A\subset C$, then 
$\gate_A=\gate_A\circ\gate_C$.) 

Now, since $B_0$ is a sub-block of the block $B_1$, and $B_0$ intersects the closure of its complement in $B_1$ along a 
common codimension--$1$ face, and $B_1$ is median-isomorphic to a finite product of intervals with the $\ell_1$--metric (by 
the definition of a block), $\gate_{B_0}|_{B_1}$ is just the natural retraction.  So, this map is
 one-to-one on $B'$, and the claim follows.
 
Now proceed inductively until we find a block $B_m$ that we
cannot extend to a block $B_{m+1}$ using the procedure above,
implying that we reached the boundary of $\mathbf F'$.  By induction, $\gate_{B_m}(\seq H)=\gate_{B_0}(\seq H)$.

Hence $\gate_{B_m}(\seq H)=B'$, since we had $B'=\gate_{B_0}(\seq H)$.  Let $q\in B_m$ lie in the boundary of $\mathbf F'$. 
 Then $\dist_{\seq{\cuco X}}(q,B')$ is at least $\dist_{\seq{\cuco X}}(q,p)-\nu \ell$, which exceeds 
$\sup_{x\in \mathbf 
F}\dist(x,\seq H)$.  Hence there exists $h\in\seq H$ with $\dist_{\seq{\cuco X}}(h,q) < \dist_{\seq{\cuco X}}(B',h)$.  This 
contradicts that $\gate_{B_m}(\seq H)=B'$.  (Here we are using the median metric $\dist_{\seq{\cuco X}}$, for which the 
notions of median gate and closest-point projection coincide, by the definition of a median gate.)  This is the required 
contradiction.
\end{proof}

\subsection{Quasiflats and hulls}

As mentioned above, we now argue that given a quasiflat in $\cuco X$ there are balls of large radius $R$ that stay $\epsilon R$-close to hulls of finitely many points, for a fixed small $\epsilon>0$. Once again we use \cite[Proposition
 1.2]{Bowditch:rigid}, which provides a subdivision of (large portions) of the ultralimit of the quasiflat into blocks, and then use the fact that each such block is the ultralimit of hulls of pairs of points.

\begin{prop}\label{prop:thanks_Brian}
 Let $F\co\reals^\nu\to\cuco X$ be a quasiflat.  Then, there exists
 $N$ (depending on $F$) so that the following holds.  For any $\epsilon>0$ and every
 $R_0$ there exists a ball $B=B_R(0)\subseteq \reals^\nu$ of radius $R\geq
 R_0$ and a set $A\subseteq \cuco X$ with $|A|\leq N$ so that
 $F(B)\subseteq \neb_{\epsilon R}(H_\theta(A))$.
\end{prop}

\begin{proof}
The proof has two parts.

 \textbf{Choosing $N$:} Let $\seq{\cuco X}$ be a \emph{fixed} asymptotic cone of $\cuco X$ with
 observation points a constant sequence $(F(0))$.  Let $\seq F\co\reals^\nu\to
 \seq{\cuco X}$ be the corresponding ultralimit of $F$.  Let $\seq B$ be a
 ball of radius $1$ in $\reals^\nu$.  By \cite[Proposition
 1.2]{Bowditch:rigid}, $\seq F(\seq B)$ is contained in a finite union of
 blocks.  Notice that each block is the convex hull of a pair of
 opposite corners. The cardinality of the number of corners provides 
 the desired $N$.  By Lemma \ref{lem:ultralimit_of_hulls}, $\seq
 F(\seq B)$ is contained in the ultralimit of hulls of pairs of points.
 Thus, $\seq F(\seq B)$ is contained in the ultralimit of a sequence of
 hulls of sets of at most $N$ points (the hull of a union contains the
 union of the hulls). 
 
\textbf{Remark on non-uniformity of $N$:}  We remark that, for the purposes of this proof, $N$ is allowed to depend on the 
particular quasiflat $F$, not just the quasi-isometry constants.  We are also allowing $N$ to depend on our choice 
$\seq{\cuco X}$ of asymptotic cone.  Bowditch's proposition (Proposition 1.2 in~\cite{Bowditch:rigid}) provides only 
that $\seq F(\seq B)$ is contained in a \emph{finite} union of 
blocks, but does not bound the number; for our result we only need 
finiteness.
 
 \textbf{Conclusion:}  Now, suppose by contradiction that the conclusion of the proposition fails.  Then for each $N$, and 
in particular the $N$ we found above, there is $\epsilon>0$ so that, for all balls $B(0,R)$ of sufficiently large radius $R$, 
we have that $F(B(0,R))$ cannot be contained in $\neb_{\epsilon R}(H_\theta(A))$ for any $A\subseteq \cuco X$ with $|A|\leq 
N$. Let $B_n=B(0,R_n)$, where $(R_n)$ is the scaling factor of the asymptotic cone $\seq{\cuco X}$ fixed above. Then $\seq B$ 
is the ultralimit of the $B_n$. The fact that $\seq F(\seq B)$ is contained in the ultralimit of a sequence of hulls 
$H_\theta(A_n)$ of sets $A_n$ of at most $N$ points implies that, for $\omega$--a.e. $n$, $F(B_n)$ is contained in 
$\neb_{\epsilon R_n}(H_\theta(A_n))$, a contradiction.
\end{proof}

The following is the most technical proposition of this section, and
it says that by shrinking the balls provided by Proposition
\ref{prop:thanks_Brian}, we obtain balls contained in a uniform
neighborhood of hulls of boundedly many points.  The rough reason for
this is the following. In view of Lemma \ref{lem:squish}, in any asymptotic
cone the ultralimit of the balls is contained in the ultralimit of the
hulls; this means that the distance from the flat to the hulls
grows more slowly than any superlinear function. From this we deduce 
the distance is bounded. 
To make this work, we must consider only asymptotic cones where the ultralimit of the balls is a bilipschitz
flat, so the observation point must be deep in the balls; in 
the proof we deal with this by 
using balls of half the radius to ensure 
this holds in the relevant asymptotic cones.

\begin{prop}\label{prop:thanks_Mauro}
 For every $K,N$ there exist $\epsilon>0$, $R_0$ and $L$ with the
 following property.  Let $B$ be a ball of radius $R\geq R_0$ in
 $\reals^\nu$, and let $F\co B\to\cuco X$ be a $(K,K)$--quasi-isometric
 embedding.  Let $A\subseteq \cuco X$ have $|A|\leq N$, and suppose
 that $F(B)\subseteq \neb_{\epsilon R}(H_\theta(A))$.  Then
 $F(B')\subseteq \neb_L(H_\theta(A))$, where $B'$ is the sub-ball of
 $B$ with the same center and radius $R/2$.
\end{prop}

\begin{proof}
If not, there exist constants $K,N$ and:
\begin{itemize}
 \item balls $B_m=B_m(0)$ of radius $R_m$ in $\reals^\nu$, and $(K,K)$--quasi-isometric embeddings $F_m:B_m\to\cuco X$,
 \item subsets $A_m\subseteq \cuco X$ with $|A_m|\leq N$ and
 $$\lim_{m\to\infty} \frac{1}{R_m}\sup_{x\in B_m} \dist (F_m(x),H_{\theta}(A_m))=0,$$
 but $\lim_{m\to\infty} \sup_{x\in B_{R_m/2}(0)} \dist (F_m(x),H_{\theta}(A_m))=\infty$.
\end{itemize}

 We define $\ell_m(t)=\sup_{x\in F_m(B_{\min\{t,R_m\}}(0))}
 \dist(x,H_\theta(A_m))$.  The ultrapower $\seq \ell$ of the $\ell_m$
 can be regarded as a function $\seq \ell\co \,^\omega
 \reals_+\to\,^\omega\reals_+$. 
Note that $\seq \ell$ is non-decreasing.
 
 Let $\sigma\in\,^\omega\reals_+$ be represented by $\seq R$. For 
 $\seq S,\seq T\in\,^\omega\reals_+$ we write $\seq S\ll\seq T$ if $\lim_\omega S_m/T_m=0$, and we write $\seq S<\infty$ if $\lim_\omega S_m\neq \infty$, i.e., if $\seq S\gg 1$ does not hold.  We find a contradiction (with the second bullet
 above) provided we show $\seq\ell(\sigma/2)=\lim_{\omega,m}
 \ell_m(R_m/2)<\infty$.

 The first part of the second bullet above implies that $\seq \ell(\sigma)\ll \sigma$. We first need:
 
 \begin{claim}\label{claim:mauro_1}
  For $\lambda\in\,^\omega \reals_+$, if $\seq\ell(\lambda)\gg 1$, then for any $\alpha\gg 1$ we have $\seq\ell(\lambda-\alpha \seq\ell(\lambda))\ll \seq\ell(\lambda)$.
 \end{claim}
 
 \begin{proof}[Proof of Claim~\ref{claim:mauro_1}]
 Suppose not.  Consider an asymptotic cone $\seq{\cuco X}$ of $\cuco
 X$ with the observation point in $F(B_{\lambda-\alpha
 \subseq\ell(\lambda)}(0))$ and scaling factor $\seq\ell(\lambda-\alpha
 \seq\ell(\lambda))$.  Then any point in the image of $\seq F$ has  
 distance from $\seq H$ bounded above by 
 $\seq\ell(\lambda)/\seq\ell(\lambda-\alpha \seq\ell(\lambda))<\infty$.  In fact,
 any point of the image of $F$ which gives a point of $\seq{\cuco X}$
 lies in a ball of radius $\lambda-\alpha \seq\ell(\lambda) +t \seq\ell
 (\lambda-\alpha \seq\ell(\lambda)) \leq \lambda-\alpha \seq\ell(\lambda) +t \seq\ell
 (\lambda)$ for some finite $t$, and hence in particular in the image of
 the ball of radius $\lambda$.
 
 By Lemma \ref{lem:squish} we have $\seq F\subseteq \seq H$.  But, we
 chose an arbitrary observation point in $F(B_{\lambda-\alpha
 \seq\ell(\lambda)}(0))$, and thus we get a contradiction by choosing a point
 that maximizes the distance from $H_\theta(A)$.
 \end{proof}
 
 We claim that there exists $T_0\in\reals_+$ so that the following holds for $\omega$--a.e. $m$: if $\ell_m(t)\geq 
T_0$ for some $t$, and $\alpha\geq T_0$, then $\ell_m(t-\alpha \ell_m(t))\leq \ell_m(t)/2$.
 
  \begin{rem}
The proof follows from Claim \ref{claim:mauro_1} by an application of
the principle from nonstandard analysis called \emph{underspill},
which says that if a predicate is true for all infinitesimal positive 
non-standard reals, 
then it is also true for all sufficiently small standard reals.

Since we
do not wish to require familiarity with non-standard analysis, 
rather than invoking this principle we instead 
provide a self-contained argument in the language of ultrafilters. 
Since our argument is a translation of the non-standard analysis 
argument, rather than providing a convoluted heuristic explanation 
for how this argument works, 
we refer the reader who would like
to do more than check that the argument is formally correct to Tao's
excellent blog post \cite{tao:nonstandard}, which explains all the 
relevant concepts. We note, though, that this argument is the 
usual one which is used to prove that ultrapowers
are saturated models and also in proving the 
nonstandard formulation of continuity, see 
\cite[Proposition 11]{tao:nonstandard}, which is a typical 
application of underspill.
 \end{rem}

 For each
$n\in\naturals$, let $\mathcal U_n$ be the set of $m\ge n$ for which
there exists
$t_{m,n},\alpha_{m,n}\in\reals_+$ so that $\ell_m(t_{m,n})\geq n$ and $\alpha_{m,n}\geq n$ and 
$\ell_m(t_{m,n}-\alpha_{m,n} \ell_m(t_{m,n}))>\ell_m(t_{m,n})/2$.  
Suppose that our claim does not hold, i.e.,  
suppose the desired $T_0$ does not exist.  Then, for arbitrarily large $n$, we have that $m\in\mathcal U_n$ for 
$\omega$--a.e. $m$.  For each $m$, let $n(m)$ be the maximal $n$ for which $m\in\mathcal U_n$.  Our assumption, and 
the fact that $m\not\in\mathcal U_n$ for $n>m$, ensures that $n(m)$ exists for $\omega$--a.e. $m$.  

Let $\lambda\in ^\omega\reals_+$ be the ultralimit of $t_{m,n(m)}$ and let $\alpha$ be that of  
$\alpha_{m,n(m)}$.  Then $\ell(\lambda)\gg 1$ and $\alpha\gg1$, so Claim~\ref{claim:mauro_1} implies that 
$\seq\ell(\lambda-\alpha \seq\ell(\lambda))\ll \seq\ell(\lambda)$.  This contradicts that 
$\ell_m(t_{m,n(m)}-\alpha_{m,n(m)} \ell_m(t_{m,n(m)}))>\ell_m(t_{m,n(m)})/2$ for $\omega$--a.e. $m$.  Thus we have 
$T_0$ with the claimed property for $\omega$--a.e. $m$.  
 
Fix one such $m$, which furthermore satisfies 
$\ell_m(R_m)\leq R_m/(4\alpha_0)$ (which is satisfied by $\omega$--a.e. $m$ by the second bullet).  Let 
$R_m^j=R_m(1+2^{-j})/2$. In particular, $R^0_m=R_m$.
 
\begin{claim}\label{claim:mauro_2}
  Either $\ell(R_m^j)\leq \ell_m(R_m)/2^{j}$ or there exists $i\leq j$ with $\ell_m(R_m^i)<T_0$.
\end{claim}
 
\begin{proof}[Proof of Claim~\ref{claim:mauro_2}]
 We argue by induction on $j$. Suppose that $R_m^j$ satisfy $\ell_m(R_m^j)\leq \ell_m(R_m)/2^{j}$ and 
$\ell_m(R_m^j)\geq T_0$. Note that $R_m^{j+1}=R^j_m-2^{-j-2}R_m=R_m^j-\alpha_m^j\ell_m(R_m^j)$ for some 
$\alpha_m^j\geq T_0$. Hence, the claim gives $\ell_m(R_m^{j+1})\leq \ell_m(R_m^j)/2\leq \ell(R_m)/2^{j+1}$, as 
required.
\end{proof}

 In either of the two cases provided by Claim~\ref{claim:mauro_2}, there exists $j$ with $\ell_m(R^j_m)<T_0$. This implies $\ell_m(R_m/2)<T_0$, and hence $\seq\ell( \sigma/2)<T_0$, as required. 
 \end{proof}

Combining Proposition \ref{prop:thanks_Brian} and Proposition \ref{prop:thanks_Mauro}, one gets:

\begin{cor}\label{cor:fixed_flat_ball_hull}
 For every quasi-isometric embedding $f\co\reals^n\to\cuco X$, there exist $L,N$ so that the following holds.   Then there 
exist arbitrarily large $R$ so that for the ball $B$ of radius $R$ around $0$, there is a set $A_R\subset\cuco X$ with 
$|A_R|\leq N$ and $f(B')\subseteq \neb_{L}(H_\theta(A_R))$, where $B'$ is as in Proposition~\ref{prop:thanks_Mauro}.
\end{cor}

\section{Orthants and quasiflats}\label{sec:main}

From now on, we fix an asymphoric HHS $(\cuco X,\mathfrak S)$ of rank 
$\nu$. The main goal of this section is to prove the quasiflats theorem, Theorem \ref{thmi:main}, which says that quasiflats in $\cuco X$ are at bounded Hausdorff distance from a finite union of standard orthants. After some preliminary work on orthants, we complete the proof in Subsection \ref{subsec:quasiflat_proof}. Then, we prove two further results giving more quantitative control on quasiflats in $\cuco X$ in terms of the number of orthants needed (Theorem \ref{thm:uniform}) and the Hausdorff distance between the quasiflat and not quite the union of the orthants, but rather the hull of the union (Lemma \ref{lem:flats_close_to_hulls}).

\subsection{Orthants in $\cuco X$}\label{subsec:orthants}
We fix once and for all a constant $D$ so that, for any $U\in\mathfrak S$, any two points in $F_U$ are connected by a 
$D$--hierarchy path. (Such a constant is provided by 
Theorem~\ref{thm:monotone_hierarchy_paths}.)

We now discuss \emph{standard orthants} in $\cuco X$, which are one of the basic objects in the statement of 
Theorem~\ref{thmi:main}.

\begin{defn}[Standard orthant, standard flat, standard partial flat]\label{defn:standard_orthant}
Let $U_1,\ldots,U_k$ be pairwise orthogonal elements of $\cuco X$.  Recall that we have a quasimedian quasi-isometric 
embedding $F_{U_1}\times\cdots\times F_{U_k}\to\cuco X$, as described in 
Section~\ref{subsubsec:basic_hhs_notions}, with 
constants independent of the $U_i$.  

For each $i\leq k$ and each $x\in\prod_{j\neq i}F_{U_j}$, the image of 
$F_{U_i}\times\{x\}$ is a (uniformly) hierarchically quasiconvex subset which, abusing notation slightly, we also denote 
$F_{U_i}\times\{x\}$, or simply by $F_{U_i}$ when the choice of parallel copy is not important.  

For each $i$, let $\gamma_i$ be a 
$D$--hierarchy ray in $F_{U_i}$ with the property that $\pi_{U_i}(\gamma_i)$ is unbounded.  We call the image of 
$\gamma_1\times\dots\times\gamma_k\subseteq F_{U_1}\times\dots\times F_{U_k}$ under the standard embedding a \emph{standard 
$k$--orthant} in $\cuco X$ with support set $\{U_i\}$.
 
 A \emph{standard orthant} is a standard $\nu$--orthant, i.e., a standard $k$--orthant of maximum possible 
dimension.

Similarly, given $U_1,\ldots, U_k$ as above, suppose we have for each $i\leq k$ a path $\gamma_i$ in $F_{U_i}$ 
such that $\gamma_i$ is either a $D$--hierarchy ray or a bi-infinite $D$--hierarchy path such that 
$\pi_{U_i}(\gamma_i)$ is unbounded.  Then the image of 
$\gamma_1\times\cdots\times\gamma_k$ is a \emph{standard 
$k$--partial flat}, or a \emph{standard partial flat} if $k=\nu$.  If every $\gamma_i$ is bi-infinite, then we 
use the term \emph{standard $k$--flat}, or \emph{standard flat} if $k=\nu$.
\end{defn}

\begin{rem}\label{rem:bounded_proj}Observe that if
$Q=\gamma_1\times\dots\times\gamma_k\subseteq F_{U_1}\times\dots\times
F_{U_k}$ is a standard $k$--orthant, then it has uniformly bounded
projection to $\fontact U$ unless $U\nest U_i$ for some $i$.  More
precisely, each $\gamma_i$ has uniformly bounded projection to
$\fontact U$ unless $U\nest U_i$ (in particular, $\pi_U(\gamma_i)$ is
uniformly bounded for $U\nest U_j,j\neq i$).  For each $i$ and each
$U\nest U_i$, we have that $\pi_U(Q)$ uniformly coarsely coincides with
$\pi_U(\gamma_i)$.
\end{rem}

The next lemma says that top-dimensional standard orthants in an
asymphoric HHS are hierarchically quasiconvex (with uniform
hierarchical quasiconvexity function).  Here, an analogy to the CAT(0)
cube complex situation is again instructive.  If $\Pi$ is a CAT(0)
cube complex, and $O\subset\Pi$ is a cubical orthant, then although
$O$ is $\ell_1$--isometrically embedded (i.e., its $0$--skeleton is a
$1$--connected median subalgebra) by definition, it need not be convex: picture the
case where $\Pi=\mathbf R^2$ and $O$ is the ray with $0$--skeleton
consisting of points $\{(n,n),(n,n+1):n\in\naturals\}$.  On the other hand, if
$O$ has the property that $\dimension O=\dimension\Pi$, then $O$
cannot contain the ``corner'' of any cube of $\Pi$ that is ``missing''
in $O$, i.e., $O$ is convex.  This cubical fact is important in
Huang's work~\cite{Huang:quasiflats}.  The final assertion of the next
lemma is analogous.

\begin{lem}[Top dimensional orthants are hierarchically quasiconvex]\label{lem:orthant_quasiconvex}
Consider a standard $k$--orthant $O$ whose support set $\{U_i\}$ has the property that,  for some $C$, we have 
$\min\{\diam_{\fontact U}(\pi_U(O)),\diam_{\fontact V}(\pi_V(O))\}\leq C$ whenever $U,V\nest U_i$ are orthogonal and $i\leq 
k$.  Then $O$ is $\kappa$--hierarchically quasiconvex, where $\kappa$ depends on $C,D,\cuco X,\mathfrak S$.

In particular, there exists a function $\kappa$, depending on $(\cuco X,\mathfrak S),D$, and the asymphoricity constant,  so 
that standard orthants are $\kappa$--hierarchically quasiconvex, and the same holds for standard $k$--orthants contained in 
standard orthants.
\end{lem}

\begin{rem}\label{rem:flats_quasiconvex}
Lemma~\ref{lem:orthant_quasiconvex} holds when the standard orthant $O$ is replaced by a standard flat or 
standard partial flat; the exact same proof works, except with some 
of the rays replaced by bi-infinite paths.  
The main lemma being used is Lemma~\ref{lem:hq_hierarchy_path}, which is stated for arbitrary hierarchy paths.
\end{rem}

\begin{proof}
Let $O$ be a standard $k$--orthant which is the image of $\prod_{i=1}^k\gamma_i$, where each $\gamma_i$ is a hierarchy path in $F_{U_i}$ and $\{U_1,\ldots,U_k\}$ is a pairwise orthogonal set supporting $O$, and let $C$ be the given constant.

By Remark~\ref{rem:bounded_proj} and the fact that hierarchy paths 
project close to geodesics, 
$\pi_U(O)$ is uniformly quasiconvex in $\fontact U$, for $U\in\mathfrak S$.

Suppose $x\in\cuco X$ has the property that $\pi_U(x)$ lies uniformly close to $\pi_U(O)$ for each $U\in\mathfrak S$; to verify hierarchical quasiconvexity of $O$, we must bound the distance from $x$ to $O$.  

By hierarchical quasiconvexity of $\prod_jF_{U_j}$, our $x$ must lie
uniformly close to $\prod_jF_{U_j}$, so it suffices to show that
$\gate_{F_j}(x)$ lies uniformly close to $\gamma_j$ for each $j$,
where $F_j$ denotes the parallel copy of $F_j$ containing the
``corner'' of $O$.  Since $\pi_U(x)$ coarsely coincides with
$\pi_U(\gate_{F_j}(x))$ when $U\nest U_i$, this follows from
hierarchical quasiconvexity of $\gamma_j$, i.e.,
Lemma~\ref{lem:hq_hierarchy_path}.
\end{proof}

The next lemma supports the preceding one.  It gives a sufficient
condition for a hierarchy path to be hierarchically quasiconvex.  The
reader familiar with the work of Huang may find it useful to compare 
this lemma 
with the notion of a ``straight'' geodesic in a CAT(0) cube complex,
defined in~\cite{Huang:quasiflats}.

\begin{lem}[``Straight'' hierarchy paths]\label{lem:hq_hierarchy_path}
Let $\gamma\co I\to\cuco X$ be a $(D,D)$--hierarchy path, where $I\subseteq\reals$ is an interval.  Suppose that there exists $C$ so that, whenever $U\orth V$, either $\pi_U(\gamma)$ or $\pi_V(\gamma)$ has diameter bounded by $C$.  Then $\gamma$ is $\kappa$--hierarchically quasiconvex, where $\kappa=\kappa(D,\cuco X,\mathfrak S,C)$.
\end{lem}

\begin{proof}
Let $i,j\in I$ and let $x=\gamma(i),y=\gamma(j)$.  Choose $M\geq\max\{C,M_0\}$, where $M_0$ is the constant from 
Theorem~\ref{thm:cubulated_hull}. By Theorem~\ref{thm:cubulated_hull}, there exists $C_1$, depending on $M$, $\mathfrak S$ 
and $\cuco X$, so that there is a CAT(0) cube complex $\cuco C(x,y)$ and a $C_1$--quasimedian $(C_1,C_1)$--quasi-isometric 
embedding $\cuco C(x,y)\to \cuco X$ whose image $C_1$--coarsely coincides with $H_\theta(x,y)$.  Since $\gamma|_{[i,j]}$ is a 
hierarchy path from $x$ to $y$, $\gamma([i,j])$ is coarsely (depending on $D$) contained in $H_\theta(x,y)$ and hence 
coarsely (depending on $C_1,D$) contained in the image of $\cuco C(x,y)$.  On the other hand, the dimension bound from 
Theorem~\ref{thm:cubulated_hull}, the hypothesized property of $C$, and our choice of $M\geq C$ imply that $\dimension\cuco 
C(x,y)\leq 1$.  Moreover, Theorem~\ref{thm:cubulated_hull} implies that $\cuco C(x,y)$ is the convex hull of a set of at most 
two $0$--cubes in $\cuco C(x,y)$, so $\cuco C(x,y)$ is a subdivided interval.  Hence $\gamma([i,j])$ and $H_\theta(x,y)$ 
uniformly coarsely coincide.

Now fix $\epsilon$ and suppose $x\in\cuco X$ has  the property that $\pi_U(x)$ lies $\epsilon$--close to the unparameterized 
$(D,D)$--quasigeodesic $\pi_U(\gamma)$ for each $U\in\mathfrak S$.  Then there exists $i\geq0$ so that $x$ lies 
$\epsilon$--close to the image of $\pi_U\circ\gamma|_{[0,i]}$ for all $U$.  Hence $x$ lies $\kappa$--close to 
$H_\theta(\gamma(0),\gamma(i))$, where $\kappa$ depends only on $\epsilon$ and the quasiconvexity function for hulls of pairs 
of points.  But by the above discussion, this implies that $x$ lies uniformly close to $\gamma([0,j])$, as required.
\end{proof}

In the proof of Theorem~\ref{thmi:main}, we will construct a
quasimedian quasi-isometric embedding of a CAT(0) cube complex into
$\cuco X$. Huang's theorem will provide cubical orthants in the CAT(0) cube complex, so we need to prove that the image of each cubical orthant is coarsely a standard
orthant.  For that, we will use the following lemma:

\begin{lem}\label{lem:quasi_median_top_dimensional}
Let $O$ be an $\nu$--dimensional cubical orthant with a quasimedian
quasi-isometric embedding $q\co O\to\cuco X$.  Then there is a standard
orthant $Q\subset \cuco X$ with $\dist_{haus}(q(O),Q)<\infty$.
\end{lem}

\begin{proof}
Let $\lambda$ be so that $q$ is $\lambda$--quasimedian and a $(\lambda,\lambda)$--quasi-isometric embedding.
 
\textbf{Related points and pairs:} We say that $x,y\in O$ are \emph{$i$--related},  for $1\leq i\leq \nu$, if they only 
differ in the $i^{th}$ coordinate. The $i$--related pairs $x,y$ and $x',y'$ are \emph{$j$--related}, for $i\neq j$, if the 
pairs $x,x'$ and $y,y'$ are $j$--related (i.e., if $x,x',y,y'$ are the vertices of a rectangle in the $(i,j)$--plane).

\textbf{Relevant domains:}  Let $M=M(\lambda,\cuco X)$ be sufficiently large.   For $1\leq i\leq \nu$, let $\mathcal U_i$ be 
the collection of all $U\in\mathfrak S$ so that there exist $i$--related $x,y\in O$ with $\dist_U(q(x),q(y))\geq M$.  For any 
$K$, we also let $\relevant_K(q(O))=\{U\in\mathfrak S:\diam_{\fontact U}(\pi_U(q(O)))\geq K\}$.

We now prove two claims about $i$--related pairs and $\cup_i\mathcal U_i$:

\begin{claim}\label{claim:small}
  There exists $C=C(\lambda,\cuco X)$ so that the following holds. Suppose that the $i$--related pairs $x,y$ and $x',y'$ are $j$--related. Then for any $U\in\mathfrak S$ either
  \begin{itemize}
   \item $\dist_U(x,y)\leq C$ and $\dist_U(x',y')\leq C$, or
   \item $\dist_U(x,x')\leq C$ and $\dist_U(y,y')\leq C$.
  \end{itemize}
 \end{claim}
 
 \begin{proof}[Proof of Claim~\ref{claim:small}]
 Let $m:O^3\to O$ be the median on $O$ coming from the cubical structure (so each cube is an $\ell_1$ $\nu$--cube of unit side length).  We have $m(x',x,y)=x$, so that in each $U\in\mathfrak S$ we have that $\pi_U(x)$ lies uniformly close to geodesics $[\pi_U(x'),\pi_U(y)]$. Similarly, $\pi_U(y')$ lies uniformly close to geodesics $[\pi_U(x'),\pi_U(y)]$. Also, $\pi_U(x')$ and $\pi_U(y)$ lie uniformly close to geodesics $[\pi_U(x),\pi_U(y')]$, forcing the endpoints of $[\pi_U(x'),\pi_U(y)]$ and $[\pi_U(x),\pi_U(y')]$ to be uniformly close in pairs, as required.  
 \end{proof}

 \begin{claim}\label{claim:orth}
 For $M$ sufficiently large, $U\orth V$ whenever $U\in\mathcal U_i, V\in\mathcal U_j$ and $i\neq j$.
\end{claim}
 
 \begin{proof}[Proof of Claim~\ref{claim:orth}]
 Consider distinct $i,j$, an $i$--related pair $x,y$ and some $U$ with
 $\dist_U(q(x),q(y))\geq M$, and a $j$--related pair $w,z$ and some
 $V$ so that $\dist_V(q(w),q(z))\geq M$.  
 
 Provided
 $M\geq 10(\nu-1)C$, applying Claim~\ref{claim:small} at most $\nu-1$ times allows us to change the coordinates of $w,z$ (other than the $j^{th}$) to find  an $i$--related pair $x',y'$ which is $j$--related to $x,y$.  Moreover, we have:
 \begin{itemize}
  \item $\dist_V(q(x),q(x'))\geq M/2$ and $\dist_V(q(y),q(y'))\geq M/2$;
  \item $\dist_U(q(x),q(y))\geq M$ and $\dist_U(q(x'),q(y'))\geq M/2$.
 \end{itemize}

 Claim~\ref{claim:small} implies that $\dist_U(q(x),q(x'))\leq C$, $\dist_U(q(y),q(y'))\leq C$ and $\dist_V(q(x),q(y))\leq C$, $\dist_V(q(x'),q(y'))\leq C$.
 
For $M$ large enough, this implies that $U\orth V$.  Indeed, if $U=V$, then the triangle inequality yields $4C\geq M/2$, a contradiction.  If $U\transverse V$, then  there exists $p\in\{x,x',y,y'\}$ with $\pi_U(p)$ $E$--far from $\rho^V_U$ and $\pi_V(p)$ $E$--far from $\rho^U_V$, contradicting consistency.  A similar contradiction arises if $U,V$ are $\propnest$--comparable.  Hence $U\orth V$, as required.
\end{proof}
 
\textbf{The candidate standard orthant:} Let $\gamma'_i$ be the image of the axis along the $i^{th}$ coordinate in $O$.
Since $q$ is quasimedian and a quasi-isometric
embedding, $\gamma'_i$ is a quasi-geodesic projecting to
unparameterized quasi-geodesics in every $\fontact U$, i.e., it is a
$D'=D'(\lambda)$--hierarchy ray, by Lemma~\ref{lem:quasimedian_hierarchy_ray}.  By \cite[Lemma 3.3]{HHS_boundary},
there exist $U^i_1,\dots,U^i_{k_i}$ so that $\pi_{U^i_j}(\gamma_i')$ is
unbounded.  Moreover, by the same lemma, for $1\leq i\leq\nu,\ 1\leq j<j'\leq k_i$, we have $U_i^j\orth U_i^{j'}$.  

Since each $U^i_j\in\mathcal U_i$, Claim~\ref{claim:orth} and the fact that $\cuco X$ has rank $\nu$ implies that $k_i=1$ for 
each $i$.  To streamline notation, let $U_i=U^i_1$.
 
Since $\{U_1,\ldots,U_\nu\}$ is a pairwise-orthogonal set, the following
holds for all $i\leq \nu$: if $U,V\propnest U_i$ have $\diam(\fontact
U),\diam(\fontact V)>E$, then $U\notorth V$, for otherwise
$\{U_1,\ldots,U_{i-1},U,V,U_{i+1},\ldots,U_\nu\}$ would contradict that
$\cuco X$ is asymphoric.  It follows from
Corollary~\ref{cor:hyperbolic} that each $F_{U_i}$ is hyperbolic.  Hence
there exists a $D''$--hierarchy ray $\gamma_i$ in $F_{U_i}$ so that the distance between $\gamma_i(t)$ and $\gamma'_i(t)$ is uniformly bounded for all $t\in[0,\infty)$.

 The $\gamma_i$ define a standard orthant $Q$
with support $\{U_i\}$.

\textbf{$q(O)$ and $Q$ lie within finite Hausdorff distance:} We claim the following. For $p\in O$ we denote by $p_i$ the point on the $i$--th coordinate axis with the same $i$--th coordinate as $p$. Then there exists $C'$ so that $d_{\fontact U}(q(p),q(p_i))\leq C'$ whenever $U\notin \bigcup_{j\neq i}\mathcal U_j$. This holds because we can find a sequence of at most $\nu$ points, starting with $p$ and ending with $p_i$, so that consecutive elements are $j$--related for $j\neq i$. By definition, if consecutive elements have far away projection to some $\fontact U$, then $U\in \mathcal U_j$ for $j\neq i$.

Now let $p\in O$. By the above claim, $\pi_U(q(p))$ coarsely coincides with $\pi_U(q(p_i))$ if $U\in \mathcal U_i$, and 
otherwise it coarsely coincides with $\pi_U(c)$, where $c$ is the image of the ``corner'' of $O$. We can find points 
$\gamma_i(t_i)$ uniformly close to $q(p_i)\in\gamma'_i$, and the $\gamma_i(t_i)$ define a point $p'$ of $Q$. It is readily 
checked that for every $U$, $\pi_U(q(p))$ coarsely coincides with $\pi_U(p')$, so that $q(p)$ and $p'$ are within uniformly 
bounded distance. This proves that $q(O)$ is contained in a finite radius neighborhood of $Q$. A very similar argument proves 
the other containment.
\end{proof}

\subsection{Coarse intersections of orthants}

In this subsection we study coarse intersections of orthants. This is mostly needed for the next section, but we need Lemma \ref{lem:intersection_of_orthants} in the proof of Theorem \ref{thm:uniform}.

\begin{defn}[Coarse intersection]\label{defn:coarse_intersection}
Let $A,B\subset \cuco X$.  Suppose that there exists $R_0$ so that 
for any $R,R'\geq R_0$, we have $\dist_{haus}(\neb_R(A)\cap\neb_R(B),\neb_{R'}(A)\cap\neb_{R'}(B))<\infty$.  Then we refer to any subspace at finite Hausdorff distance from $\neb_{R_0}(A)\cap\neb_{R_0}(B)$ as the \emph{coarse intersection of $A$ and $B$}, which we denote $A\tilde{\cap}B$.
\end{defn}

In the next lemma, we show that, for pairs of hierarchically quasiconvex sets, an $R_0$ as in the definition above exists, and so the coarse 
intersection is well-defined.  This is one of the places where we use the bridge lemma (Lemma~\ref{lem:parallel_gates}).

\begin{lem}[Coarse intersections coarsely coincide with gates]\label{lem:coarse_intersect_hq}
 For all $\kappa,r$, there exists $R_0$ such that the following holds.  Let $A,B$ be $\kappa$--hierarchically quasiconvex 
and suppose $\dist(A,B)\leq r$.  Then for all $R,R'\geq R_0$,  we have 
$\dist_{haus}(\neb_R(A)\cap\neb_R(B),\neb_{R'}(A)\cap\neb_{R'}(B))<\infty$, so $A \tilde{\cap}B$ is well-defined.  Moreover, 
there exists $K=K(\kappa,r)$ such that $A\tilde\cap B$ is at Hausdorff distance at most $K$ from $\gate_A(B)$.
\end{lem}

\begin{proof}
By Lemma~\ref{lem:parallel_gates}.\eqref{item:bridge}, there exists $K_1$, depending only on $\kappa(0)$ and $E$, and a 
$(K_1,K_1)$--quasi-isometric embedding $f\co \gate_A(B)\times H_\theta(\{a,b\}\to\cuco X$ such that $\image f$ 
$K_1$--coarsely coincides with $H_\theta(\gate_A(B)\cup\gate_B(A))$, where $a=\gate_A(b)\in\gate_A(B)$ and $b=\gate_B(a)$. By Lemma \ref{lem:distance_between_sets} (applied exchanging the roles of $A$ and $B$), $\dist(a,b)$ is bounded in terms of $\kappa$ and $r$. Hence,
there exists $R_1$, depending on $\kappa$, $K_1$, and $r$, 
such that any point in $\gate_A(B)\subseteq A$ lies at distance at most $R_1$ from $\gate_B(A)\subseteq B$, and hence at distance 
at most $R_1$ from $B$.  So, $\gate_A(B)\subset \neb_{R_1}(A)\cap \neb_{R_1}(B)$.  

On the other hand, if $p\in\neb_R(A)\cap\neb_R(B)$ for some $R$, then apply Lemma~\ref{lem:distance_between_sets} to find 
$K=K(\kappa)$ such that $\dist(p,A)\asymp_{K,K}\dist(p,\gate_A(p))$ and 
$\dist(A,B)\asymp_{K,K}\dist(\gate_A(p),\gate_B(\gate_A(p)))$.  So, $\dist(p,\gate_B(\gate_A(p)))\preceq_{K,K} 
 R+r$.  In other words, $\dist(p,\gate_B(A))$ is uniformly bounded (in terms of $\kappa,R$ and $r$) and 
$\dist(p,\gate_A(B))$ is bounded similarly.  So $\neb_R(A)\cap\neb_R(B)$ uniformly coarsely coincides with $\gate_A(B)$, 
proving the second claim.  Since any two neighborhoods of $\gate_A(B)$ coarsely coincide, the first claim follows.
\end{proof}

The following lemma describes coarse intersections of orthants, 
which, as one might hope, turn out to be sub-orthants.

\begin{lem}[Coarse intersections of orthants]\label{lem:intersection_of_orthants}
Let $O,O'$ be standard orthants in $\cuco X$ with supports
$\{U_i\}_{i\leq \nu},\{U'_i\}_{i\leq \nu}$.  Then $O\tilde{\cap}O'$ is
well-defined, and coarsely coincides with $\gate_{O}(O')$, as well as
with a standard $k$--orthant whose support is contained in
$\{U_i\}_{i\leq \nu}\cap\{U'_i\}_{i\leq \nu}$.
\end{lem}

\begin{proof}
By Lemma \ref{lem:coarse_intersect_hq}, we only need to show that
$\gate_O(O')$ coarsely coincides with a standard $k$--orthant whose
support is contained in $\{U_i\}\cap\{U'_i\}$.

Let $\gamma_i$ be the hierarchy ray in $F_{U_i}$ participating in $O$,
and similarly for $\gamma'_i$ and $O'$.  Let $\{V_j\}_{j=1,\dots,k}$
be the set of all $V_j=U_i=U'_{i'}$ so that $\gamma_i$ and
$\gamma'_{i'}$ lie within bounded Hausdorff distance, in which case
set $\alpha_j=\gamma_i$.  Let $O''$ be a standard $k$--orthant
contained in $O$ with support set $\{V_j\}$ defined by the $\alpha_j$.
We claim that $O''$ represents $O\tilde{\cap}O'$.

By Lemma~\ref{lem:orthant_quasiconvex}, $O''$ is hierarchically quasiconvex, and $G=\gate_{O}(O')$ is hierarchically quasiconvex by Lemma~\ref{lem:parallel_gates}.\eqref{item:gate_hq}.  We claim that $O''$ coarsely coincides with $G$.  Since they are hierarchically quasiconvex, we only need to argue that their projections to each $\fontact U$ coarsely coincide. 

By Remark~\ref{rem:bounded_proj}, for each $U$, $\pi_U(O'')$ coarsely coincides with some $\pi_U(\alpha_j)$.  
In particular, if $U$ is not nested in some $U_j$, then $\pi_U(O'')$ uniformly coarsely coincides with  
each $\pi_U(\alpha_j(0))$.  

Also, $\pi_U(G)$ coarsely coincides with the projection of a single $\gamma_i$, if $\gamma_i=\alpha_j$ for some 
$j$. Otherwise $\pi_U(G)$ coarsely coincides with $\pi_U(\alpha_j(0))$ for each $j$.  Hence $\pi_U(G)$ and 
$\pi_U(O'')$ coarsely coincide for all $U$.
\end{proof}

In the proof of Theorem~\ref{thm:clique_map} below, we will need the following version of the above lemma, 
stated for coarse intersections of standard flats instead of standard orthants.

\begin{lem}[Coarse intersection of standard flats]\label{lem:standard_flat_intersection}
 Let $F,F'$ be standard flats in $\cuco X$ with supports $\{U_i\}_{i=1}^\nu$ and $\{U_i'\}_{i=1}^\nu$ 
respectively.  Then $F\tilde\cap F'$ is well-defined, and coarsely coincides with $\gate_F(F')$.  Moreover, 
suppose that $\{U_i\}\cap\{U_i'\}=\{U\}$ for some $U\in\mathfrak S$.  Then $F\tilde\cap F'$ is either 
a bounded set or coarsely coincides with a standard $1$--orthant or standard $1$--flat with support $\{U\}$.

Similarly, if $\{U_i\}\cap\{U_i'\}=\{U,V\}$ for some (necessarily orthogonal) $U,V\in\mathfrak S$, then 
$F\tilde\cap F'$ coarsely decomposes as the product of two hierarchically quasiconvex subspaces $\alpha,\beta$, 
each of which is either bounded or coarsely coincides with a standard $1$--orthant or standard $1$--flat.
\end{lem}

\begin{proof}
The standard flats $F,F'$ are uniformly hierarchically quasiconvex by Remark~\ref{rem:flats_quasiconvex}.  
Lemma~\ref{lem:coarse_intersect_hq} implies that $F\tilde\cap F'$ is well-defined and coarsely coincides with 
$\gate_F(F')$.  So, we just need to show that $\gate_F(F')$ is a standard $1$--orthant or $1$--flat with 
support $\{U\}$, or $\gate_F(F')$ is bounded.  For each $i\leq\nu$, let $\gamma_i$ be the hierarchy path in 
$F_{U_i}$ which is the $i^{th}$ factor of $F$, and define $\gamma'_i$ analogously for $F'$.  Re-labeling if 
necessary, let $U=U_1=U_1'$.  Note that by Lemma~\ref{lem:hq_hierarchy_path}, each $\gamma_i,\gamma_i'$ is 
uniformly 
hierarchically quasiconvex.  Indeed, $\pi_V(\gamma_i)$ has uniformly bounded diameter unless $V\nest U_i$.  
But if $V,W\nest U_i$ are orthogonal, then $\{V,W\}\cup\{U_j\}_{j\neq i}$ is a pairwise-orthogonal set of 
$\nu+1$ elements, so by asymphoricity, $\fontact V$ (say) has diameter at most $E$, so the same is true of 
$\pi_V(\gamma_i)$.  Hence Lemma~\ref{lem:hq_hierarchy_path} applies.

Let $\alpha=\gate_{\gamma_1}(\gamma_1')$.  Arguing as in the proof of Lemma~\ref{lem:intersection_of_orthants} 
shows that $\alpha$, which coarsely coincides with $F\tilde\cap F'$, is either bounded or 
coarsely coincides with a $1$--orthant or $1$--flat.  This proves the first assertion.

The second assertion follows similarly.  Again, $F\tilde\cap F'$ coarsely coincides with $\gate_F(F')$ by 
Lemma~\ref{lem:coarse_intersect_hq}, so it suffices to show that $\gate_F(F')$ coarsely coincides with a 
product $\alpha\times\beta$ as in the statement.  Label the supports of $F,F'$ so that $U=U_1=U_1'$ and 
$V=U_2=U_2'$.  Let $\alpha=\gate_{\gamma_1}(\gamma_1')$ and let $\beta=\gate_{\gamma_2}(\gamma_2')$.  Then 
argue as in Lemma~\ref{lem:intersection_of_orthants} to see that $\alpha\times\beta$ coarsely coincides with 
$\gate_F(F')$.
\end{proof}

\subsection{Quasiflats theorem}\label{subsec:quasiflat_proof}
We are now ready to prove Theorem~\ref{thmi:main}, which we restate as:

\begin{thm}\label{thm:non_uniform}
Let $\cuco X$ be an asymphoric HHS of rank $\nu$ and let
$f\co\reals^\nu\to\cuco X$ be a quasi-isometric embedding.  Then there
exists a finite set of standard orthants $Q_i\subseteq \cuco X$ for 
$1\leq i\leq k$, for which:
$$\dist_{haus}(f(\reals^\nu),\cup_{i=1}^kQ_i)<\infty.$$
\end{thm}

\begin{proof}
 Let $L,N$ be as in Corollary~\ref{cor:fixed_flat_ball_hull}. Then there exists an increasing unbounded sequence 
$R_1<R_2<\ldots$ and sets $A_i\subseteq \cuco X$ of cardinality at 
most $N$ for which the following holds. Let $B_i$ be
 the ball in $\reals^\nu$ of radius $R_i$ centered at a fixed
 basepoint, and let $H_i=H_\theta(A_i)$.
 Then $f(B_i)\subseteq\neb_L(H_i)$.  Let $c_i\co\cuco Y_i\to H_i$ be
 the $C$--quasimedian $(C,C)$--quasi-isometry provided by
 Theorem~\ref{thm:cubulated_hull}, so $\cuco Y_i$ is a CAT(0) cube
 complex of dimension $\leq \nu$ and the constant $C$ depends on $N$. 
 
 Now we pass to (non-rescaled!) ultralimits\footnote{If $\cuco X$ is 
 proper, one can take Hausdorff limits instead. To avoid that 
 assumption, we use ultralimits instead. If $\cuco X$ is 
 not proper then $\widehat{\cuco X}$ is (much) bigger than $\cuco X$. }. 
 More specifically, $f$ has an ultralimit which is a 
 $(K,K)$--quasi-isometric embedding $\hat{f}\co\reals^\nu\to \widehat{\cuco X}$, for some ultralimit $\widehat{\cuco X}$ of $\cuco X$.  It is easily deduced from Corollary~\ref{cor:rank_coarse_median} that $\widehat{\cuco X}$ is a coarse median space and we have the following: there is a CAT(0) cube complex $\hat{\cuco Y}$, an ultralimit of the $\cuco Y_i$, endowed with a $C$--quasimedian $(C,C)$--quasi-isometry $\hat{c}:\hat{\cuco Y}\to \widehat{\cuco X}$ so that the image of $\hat{f}$ lies in the $L$--neighborhood of $\image(\hat{c})$.

 By a theorem of Huang --- Theorem~1.1 of~\cite{Huang:quasiflats} --- there exist
 $n$--dimensional cubical orthants $O_1,\ldots,O_k$ in $\hat{\cuco Y}$
 so that $\dist_{haus}(\hat
 f(\reals^\nu),\hat{c}(\cup_{j=1}^kO_j))<\infty$.  Moreover,
 $\hat{c}(O_j)$ lies within finite Hausdorff distance of
 $\hat{f}(O'_j)$ for some $O'_j\subseteq \reals^\nu$.  Hence,
 $Q_j=f(O'_j)$ is the image of a $C'$--quasimedian
 $(C',C')$--quasi-isometric embedding. Thus, by Lemma
 \ref{lem:quasi_median_top_dimensional}, it lies within finite
 Hausdorff distance of a standard orthant.  The $Q_i$ are as required.
\end{proof}

\subsection{Controlled number of orthants}

We now improve Theorem \ref{thm:non_uniform}, by showing that the number of standard orthants required can be bounded in terms of the quasi-isometry constants:

\begin{thm}[Bounding the number of orthants]\label{thm:uniform}
 Let $\cuco X$ be an asymphoric HHS of rank $\nu$.  For every $K$ there
 exists $N$ so that the following holds.  Let $f\co\reals^\nu\to \cuco X$
 be a $(K,K)$--quasi-isometric embedding.  Then there exist standard
 orthants $Q_i\subseteq \cuco X$, $i=1,\dots,N$, so that
 $\dist_{haus}(f(\reals^\nu),\cup_{i=1}^NQ_i)<\infty$.
\end{thm}

The idea of the proof is as follows.  First, by the above we have that  $f(\reals^\nu)$ lies Hausdorff-close to a finite union 
of standard orthants $O_1,\ldots,O_k$.  Now, each $O_i$ makes a definite contribution to the volume growth in the quasiflat 
$f(\reals^\nu)$, and this growth is in turn bounded by the quasi-isometry constants.  So, $k$ must be bounded.  This is 
formalized in Proposition~\ref{prop:growth}.  

First, we need the following lemma, which is a slightly stronger version of the well-known fact that quasi-isometric 
embeddings of $\mathbb R^n$ into itself are coarsely surjective, see \cite[Corollary 2.6]{KapovichLeeb:haken}.

\begin{lem}\label{lem:Borsuk_Ulam}
 For every $K,n\geq1$ there exists $C$ so that the following holds.  Let
 $f\co\reals^n\to\reals^n$ be a $(K,K)$--coarsely Lipschitz proper map.
 Then $\dist_{haus}(f(\reals^n),\reals^n)\leq C$.
\end{lem}

\begin{proof}
 We actually show that if $f\co\reals^n\to\reals^n$ is continuous and proper, then $f$ is surjective, and the lemma follows 
from the fact that $f$ can be approximated by a continuous map.
 
 Since $f$ is proper, it extends to a continuous map
 $\overline{f}:\overline{\reals^n}\to \overline{\reals^n}$ between two
 copies of the $1$--point compactification $\overline{\reals^n}$ of
 $\reals^n$, which is homeomorphic to the sphere $S^n$.  Also, it is
 easily seen that we can identify the domain 
 $\overline{\reals^n}$ with $S^n$ in such a way that, since $f$ is 
 coarsely Lipschitz, no pair of
 antipodal points have the same image. But then $\overline{f}$ must
 be surjective, for otherwise the Borsuk-Ulam theorem would force the
 existence of such pair of antipodal points.  Since $\overline{f}$ is
 surjective, then so is $f$, as required.
\end{proof}

\begin{prop}[Volume growth]\label{prop:growth}
 For every $K$ there exists $N$ so that the following holds.  Let $F\co\reals^\nu\to\cuco X$ be a $(K,K)$--quasi-isometric 
embedding whose image lies at finite Hausdorff distance from $\bigcup_{i=1}^kO_i$, where each $O_i$ is a standard orthant. If 
$\dist_{haus}(O_i,O_j)=\infty$ when $i\neq j$, then $ k\leq N$.
\end{prop}

\begin{proof}
 The idea of the proof is that each of the $k$ orthants contributes at 
 least $\epsilon R^\nu$ volume growth to $F(\reals^\nu)$, but the volume growth
 of $F(\reals^\nu)$ is bounded above by $R^\nu$ times a (large) constant
 depending on $K$.

 Let $D=\dist_{haus}(F(\reals^\nu),\bigcup_{i=1}^k O_i)$.  By 
 Lemma~\ref{lem:intersection_of_orthants}, 
 since the $O_{i}$ are pairwise at infinite Hausdorff distance, for each $i$ 
 we can find a sub-orthant $O'_i\subset O_i$ so that for each $i,j$, 
 $\dist(O'_i,O'_j)\geq 2D+1$. We will identify $O'_i$ with 
 $[0,\infty)^\nu$.
 
 Let $A_i\subseteq \reals^\nu$ be the set of points whose image under  
 $F$ is at
 distance at most $D$ from $O'_i$.  Note that the $A_i$ are
 disjoint.  
  
  Let $g_i$ be the composition of $F$ and the gate map to $O'_i$; the 
  map $g_i$ is $(K',K')$--coarsely  Lipschitz for some $K'=K'(K,\cuco X)$, and it is a quasi-isometric embedding with constant depending on $K,\cuco X$, and $D$ (this dependence on $D$ is the reason why we need Lemma \ref{lem:Borsuk_Ulam} dealing with proper maps). Up to increasing $K'$, we can further assume that there is a $(K',K')$--quasi-isometry from $O'_i$ to an orthant in $\reals^\nu$ so that the composition of $g_i$ and the quasi-isometry is also $(K',K')$--coarsely Lipschitz.
  
  Notice that for each $R$ and $i$, there exists a
   sub-orthant $O^R_i \subset O'_i$ so that if $x\in A_i$ has $g_i(x)\in O^R_i$, then
   $B_R(x)\subseteq A_i$.

  Let $C$ be as in Lemma
  \ref{lem:Borsuk_Ulam} for $K'$, and set $C_1=K'C+(K')^2$.  Since the orthants $O^R_i$ are quasi-geodesic spaces with constant depending on $\cuco X$ only, up to increasing $C$ we can assume the following. Suppose that we have a subset $A\subseteq  O^R_i$, for some $R,i$, with the property that $O^R_i\nsubseteq N_{C_1}(A)$. Then there exists $x\in O^R_i$ so that $\dist(x,A)\leq 2C_1-1$ but $\dist(x,A)> C_1$.

%
%
  
\textbf{A further sub-orthant:}  We claim that for each $i$, there is a sub-orthant  $O''_i\subset O'_i$ with the 
property that $O''_i\subset N_{C_1}(g_i(A_i))\cap O'_i$.

Let $n\in\naturals$.  Let $O^n_i$ be the sub-orthant of $O'_i$ defined above, which 
has the property that for all $x\in A_i$ with $g_i(x)\in O^n_i$, we have $B_n(x)\subset A_i$.  If the 
sub-orthant $O''_i$ with the claimed property does not exist, then, for each $n$, there exist $p_n\in A_i$ and $x_n\in O^n_i$ such that the following hold:
\begin{itemize}
\item $g_i(p_n)\in O^n_i$;

     \item $\dist(x_n,g_i(p_n))\leq 2C_1$;
     \item $\dist(x_n,g_i(A_i))> C_1$.
\end{itemize}

In fact, we can choose any $x_n\in O^n_i$ with $\dist(x_n,g_i(A_i))\leq 2C_1-1$ but $\dist(x_n,g_i(A_i))> C_1$, and then pick $p_n\in A_i$ ``nearly witnessing'' the first inequality, meaning $p_n$ so that $\dist(x_n,g_i(p_n))\leq 2C_1$.

  Now, consider the (non-rescaled!)  ultralimit $\cuco R$ of
  $\reals^\nu$ with observation point $(p_n)$, which is isometric to
  $\reals^\nu$.  The process of taking ultralimits induces a
  $(K',K')$--coarsely Lipschitz map $f$ from $\cuco R$ to an
  ultralimit of $O'_i$ with observation point $(x_n)$. Moreover, $f$ is proper since $g_i$ is a quasi-isometric embedding. Being an
  ultralimit of orthants, this ultralimit admits a quasi-isometry, $h$,
  to a subspace of $\reals^\nu$,  with
  constants depending only on $\cuco X$ 
  (actually, the ultralimit of the orthants is quasi-isometric to $\reals^\nu$, but we do not need
  this).  In fact, by the choice of $K'$, we can choose a $(K',K')$--quasi-isometry $h$ as above in such a way that $h\circ f$ 
  is $(K',K')$--coarsely
  Lipschitz,  and notice that $h\circ f$ is still proper.  However, by construction the map $f$ is not 
  $C_1$--coarsely surjective, and thus 
  $h\circ f$ is not 
  $C$--coarsely surjective, 
  contradicting Lemma \ref{lem:Borsuk_Ulam} and thus verifying the
  claim.
  
\textbf{Conclusion:} We now bound from below $\beta_R=|\{x\in\mathbb Z^\nu: F(x)\in
 B_R(F(0))\}|$.  There exists $t=t(K)$ so that
 $\beta_R\leq t R^\nu$.  Let $C'=C'(C,\nu,K)$ satisfy $O''_i\subset
 N_{C'}(g_i(A_i\cap \mathbb Z^\nu))\cap O'_i$.  Consider a maximal
 $(2C'+1)$--net $N_i$ in $O''_i$ and, for any point $p$ of the net, 
 choose some $q\in A_i\cap \mathbb Z^\nu$ with $\dist(p,F(q))\leq C'$.
 Distinct $p$ yield distinct $q$.  Moreover, $ |N_i\cap B_R(F(0))|\geq
 t' R^\nu$ for all sufficiently large $R$ and some $t'=t'(C',\cuco X)$.
 Since the $A_i$ are disjoint, we have $\beta_R\geq k t'R^\nu$ for all
 sufficiently large $R$.  Hence $k\leq t/t'$, and we are done.
\end{proof}

\begin{proof}[Proof of Theorem \ref{thm:uniform}]
 By Theorem \ref{thm:non_uniform}, the image of $F$ lies at finite Hausdorff distance from a union of orthants $\bigcup_{i=1}^k O_i$. We can assume that $\dist_{haus}(O_i,O_j)=\infty$ when $i\neq j$; indeed, if not, then we can drop $O_i$ or $O_j$ from the collection without affecting the conclusion. Hence, $k\leq N$, for $N$ as in Proposition \ref{prop:growth}.
\end{proof}

\subsection{Controlled distance}
As in the cubical case, it is not possible in general to give an effective bound on the Hausdorff distance between a quasiflat and the corresponding union of orthants. However, we have the following:

\begin{lem}\label{lem:flats_close_to_hulls}
For every $K,N$ there exists $L$ so that the following holds. Let 
$F\co\reals^\nu\to\cuco X$ be a $(K,K)$--quasi-isometric embedding whose image lies at finite Hausdorff distance from $\bigcup_{i=1}^NO_i$, where each $O_i$ is a standard orthant.  Then $F\subset\neb_L(H_\theta(\bigcup_{i=1}^NO_i))$.
\end{lem}

\begin{proof}
Let $F$ and $O_i$ be as in the statement.  Any bounded set in $O_i$
lies in a uniform neighborhood of the hull of the ``corner point'' of
$O_i$ and some point along the diagonal.  Hence, there exists $D$ so
that any ball $B$ in $\reals^n$ has the property that $F(B)$ is
contained in the $D$--neighborhood of $H_\theta(A)$ for some
$A\subseteq \bigcup_iO_i$ with $|A|\leq 2N$.  For $L$ as in
Proposition \ref{prop:thanks_Mauro}, there exist arbitrarily large
balls $B'$ in $\reals^\nu$ so that $F(B')\subseteq
\neb_L(H_\theta(A))\subseteq \neb_L(H_\theta(\bigcup_{i=1}^NO_i))$ for
some $A\subseteq \bigcup_iO_i$.  Hence, the same holds for $\reals^\nu$,
as required.
\end{proof}

\begin{cor}\label{cor:flats_close_to_hulls}
For each $K$ there exists $L,N$ so that the following holds.  Let 
$F\co\reals^\nu\to\cuco X$ be a $(K,K)$--quasi-isometric embedding.  Then there exist standard orthants $O_1,\ldots,O_N$ so that  $F\subset\neb_L(H_\theta(\bigcup_{i=1}^NO_i))$.
\end{cor}

\begin{proof}
Follows immediately from Theorem~\ref{thm:uniform} and Lemma~\ref{lem:flats_close_to_hulls}.
\end{proof}

\section{Induced maps on hinges: mapping class group rigidity}\label{sec:standard_flat}

Let $(\cuco X,\mathfrak S)$ be an HHS. We have in mind the case where $\cuco X$ is the Cayley graph of the 
mapping class group of a finite-type surface, equipped with the HHS structure from~\cite[Section 11]{hhs2}.

In this section, we provide a new proof of quasi-isometric rigidity of mapping class groups. More generally, we study intersection 
patterns of quasiflats in $\cuco X$ and, under favorable conditions, extract suitable ``combinatorial data'' from it.

In the rest of this section, we will abstract from the mapping class group to the greatest extent permitted by our 
methods.  We will need $(\cuco X,\mathfrak S)$ to be asymphoric, 
which, as previously noted, is a weak assumption.  We will also 
impose three additional, 
more restrictive,  assumptions on $(\cuco X,\mathfrak S)$, which are satisfied by the
standard HHS structure on the mapping class group. First, we introduce a few relevant definitions and state 
the additional assumptions.  Then, we discuss the generality in which these assumptions hold.  

The next definition describes those subsets of $\frak S$ which give 
rise to standard flats (as defined in 
Definition~\ref{defn:standard_orthant}).
\begin{defn}[Complete support set]\label{defn:support_set}
A \emph{complete support set} is a subset $\{U_i\}_{i=1}^\nu\subset\mathfrak S$ whose elements are pairwise orthogonal and 
satisfy $\diam(\fontact U_i)=\infty$ for all $i\leq \nu$. 
\end{defn}

For each $U\in\mathfrak S$, we let $\boundary\fontact U$ denote the Gromov boundary of $\fontact U$.

Note that a complete support set $\{U_{i}\}$ and a 
pair of distinct points $\{p_{i}^{\pm}\}\in\boundary \fontact U_{i}$ for each $i$, allows one to 
construct a standard flat, $\calF_{\{(U_{i},p_{i}^{\pm})\}}$ associated to some choice of bi-infinite
hierarchy paths in each $F_{U_{i}}$ whose projection to $\fontact U_i$ has limit points $\{p_{i}^{\pm}\}$ in $\fontact U_i$. 

Accordingly, it is easy to verify that a complete support set is the 
support set of some standard flat if and only if each $\boundary 
\fontact U_{i}$ contains at least two points.

\begin{defn}[Hinge, orthogonal hinges]\label{defn:hinge}
A \emph{hinge} is a pair $(U,p)$ with:
\begin{itemize}
 \item $U\in\mathfrak S$;
 \item $U$ belongs to some complete support
set; and,
\item $p\in\boundary\fontact U$.
\end{itemize}
Let $\hinges(\mathfrak S)$ be the set of hinges.  We say $(U,p),(V,q)\in\hinges(\mathfrak S)$ are \emph{orthogonal} if $U\orth V$.
\end{defn}

\begin{defn}[Ray associated to a hinge]\label{defn:hinge_ray}
For any $\mu\ge0$, a \emph{$\mu$--ray associated to} a hinge $\sigma=(U,p)$ is a $\mu$--hierarchy path $\mathfrak h_\sigma$ 
so that $\pi_U(\mathfrak h_\sigma)$ is a quasigeodesic ray representing $p$ and so that $\diam(\pi_V(\mathfrak 
h_\sigma))\leq\mu$ for $V\neq U$.
\end{defn}

\begin{rem}\label{rem:hinge_unique} Any two candidates for $\mathfrak
h_\sigma$ lie at finite Hausdorff distance, so for our purposes an
arbitrary choice is fine.  If $\sigma\neq\sigma'\in\hinges(\mathfrak S)$, then 
$\dist_{haus}(\mathfrak h_\sigma,\mathfrak h_{\sigma'})=\infty$.
\end{rem}

\begin{rem} Each hinge corresponds to a $0$--simplex in the HHS boundary $\boundary\cuco X$; see~\cite{HHS_boundary}.

\end{rem}

The first additional assumption holds, for example, in any hierarchically hyperbolic group for which the product regions 
$P_U$ can be taken to be subgroups. This is the case for all 
naturally occurring hierarchically hyperbolic structures on groups of 
which we are aware. However, there are some pathological structures, even 
on a free group, where the assumption fails.

\begin{assumption}\label{assumption:line}
For every $U\in\mathfrak S$, either $\diam(\fontact U)\leq E$ or $|\boundary\fontact U|\geq2$ has at least two points at infinity.
\end{assumption}

\begin{rem}In what follows, we could replace
Assumption~\ref{assumption:line} with: for each $U\in\mathfrak S$
which is the first coordinate of some hinge, $|\boundary\fontact
U|\geq2$.  Equivalently, each $U\in\mathfrak S$ which is the first
coordinate of some hinge is the first coordinate of at least two
hinges.\end{rem}

The second assumption roughly says that, if a standard 1--flat  
is contained in some standard flat, then it can be realized as the intersection of a pair of standard flats.

\begin{assumption}\label{assumption:intersect}
For every $U$ contained in a complete support set there exist complete support sets $\mathcal U_1,\mathcal U_2$ with $\{U\}=\mathcal U_1\cap \mathcal U_2$. 
\end{assumption}

The third assumption is a two-dimensional version of the second one; this 
assumption says that if a standard 2--flat is contained in a standard 
flat, then it can be obtained as the intersection of some 
pair of standard flats. 
 \begin{assumption}\label{assumption:extend}
  If $\nu>2$, then for every
 $U,V$, with each contained in a complete support 
 set and with $U\orth V$, there
 exist complete support sets $\mathcal U_1,\mathcal U_2$ with
 $\{U,V\}=\mathcal U_1\cap \mathcal U_2$.
 \end{assumption}
 
 \begin{remi}
The three preceding assumptions are, taken together, fairly
restrictive.  The first, as we said, is very general and 
holds for all ``interesting''
HHGs, including: mapping class groups, 
$3$--manifolds groups which are HHG, 
all groups acting geometrically on CAT(0) cube
complexes with \emph{factor systems} (see~\cite{hhs1,HagenSusse}) (a 
class which includes all compact special groups in the sense of
Haglund--Wise~\cite{HaglundWise}, and in particular all right-angled
Artin and Coxeter groups), etc. More generally, this first 
assumption also holds for a number of interesting HHSs as well.  These include Teichm\"uller spaces 
with the Weil-Petersson metric. On 
the other hand, this condition fails to hold for the HHS structure on a 
Teichm\"uller space endowed with the Teichm\"uller metric, since in such structure certain $\fontact U$ are (isometric to) horoballs in the hyperbolic plane, and thus have a single 
point as their boundary.

To see why the second condition is more restrictive, consider a 
right-angled Artin group $A_\Gamma$ presented by a finite simplicial
graph $\Gamma$.  There are two ``standard'' HHS structures
(see~\cite[Section 8]{hhs1} for more details), but for our purposes,
we take the one described in the introduction.  The second condition
implies that for each vertex $v\in\Gamma$ that is contained in a
maximal clique, there are two maximal cliques whose intersection is
$v$.  One can articulate a similar combinatorial condition on
right-angled Coxeter groups.  So, for example, the results in this
section do not immediately improve upon, or even recover, Huang's
results on quasi-isometric rigidity for right-angled Artin groups.

The third condition in the right angled Artin group 
case can similarly be interpreted as a combinatorial 
constraint on the intersection pattern of cliques in the presentation 
graph.
 \end{remi}

 In the following theorem we show that, under the additional assumptions stated above, quasi-isometries between HHSs naturally induce (orthogonality-preserving) bijections between corresponding sets of hinges. We think of such bijections as ``combinatorial data'' that we extract from the quasi-isometry.
 The proof relies on studying coarse-intersection patterns of orthants.

\begin{thm}\label{thm:clique_map}
Let $(\cuco X,\mathfrak S)$, $(\cuco Y,\mathfrak T)$ be asymphoric HHS satisfying 
assumptions (\ref{assumption:line}), (\ref{assumption:intersect}) and 
(\ref{assumption:extend}).  For any quasi-isometry $f\co\cuco X\to \cuco Y$,
there exists a bijection $f^\sharp\co\hinges(\mathfrak
S)\to\hinges(\mathfrak T)$ satisfying:
\begin{itemize}
\item $f^\sharp$ preserves orthogonality of hinges;
 \item for all $\sigma\in\hinges(\mathfrak S)$, we have $\dist_{haus}(\mathfrak h_{f^\sharp(\sigma)},f(\mathfrak h_\sigma))<\infty$.
\end{itemize}
\end{thm}

\begin{rem}
Under suitable conditions, we expect that there exists an analogue of
Theorem~\ref{thm:clique_map} in which hinges are replaced by sets of
pairs $\{(U_i,p_i)\}$, where $\{U_i\}_i$ is a pairwise orthogonal set
and $p_i\in\boundary\fontact U_i$.  In particular, one should be able
to show in this way that isolated flats are taken close to isolated
flats.  More strongly, one could consider the situation where flats
coarsely intersect in subspaces of codimension $\geq2$, as
in~\cite{Alex:random_self_citation}.
\end{rem}

\begin{proof}[Proof of Theorem \ref{thm:clique_map}]
Let $\sigma=(U,p)\in\hinges(\mathfrak S)$. 

\textbf{How we will define $f^\sharp$:}  We will produce a hinge $\sigma'$ so that $\dist_{haus}(\mathfrak h_{\sigma'},f(\mathfrak h_\sigma))<\infty$.  Remark~\ref{rem:hinge_unique} implies that $\sigma'$ is uniquely determined by this property, so we can set $f^\sharp(\sigma)=\sigma'$.  To see that this is a bijection, let $\bar f:\cuco Y\to\cuco X$ be a quasi-inverse of $f$.  Then $\dist_{haus}(\bar f(h_{\sigma'}),h_\sigma)<\infty$, so we can define an inverse for $f^\sharp$ in the same way.

\textbf{Choosing $\sigma'$:} Since $(U,p)$ is a hinge, $U$ is in a complete support set.

Notice that, by Assumption~\ref{assumption:line}, for any complete support set $\{U_i\}_i$ we have $|\boundary\fontact U_{i}|\geq 2$ for each $i$, and hence there exists a standard flat $\mathcal F$ with support $\{U_i\}_i$. 

In view of this, Assumption~\ref{assumption:intersect} provides two standard flats $\calF_1,\calF_2$, the intersection of whose 
support sets is $\{U\}$. Furthermore, we claim that we can arrange that 
$\calF_1\tilde\cap\calF_2$ is coarsely a line and coarsely contains $\mathfrak h_\sigma$.  This can done as follows. Consider any hinge $(U,q)$ with $q\neq p$ (which exists by Assumption~\ref{assumption:line}).  Assumption~\ref{assumption:intersect} provides 
complete support sets $\{V_j\},\{V_j'\}$ whose intersection is $\{U\}$.  Label so that $V_1=V_1'=U$.  So, 
by choosing distinct points $a_j,b_j\in\boundary\fontact V_j$ and $a_j',b_j'\in\boundary\fontact V_j'$ in such 
a way that $a_1=a_1'=p$ and $b_1=b_1'=q$, we obtain standard flats $\mathcal F_1,\mathcal F_2$ with the 
given support sets and, by Lemma~\ref{lem:standard_flat_intersection}, coarse intersection which is either bounded, a standard $1$--orthant, or a standard $1$--flat supported on $U$. Both flats coarsely contain a standard $1$--flat (with ``limit points'' $p$ and $q$), so the last case must hold. Moreover, the aforementioned standard $1$--flat coarsely contains $\mathfrak h_\sigma$ by Remark \ref{rem:hinge_unique}.

By Theorem 
\ref{thm:non_uniform} (Quasiflats Theorem), $f(\calF_1)$ and $f(\calF_2)$ are coarsely equal to unions of finitely many standard 
orthants.  Hence $f(\calF_1)\tilde\cap f(\calF_2)$ has the following three properties:

\begin{itemize}
 \item $f(\calF_1)\tilde\cap f(\calF_2)$ is coarsely a finite union of coarse intersections of pairs of
standard orthants.  Indeed, $f(\calF_1)$ coarsely coincides with $\bigcup_sO_s$ and $f(\calF_2)$ coarsely 
coincides with $\bigcup_tO'_t$, where $O_s,O'_t$ are standard orthants provided by 
Theorem~\ref{thm:non_uniform}.  Hence $f(\calF_1)\tilde\cap f(\calF_2)$ coarsely coincides with 
$\bigcup_{s,t}O_s\tilde\cap O_t'$.
 \item $f(\calF_1)\tilde\cap f(\calF_2)$ is coarsely $\reals$, because $\calF_1\tilde\cap\calF_2$ was coarsely 
$\reals$.
 \item $f(\calF_1)\tilde\cap f(\calF_2)$ coarsely contains $f(\mathfrak h_\sigma)$, because 
$\calF_1\tilde\cap\calF_2$ coarsely contained $\mathfrak h_\sigma$. 
\end{itemize}

By Lemma \ref{lem:intersection_of_orthants} and the first of the above
properties, $f(\calF_1)\tilde\cap f(\calF_2)$ is coarsely the finite union of
standard $k$--orthants, which arise as coarse intersections of pairs of standard orthants.
Hence, one of these pairs gives a $1$--orthant (in particular,
a copy of $\reals_{+}$) which coarsely coincides with $f(\mathfrak
h_\sigma)$.

Let $\sigma'$ be the hinge $(V,q)$, where $V$ is the domain of the orthant just determined and $q$ is the unique point in $\boundary V$ determined by the fact that $f(\mathfrak h_\sigma)$ projects to a quasi-geodesic ray in $\fontact V$.  Then $\sigma'$ is the hinge uniquely determined by $f(\mathfrak h_\sigma)$, as required.

\textbf{Preservation of orthogonality:}  Let
$\sigma,\sigma'$ be orthogonal hinges.  Assumption~\ref{assumption:extend} and 
Lemma~\ref{lem:standard_flat_intersection} provide a standard $2$--flat,
$\calF$, coarsely containing $\mathfrak h_\sigma$ and $\mathfrak
h_{\sigma'}$.  Moreover, $\calF$ coarsely coincides with $\calF_1\tilde\cap\calF_2$, for standard flats $\calF_1,\calF_2$.

Hence $f(\calF_1)\tilde\cap f(\calF_2)$ is a $2$--dimensional quasiflat.  On the other hand, by 
Theorem~\ref{thm:non_uniform},  $f(\calF_1)\tilde\cap f(\calF_2)$ is coarsely the union of finitely many 
coarse intersections of pairs of standard orthants. Lemma~\ref{lem:intersection_of_orthants} shows that each 
of these intersections is coarsely a standard $k$--orthant for $k\geq 2$. Since 
$f(\calF_1)\tilde\cap f(\calF_2)$ is a $2$--dimensional quasiflat, we can discard any of the above 
intersections which is coarsely a $0$--orthant or $1$--orthant.  In other words, $f(\calF_1)\tilde\cap 
f(\calF_2)$ is coarsely the union of disjoint standard $2$--orthants $O_0,\ldots,O_{t-1}$.  Moreover, 
$\mathfrak h_{f^\sharp(\sigma)}$ and $\mathfrak h_{f^\sharp(\sigma')}$ coarsely coincide with coordinate rays 
of some $O_i,O_j$.

\begin{figure}[h]
\begin{overpic}[width=0.5\textwidth]{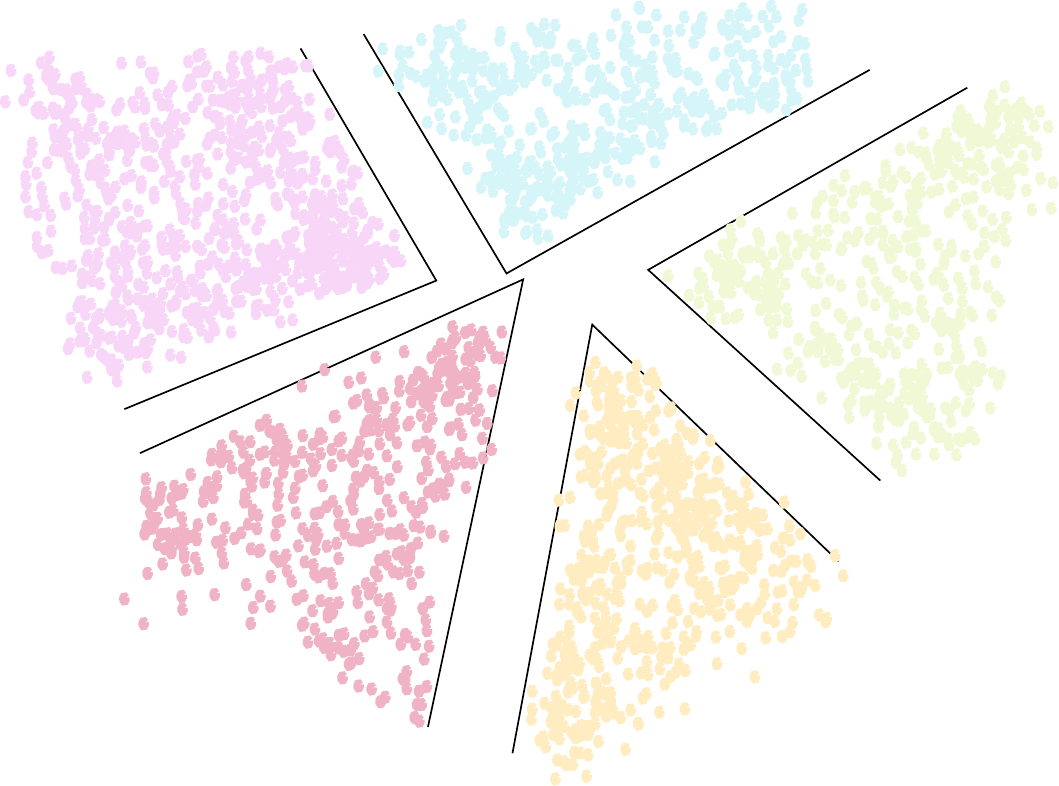}
\put(22,55){$O_0$}
\put(47,60){$O_1$}
\put(72,48){$O_2$}
\put(57,20){$O_3$}
\put(27,20){$O_4$}
\put(38,54){$r^+_0$}
\put(67,56){$r^+_1$}
\put(72,30){$r^+_2$}
\put(45,20){$r^+_3$}
\put(28,40){$r^+_4$}
\end{overpic}
\caption{The $2$--orthants $O_0,\ldots,O_t$ and the cyclic ordering of their coordinate rays (up to coarse coincidence).}\label{fig:ordering}
\end{figure}

Now, as shown in Figure~\ref{fig:ordering}, we can cyclically order the coordinate rays in $O_0,\ldots,O_{t-1}$.  First, label the orthants so that for each $s\in\integers_t$, the $2$--orthant $O_s$ has the property that one of its coordinate rays $r_s^-$ coarsely coincides with a coordinate ray in $O_{s-1}$ and the other, $r_s^+$, coarsely coincides with a coordinate ray in $O_{s+1}$.  Now cyclically order the coarse equivalence classes of rays: $r_0^+,r_1^+,\ldots,r^+_{t-1}$.

We claim that $\mathfrak h_{f^\sharp(\sigma)}$,
$\mathfrak h_{f^\sharp(\sigma')}$ must be adjacent in this order.  This will imply that they are coarsely contained in a common $2$--orthant, and hence $f^\sharp(\sigma)\orth f^\sharp(\sigma')$, as required.

Indeed, if there was a coordinate ray $r$ between $\mathfrak
h_{f^\sharp(\sigma)}$ and $\mathfrak h_{f^\sharp(\sigma')}$, then $r$
is coarsely $\mathfrak h_{f^\sharp(\sigma'')}$, so that by definition
$f^{-1}(r)$ is coarsely $\mathfrak h_{\sigma''}$. (Here we used Assumption \ref{assumption:intersect}, which guarantees that 
$r$ is the ray associated to some hinge, and bijectivity of $f^\sharp$.) But then $\mathfrak
h_{\sigma},\mathfrak h_{\sigma'}, \mathfrak h_{\sigma''}$ pairwise
have infinite Hausdorff distance, are contained in the same standard
$2$--orthant, and they each arise as the coarse intersection with some
other orthant, contradicting Lemma \ref{lem:intersection_of_orthants}.
\end{proof}

\subsection{Sharpening of $f^\sharp$}

The hinge $f^\sharp(\sigma)$ prescribes a hierarchy ray which lies within finite
distance of $f(\mathfrak h_\sigma)$, but it does not (and cannot)
provide a uniform bound on the distance; which is what one typically
needs to show that two given quasi-isometries coarsely coincide.
Under many circumstances, finiteness can actually be promoted to a
uniform bound, with little extra work.  As an illustration of this, we
give an example tailored to the mapping class group case in the
following lemma. The content of the lemma is that if a quasi-isometry matches the ``combinatorial data'' of a standard flat 
to the data of another standard flat,  then it maps the former flat within uniform distance of the latter.

\begin{lem}[Flats go to flats]\label{lem:flat_to_flat}
 Let $(\cuco X,\frak S), (\cuco Y, \mathfrak T)$ be asymphoric HHS satisfying 
 Assumptions (\ref{assumption:line}), (\ref{assumption:intersect}) and 
   (\ref{assumption:extend}). There exists $C$ with the following property.  Let $\{U_i\}_{i=1}^n\subseteq \frak S$
 be a complete support set, and let $p_i^\pm$ be distinct points in 
 $\boundary\fontact U_i$.
  Suppose that there exists a complete
 support set $\{V_i\}_{i=1}^n\subseteq \frak T$ and distinct points $q_i^\pm\in\partial\fontact V_i$ so that for each
 $k=1,\dots,n$ we have $f^\sharp(U_{k},p^{\pm}_k)=(V_k,q_k^\pm)$.  Then, 
 $\dist_{haus}(f(\calF_{\{(U_i,p_i^\pm\}\}}),\calF_{\{(V_j,q_j^\pm)\}})\leq C$.
\end{lem}

\begin{proof}
Hierarchical quasiconvexity of $\calF_{\{(V_j,q_j^\pm)\}}$ implies it 
uniformly coarsely coincides with
$H_\theta(\calF_{\{(V_j,q_j^\pm)\}})$.
Containment of $f(\calF_{\{(U_i,p_i^\pm)\}})$ in a uniform neighborhood of
$\calF_{\{(V_j,q_j^\pm)\}}$ then follows from 
Lemma~\ref{lem:flats_close_to_hulls}.
The other containment follows by 
applying the same argument to a quasi-inverse of $f$.
\end{proof}

\subsection{Mapping class groups}\label{subsec:mcgqi}

We now use Theorem \ref{thm:clique_map} to provide a new proof of
quasi-isometric rigidity of mapping class groups. 
Like all proofs of quasi-isometric rigidity for mapping class groups, 
the goal of our proof is to prove that any quasi-isometry of the 
mapping class group induces a simplicial automorphism of
$\fontact S$, at which point we can apply Ivanov's theorem 
\cite{Ivanov} to conclude that the automorphism 
is induced by an element of the mapping class group. 
Using our quasiflats theorem 
we can readily convert the geometric information 
of a quasi-isometry to combinatorial information about the structure 
of standard flats. Then, via 
Theorem~\ref{thm:clique_map}, from the combinatorial structure of 
quasiflats we can extract an induced map on certain coordinate directions in the 
standard flats. In the mapping class group setting, these directions 
correspond to Dehn twist directions, thus giving us the automorphism 
of the curve graph which is needed to apply Ivanov's theorem.

\begin{thm}\cite{BKMM}
 Let $\cuco X$ be the the mapping class group of a non-sporadic
 surface $S$.  Then for any $K$ there exists $L$ so that: for each $(K,K)$--quasi-isometry $f\co\cuco X\to\cuco X$ there 
exists a
 mapping class $g$ so that $f$ $L$--coarsely coincides with
 left-multiplication by $g$.
\end{thm}

\begin{proof}
 Consider the standard HHS structure on $\cuco X$, so that $\mathfrak
 S$ is the collection of all essential subsurfaces, and the $\fontact
 Y$ are curve complexes. (For details on the structure, see 
 \cite[Section~11]{hhs2}.)

A subsurface $Y$ lies in a complete support set if and only if it is
an annulus, a once-punctured torus or a 4-holed sphere.  The
assumptions of Theorem~\ref{thm:clique_map} are clearly satisfied.

Consider any quasi-isometry $f\co\cuco X\to\cuco X$. A 
hinge $(U,p)$ is \emph{annular} if $U$ is an annulus. We now show 
that if a 
hinge $\sigma$ is annular, then so is $f^\sharp(\sigma)$.  Indeed, a hinge $\sigma$ being annular is characterized by the following property: $\sigma$ is contained in a maximal collection $\mathcal H$ of pairwise orthogonal hinges, and there exists a unique hinge $\sigma'$ so that $(\mathcal H-\{\sigma\})\cup\{\sigma'\}$ is a maximal pairwise orthogonal set of hinges.  This property is illustrated in Figure~\ref{fig:nonreplaceable}, where, if $\sigma$ is $(U,p^+)$, then $\sigma'$ is $(U,p^-)$, where $\boundary\fontact U=\{p^\pm\}$.

Since the bijection $f^\sharp$ preserves orthogonality and non-orthogonality, it preserves the above property, so $f^\sharp$ preserves being annular.  
\begin{figure}[h]
    \begin{overpic}[width=0.6\textwidth]{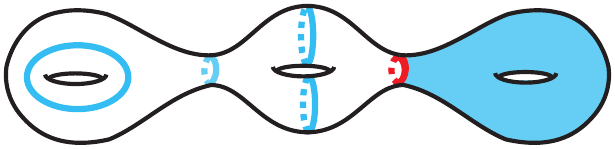}
    \put(64,17){$U$}
    \end{overpic}
 \caption{This figure shows 
 a complete support set, consisting of five annuli and one 
 once-punctured torus. This is the only complete support set 
 containing all the subsurfaces except the annulus about the boundary 
 of the once-punctured torus, denoted $U$ in the figure. In this sense, $U$ is non-replaceable.}\label{fig:nonreplaceable}
\end{figure}

Note that for any annulus $U$, the set $\boundary \fontact U$ has
exactly two points.  We now claim that for each annulus $U$ there
exists an annulus $V$ so that, denoting $\{p^\pm\}=\boundary\fontact
U$, we have $f^\sharp(U,p^\pm)=(V,q^\pm)$ for $q^\pm\in\boundary
\fontact V$.  This holds as above, since for some maximal set
$\mathcal H$ of pairwise orthogonal hinges containing $(U,p^+)$, the
hinge $(U,p^-)$ is the only hinge such that $\mathcal
H-\{(U,p^+)\}\cup\{(U,p^-)\}$ is a maximal set of pairwise orthogonal
hinges.  In this sense, the annulus $U$ is ``non-replaceable''.  We
write $V=f^*(U)$.  Notice that Lemma \ref{lem:flat_to_flat} now
applies to show that any Dehn twist flat of $\cuco X$ is mapped within
uniformly bounded distance of a Dehn twist flat.

Moreover, we have a well defined simplicial automorphism $\phi$ of the
curve graph $\fontact S$, where $\phi(\alpha)=\beta$ if $B=f^*(A)$,
where the annuli $A,B$ have core curves $\alpha,\beta$ respectively.
By a theorem of Ivanov \cite{Ivanov}, any simplicial automorphism of
$\fontact S$ is induced by an element of the mapping class group; we
denote by $g$ the element corresponding to $\phi$.

Suppose we are given a Dehn twist flat $\calF$ with complete support 
set $\mathcal U$. Then, as noted above, $f(\calF)$ is coarsely a Dehn twist  
flat with complete support set 
$\{f^*(U)\}_{U\in\mathcal U}=\{gU\}_{U\in\mathcal U}$.

We can now conclude that for any Dehn twist flat $\calF$,  we have 
that $f(\calF)$ and $g\calF$ are  
within bounded Hausdorff distance.
For any point $x\in\cuco X$, we can find Dehn twist flats $\calF^x_1,\calF^x_2$ that have neighborhoods of uniformly bounded radius whose intersection contains $x$ and has uniformly bounded diameter. Since $g\calF^x_i,f(\calF^x_i)$ coarsely coincide for $i=1,2$, we see that $gx$ and $f(x)$ must coarsely coincide. Hence we get that the automorphism of $\cuco X$ given by 
left-multiplication by $g$  is 
uniformly close to the quasi-isometry $f$, as desired.
\end{proof}

\section{Factored spaces}\label{sec:factored}
In this section we show that, under certain circumstances,
quasi-isometries between HHSs descend to quasi-isometries between some
of their ``factored'' spaces, which are spaces obtained by coning off a
collection of standard product regions. These factored spaces 
are HHSs themselves and their
complexity is lower than the complexity of the original HHS. Hence, 
studying 
induced quasi-isometries on factored spaces can be part of an
inductive procedure for studying quasi-isometries of the original 
space (see also the
introduction).

\begin{notation}\label{not:factored}
 Given $\mathfrak U\subseteq \mathfrak S$, let $\mathfrak U^\nest$ be
the collection of all $V\in\mathfrak S$ so that there exists
$U\in\mathfrak U$ with $V\nest U$.
We let $\mathfrak U=\mathfrak U_{\cuco X}\subset\mathfrak S$ denote the union of all cardinality--$\nu$
pairwise-orthogonal subsets of $\mathfrak S$.  Let $\widehat{\cuco X}$
be the factored space associated to $\mathfrak U^\nest$, which is the
space obtained from $\cuco X$ by coning off all $F_U$ for
$U\in\mathfrak U^\nest$ (as described in~\cite[Definition 2.1]{hhs3}). There exists a Lipschitz factor
map $q=q_{\cuco X}\co \cuco X\to\widehat{\cuco X}$ by 
\cite[Proposition 2.2]{hhs3}.
\end{notation}

By \cite[Proposition 2.4]{hhs3}, $\widehat{\cuco X}$ has a natural HHS structure with
index set $\mathfrak S- \mathfrak U^\nest$.

\begin{thm}[Quasiflats collapse in factored spaces]\label{thm:cone_bound}
Let $\cuco X$ be an asymphoric HHS of rank $\nu$.  For any $K$, there
exists $\Delta$ so that for all $(K,K)$--quasi-isometric embeddings
$f\co\reals^\nu\to\cuco X$, we have $\diam(q\circ
f(\reals^\nu))\leq\Delta$.
\end{thm}

\begin{proof}
Observe that if $A\subset\cuco X$ is an arbitrary subset, then $q(H_\theta(A))$ lies at
uniformly bounded Hausdorff distance from $H_\theta(q(A))$ (where we
take hulls in $\widehat{\cuco X}$ in the second expression).  In
particular, if $\diam_{\widehat{\cuco X}}(q(A))\leq C$ for some $C$,
then there exists $C'=C'(C,E,\theta)$ so that 
for any $B\subset H_\theta(A)$ we have $\diam_{\widehat{\cuco
X}}(q(B))\leq C'$.

Hence, by Corollary~\ref{cor:flats_close_to_hulls}, it suffices to
prove that $\diam_{\widehat{\cuco X}}q(\bigcup_{i=1}^NO_i)\leq C$,
where the orthants $O_i$ are as in the Corollary and
$C=C(N,E,K,\mu_0)$.  By the construction of $q$, it follows 
easily that there exists
$C'=C'(\mu_0,E)$ such that $\diam_{\widehat{\cuco X}}(q(O_i))\leq C'$
for each $i$.  By Proposition~\ref{prop:orthant_chain}, it suffices to
bound the diameter of $q(O_i\cup O_j)$ in the case where
$O_i\tilde\cap O_j$ is a codimension-$1$ sub-orthant; this is done
in Lemma~\ref{lem:ass_bump}.
\end{proof}

Before proceeding to the technical Lemmas and Propositions we needed 
to prove the above theorem, we state 
the following corollary which we consider the main result of this 
section.

\begin{cor}[Quasi-isometries descend to factored spaces]\label{cor:induced_qi}
 Let $\cuco X,\cuco Y$ be asymphoric HHSs. Suppose that there exists $D$ so that for each $U\in\mathfrak U_{\cuco X}$ or 
$U\in\mathfrak U_{\cuco Y}$, for any $x,y\in F_U$ there exists a bi-infinite $(D,D)$--quasi-geodesic containing $x,y$. Then 
for every quasi-isometry $f:\cuco X\to \cuco Y$ there exists a quasi-isometry $\hat f:\widehat{\cuco X}\to\widehat{\cuco Y}$ 
so that the diagram
 
 \begin{center}
  $
  \begin{diagram}
   \node{\cuco X}\arrow{e,t}{f}\arrow{s,l}{q_{\cuco X}}\node{\cuco Y}\arrow{s,r}{q_{\cuco Y}}\\
   \node{\widehat{\cuco X}}\arrow{e,b}{\hat f}\node{\widehat{\cuco Y}}
  \end{diagram}
  $
 \end{center}
commutes. 
\end{cor}

\begin{proof}
 Since $\widehat{\cuco X}, \widehat{\cuco Y}$ are just re-metrized copies of $\cuco X$, $\cuco Y$ (see~\cite{hhs3}), we can 
take $\hat f=f$.
 
 We now show that $\hat f$ is coarsely Lipschitz, and observe that the corresponding map for a quasi-inverse of $f$ gives a coarsely Lipschitz inverse of $\hat f$.

 By the definition of the metric on $\widehat{\cuco X}, \widehat{\cuco Y}$ (\cite[Definition 2.1]{hhs3}), we just have to verify that if $x,y$ lie in some $F_U$ for $U\in\mathfrak U^\nest_{\cuco X}$, then their images are uniformly close in $\widehat{\cuco Y}$. By assumption, $x,y$ lie close to a quasiflat with uniform constant, so that the conclusion follows from Theorem \ref{thm:cone_bound}. 
\end{proof}

We can also now prove Theorem~\ref{thmi:MCG} from the introduction.

\begin{proof}[Proof of Theorem~\ref{thmi:MCG}]
In the case where $\cuco X$ is a mapping class group, we have seen that the standard HHS structure $(\cuco X,\mathfrak S)$ 
is asymphoric of finite rank $\nu$ equal to the complexity.  Let $\mathfrak U$ be as in Notation~\ref{not:factored} and let 
$q:\cuco X\to\widehat{\cuco X}$ be the factor map described above.  Then any $\nu$--dimensional quasiflat in $\cuco X$ has 
uniformly bounded image in $\widehat{\cuco X}$ by Theorem~\ref{thm:cone_bound}.  Now, letting $S\in\mathfrak S$ be the 
unique $\nest$--maximal element, we have that $\pi_U:\widehat{\cuco X}\to\fontact S$ is a Lipschitz map, and the theorem 
follows.
\end{proof}

Now we turn to the lemmas. 

The following lemma identifies possible distance formula terms for pairs of points each in a given orthant. Roughly, they can be of two types, each corresponding to one of the factors of the bridge as in Lemma~\ref{lem:parallel_gates} between the orthants.

\begin{lemma}\label{lem:shared}
 There exists $\tau$ with the following property.  Let $O,O'$ be
 standard orthants in $\cuco X$ with supports $\mathcal U_1,\mathcal
 U_2$.  Suppose that $O\tilde{\cap}O'$ is a $k$--orthant whose support
 is $\mathcal U$.  Then for each $x,y\in O\cup O'$ we have that any
 $U\in \mathfrak S$ with $\dist_U(x,y)\geq \tau$ is either nested into some $U'\in\mathcal
 U_1\cap \mathcal U_2$ 
 or orthogonal to all $U'\in\mathcal U$.
\end{lemma}

\begin{proof}
Recall that $O\tilde{\cap}O'$ coarsely coincides with $\gate_O(O')$ by
Lemma \ref{lem:intersection_of_orthants} (and also with a 
standard orthant whose support is contained in 
$\mathcal U_1\cap \mathcal U_2$, thereby describing $\mathcal U$).

Let $x,y,U$ be as in the statement.  If $x,y\in O$, then by the definition of a standard orthant, either 
$\dist_U(x,y)$ is uniformly bounded or $U\nest U'$ for some $U'\in\mathcal U_1$.  If $U'\in\mathcal 
U_1\cap\mathcal U_2$, we are done; otherwise, $U'\orth V$ for all $V\in\mathcal U$ by the 
definition of a standard orthant.  An identical argument works if $x,y\in O'$.

So, assume that $x\in O',y\in O$. If $U$ is nested in some element of $\mathcal U_i$, for some $i\in\{1,2\}$, then either $U$ is nested into some element of $\mathcal U_1\cap\mathcal U_2$, or $U$ is orthogonal to every element of $\mathcal U$ by the definition of a 
standard orthant. Hence, suppose that $U$ is not nested into any element of 
$\mathcal U_1$ or $\mathcal U_2$. In particular, $\pi_U(O)$ and $\pi_U(O')$ have uniformly bounded diameter, so $\dist_U(x,y)\asymp\dist_U(O,O')$. 
 Therefore, provided $\tau$ is sufficiently large, $\dist_U(x,y)>\tau$ implies that $\dist_U(O,O')$ is large compared to the constant $K_2$ from 
Lemma~\ref{lem:parallel_gates}, so that conclusion \eqref{item:df_bridge} of the same lemma (applied with any $p\in\gate_O(O')$, $t_1=a$, and $t_2=b$) shows that $U$ is orthogonal to each element of 
$\mathcal U$.  
\end{proof}

If the coarse intersection $O\tilde\cap O'$ is a codimension-$1$ sub-orthant, then $q(O\cup O')$ is \emph{uniformly} bounded:

\begin{lem}\label{lem:ass_bump}
There exists $C=C(E,\mu_0)$ so that the following holds.  Let $O,O'$ be standard orthants with $O\tilde\cap O'$ a codimension-$1$ sub-orthant.  Then $\diam_{\widehat{\cuco X}}(q(O\cup O'))\leq C$.
\end{lem}

\begin{proof}
Let $x\in O,y\in O'$.  Let $\mathcal M=\{U\in\mathfrak S : \dist_U(x,y)\geq\tau\}$. 
By
Lemma~\ref{lem:shared}, 
each $U\in\mathcal M$ belongs to a 
set of pairwise-orthogonal elements of size $\nu$ (note that in the 
case that $U$ is orthogonal to the intersection, this has maximal 
rank because of the fact that we are assuming the intersection has 
co-dimension-1).
Hence
$\dist_U(q(x),q(y))\leq\tau$ for all $U\in\mathfrak S-\mathfrak U$, so
$q(x)$ is uniformly close to $q(y)$ by the uniqueness axiom.
\end{proof}

\begin{prop}\label{prop:orthant_chain}
 Suppose that the quasiflat $\calF$ lies within finite Hausdorff distance of $\bigcup_{i=1}^m O_i$, where the $O_i$ are standard orthants with $\dist_{haus}(O_i,O_j)=\infty$ for $i\neq j$. Then for each pair of distinct orthants $O_j,O_k$ there exists a sequence $j=j_0,\dots,j_l=k$ so that the coarse intersection of $O_{j_i}$ and $O_{j_{i+1}}$ is an $(\nu-1)$--orthant.
\end{prop}

\begin{proof}
 Passing to an asymptotic cone, we get a bilipschitz copy $\seq \calF$ of $\reals^\nu$ filled by bilipschitz copies $\seq O_i$ of $[0,\infty)^\nu$. The intersections of the $\seq O_i$ have some basic properties:
 
 \begin{lemma}
 \begin{enumerate}
 \item[]
  \item The intersection of $\seq O_i$ and $\seq O_j$ is bilipschitz equivalent to $[0,\infty)^t$ for some $t=t(i,j)$.
  \item $t(i,j)=\nu-1$ if and only if $\seq O_i$ and $\seq O_j$ coarsely intersect in an $(\nu-1)$--orthant.
 \end{enumerate}
 \end{lemma}
 
 \begin{proof}
 Recall that the coarse intersection of two standard orthants coarsely
 coincides with a standard $k$--orthant, as well as with the gate of
 one in the other (Lemma \ref{lem:intersection_of_orthants}).  We now
 show the following, which implies both statements: if the ultralimits $\seq A$, $\seq B$ of uniformly
 hierarchically quasiconvex sets have non-empty intersection, then
 their intersection is the ultralimit $\gate_{\subseq A}(\seq B)$ of
 the gates. By Lemma
 \ref{lem:parallel_gates}.\eqref{item:bridge}, $\gate_{\subseq A}(\seq
 B)$ is contained in $\seq A\cap \seq B$ (this uses $\dist(\seq A,\seq
 B)=0$).  Lemma \ref{lem:parallel_gates}.\eqref{item:take_the_bridge}
 implies that the other containment holds.
 \end{proof}

Now, consider the subspace $X\subset \seq \calF$ consisting of the union of all $\seq O_i\cap\seq O_j$ for $i,j$ with $t(i,j)=\nu-1$.  Let $\mathcal Y$ be the set of all $\seq O_i\cap\seq O_j$ with $i\neq j$ and $t(i,j)<\nu-1$.  Let $Y=\bigcup_{O\in\mathcal Y}O$.

 \begin{lem}\label{lem:mayer_vietoris}
$\seq \calF-Y$ is path-connected.
\end{lem}

\begin{proof}
In this proof, when referring to homology, we always mean singular homology with rational coefficients. The goal is to show $H_0(\seq \calF-Y)=\rationals$.

If $\dimension\seq \calF\leq 2$, then $\mathcal Y$ is a finite set (which is empty when $\dimension\seq \calF\leq 1$) and the claim is clear.  Hence suppose that $\dimension\seq \calF\geq 3$.  We argue by induction on $|\mathcal Y|$.  

We first claim that for any $O\in\mathcal Y$ and any closed $O'\subset O$, $\seq \calF-O'$ is path-connected and $H_1(\seq 
\calF-O')=0$.  We use the fact that, for $A,B$ closed homeomorphic subsets of $\reals^\nu$, we have 
$H_*(\reals^\nu-A)=H_*(\reals^\nu-B)$, see e.g.~\cite{Dold}. Hence, we can regard $O$ as a coordinate orthant in 
$\reals^\nu\cong \calF$.  Hence the claim holds for $O'=O$.  The fact that $H_1(\seq \calF-O')=0$ follows from the fact that 
$H_1(\seq \calF-O)=0$, since a $1$--cycle in $\seq \calF-O'$ is homologous to one in $\seq \calF-O$ by, for example, a 
transversality argument.  The same holds for $H_0(\seq \calF-O')$.  

For the inductive step, let $A$ be the union of all but one element of $\mathcal Y$, and let $B$ be the remaining one.  We have a Mayer-Vietoris sequence:
$$H_1(\seq \calF-(A\cap B))\to H_0(\seq \calF-(A\cup B))\to H_0(\seq \calF-A)\oplus H_0(\seq \calF-B)\to H_0(\seq \calF-(A\cap B))\to 0.$$
By the claim above, the first term is $0$, the last term is $\rationals$, and $H_0(\seq \calF-B)=\rationals$.  By induction, $ H_0(\seq \calF-A)=\rationals$.  Hence $\seq \calF-(A\cup B)$ is connected.
\end{proof}

We now finish the proof of Proposition~\ref{prop:orthant_chain}.

Let $\seq O_j,\seq O_k$ be orthants.  We will now produce a sequence  
$\seq O_j=\seq O_{j_0},\ldots,\seq O_{j_l}=\seq O_k$ of orthants so
that $t(j_i,j_{i+1})=\nu-1$ for $0\leq i\leq l-1$.  Choose $\seq
x\in\interior{\seq O_i},\seq y\in\interior{\seq O_j}$ and let
$\sigma\co [0,1]\to \seq \calF-Y$ be a path joining them, which is
provided by Lemma~\ref{lem:mayer_vietoris}.  Let $t_0$ be the maximal
$t$ so that $\sigma(t)\in\seq O_j$.  If $t_0=1$, then we take $l=0$.
Otherwise, there exists $\seq O_{j_1}\neq\seq O_j$ so that $\seq
O_j\cap\seq O_{j_1}$ has dimension $\nu-1$ and contains $\sigma(t_0)$.
Now apply the same argument to $\sigma|_{[t_0,1]}$ and induct.

The sequence in the cone yields a sequence of orthants in the space 
with the desired property.
\end{proof}

\bibliographystyle{alpha}
\bibliography{quasiflats}
\end{document}